\documentclass[a4paper,12pt,oneside,reqno]{amsart}

\usepackage{hyperref}
\usepackage[headinclude,DIV13]{typearea}
\areaset{15.1cm}{25.0cm}
\usepackage{amsfonts,amssymb,amsmath,amsthm,bbm}
\usepackage[latin1] {inputenc}
\usepackage{graphicx, psfrag}
\usepackage{color}
\usepackage{bold-extra}

\newtheorem{theorem}{Theorem}[section]
\newtheorem{lemma}[theorem]{Lemma}
\newtheorem{proposition}[theorem]{Proposition}
\newtheorem{corollary}[theorem]{Corollary}

\theoremstyle{definition}
\newtheorem{definition}[theorem]{Definition}

\theoremstyle{remark}
\newtheorem*{remark}{Remark}

\def\paragraph#1{\noindent \textbf{#1}}

\numberwithin{equation}{section}

\def\rank{\mathop{\rm rank}\nolimits}
\def\d{\mathrm{d}}

\def\<{\langle}
\def\>{\rangle}
\def\a{\alpha}
\def\b{\beta}
\def\e{\epsilon}

\def\g{\gamma}

\def\l{\lambda}

\def\s{\sigma}

\def\o{\omega}

\def\L{\Lambda}
\def\G{\Gamma}
\def\O{\Omega}
\def\S{\Sigma}

\def\del{\partial}
\def\R{{\Bbb R}}  
\def\N{{\Bbb N}}  
\def\P{{\Bbb P}}  
\def\Z{{\Bbb Z}}  
  
\def\C{{\Bbb C}}  
\def\E{{\Bbb E}}

\let\cal=\mathcal

\def\BB{{\cal B}}
\def\CC{{\cal C}}
\def\DD{{\cal D}}
\def\EE{{\cal E}}
\def\FF{{\cal F}}
\def\GG{{\cal G}}
\def\HH{{\cal H}}
\def\II{{\cal I}}
\def\JJ{{\cal J}}

\def\NN{{\cal N}}
\def\OO{{\cal O}}

\def\RR{{\cal R}}
\def\SS{{\cal S}}

\def\VV{{\cal V}}

\def\VV{{\cal V}}
\def\WW{{\cal W}}

 \def \G {{\Gamma}}
 \def \L {{\Lambda}}
 
 \def \b {{\beta}}
\def \e {{\epsilon}}
 \def \s {{\sigma}}

 \def \g {{\gamma}}
 \def \l {{\lambda}}
 \def \d {{\delta}}
 \def \a {{\alpha}}
 \def \o {{\omega}}
 \def \O {{\Omega}}

 \def \del {{\partial}}

 \def \ba {\begin{array}}
 \def \ea {\end{array}}

 \newcommand{\be}{\begin{equation}}
 \newcommand{\ee}{\end{equation}}

\newcommand{\bea}{\begin{eqnarray}}
 \newcommand{\eea}{\end{eqnarray}}
\def\TH(#1){\label{#1}}\def\thv(#1){\ref{#1}}
\def\Eq(#1){\label{#1}}\def\eqv(#1){(\ref{#1})}

\def\sfrac#1#2{{\textstyle{#1\over #2}}}

 \def \1{\mathbbm{1}}
\def\wt {\widetilde}

\def\Tanh{{\mathrm{Th}}}
\def\erfc{{\rm{erfc}}}

\def\diag{\textrm{diag}}

\begin{document}

\title[Emergence of near-TAP free energy functional]{Emergence of near-TAP free energy functional in the SK model at high temperature}

\author[V. Gayrard]{V\'eronique Gayrard}
\address{
V. Gayrard\\ Aix Marseille Univ, CNRS, I2M, Marseille, France
}
\email{veronique.gayrard@math.cnrs.fr}


\date{\today}

\begin{abstract} 
We study the SK model at inverse temperature $\b>0$ and strictly positive field $h>0$ in the region of  $(\b,h)$ where the replica-symmetric formula is valid. An integral representation of the partition function derived from the Hubbard-Stratonovitch transformation combined with a duality formula is used to prove that  the infinite volume free energy of the SK model can be expressed as a variational formula on the space of magnetisations, $m$. The resulting free energy functional differs from that of Thouless, Anderson and Palmer (TAP) by the term
$
-\frac{\b^2}{4}\left(q-q_{\text{EA}}(m)\right)^2
$ 
where $q_{\text{EA}}(m)$ is the Edwards-Anderson parameter and $q$ is the minimiser of the replica-symmetric formula. Thus, both functionals have the same critical points and take the same value on the subspace of magnetisations satisfying $q_{\text{EA}}(m)=q$. This result is based on an in-depth study of the global maximum of this near-TAP free energy functional using Bolthausen's solutions of the TAP equations, Bandeira \& van Handel's  bounds on the spectral norm of non-homogeneous Wigner-type random matrices, and Gaussian comparison techniques. It holds for $(\b,h)$ in a large subregion of the de Almeida and Thouless high-temperature stability region.
\end{abstract}

\thanks{V.~Gayrard would like to thank the Institute for Applied Mathematics of the University of Bonn for its kind hospitality during the writing of this work. Funding for her stay was provided by the Gay Lussac-Humboldt Research Award of the Alexander von Humboldt Foundation and the Bonn Research Chair program of the Hausdorff Center for Mathematics.}

\maketitle


\section{Introduction}
    \TH(S1)

\subsection{Background, motivation and main result} 
    \TH(S1.1)     

The Hamiltonian of the Sherrington and Kirkpatrick model (hereafter SK model) at inverse temperature $\b>0$ and external field $h\geq 0$ is the random function defined on $\S_N\equiv\{-1,1\}^N$ by
\be
H_{N,\b,h}(\s)=-\b \frac{1}{2\sqrt N}\sum_{1\leq i, j\leq N}J_{ij}\s_i\s_j -h\sum_{1\leq i\leq N}\s_i
\Eq(1.1)
\ee
where $\s=(\s_i)_{1\leq i\leq n}\in\S_N$ and, given a  collection $(g_{ij})_{1\leq i,j\leq N}$ of i.i.d.~Gaussian random variables with variance $1/N$, $J_N=(J_{ij})_{1\leq i,j\leq N}$ is the symmetric matrix with entries
\be
\frac{J_{ij}}{\sqrt N}=\frac{1}{\sqrt{2}}\left(g_{ij}+g_{ji}\right), \quad1\leq i,j\leq N.
\Eq(1.01)
\ee
Although the explicit representation \eqv(1.01) comes into play only in Section \thv(S3) and \thv(S4), we introduce it now in order to avoid the confusion of having to change the underlying probability space during the proof. 
We call  this probability space $(\O,\FF,\P)$.

Denoting by $Z_{N,\b,h}$ the partition function associated to \eqv(1.1)
\be
Z_{N,\b,h}=\sum_{\s\in\S_N}e^{-H_{N,\b,h}(\s)}
\Eq(1.2)
\ee
the free energy, $F_{N,\b,h}$, is defined as\footnote{The above definitions and terminology are not the standards of physics (in particular, $h$ should be replaced by $\b h$ and \eqv(1.3) does not define the free energy but the pressure, the free energy being the quantity $-\b^{-1}F_{N,\b,h}$). However, they have become commonly used in mathematics, especially in the publications to which we will refer extensively. Therefore, for the sake of clarity, we stick to them. We also note that while it is customary to set the diagonal couplings $J_{i,i}$ to zero, the contribution of these terms to $F_{N,\b,h}$ vanishes as $N\rightarrow\infty$.}
\be
F_{N,\b,h}=\frac{1}{N}\log Z_{N,\b,h}.
\Eq(1.3)
\ee
It is known that its $N\rightarrow\infty$ limit (called the infinite volume limit) exists  and is ``self-averaging'' \cite{GT02} 
\be
f(\b,h)\equiv \lim_{N\rightarrow\infty}F_{N,\b,h}=\lim_{N\rightarrow\infty}\E F_{N,\b,h} \quad \P\textstyle{-a.s.}
\Eq(1.4)
\ee
It is also known that this limit is given by \emph{Parisi variational formula} \cite{P79}, \cite{Ta12}, \cite{PanBook}. In this paper we focus on the simplest situation where the model is at high temperature, a regime which is defined as follows. Consider the equation
\be
q=E \tanh^2(\b\sqrt{q}Z+h)
\Eq(1.5)
\ee
where $Z$ is a standard gaussian random variable and $E$  denotes the expectation with respect to $Z$. 
It is well known \cite{Ta12} that \eqv(1.5) has a unique solution $q\equiv q(\b,h)>0$  for all $\b>0$  if  $h\neq 0$ and that it has 
for unique solution $q(\b,0)\equiv 0$ for all $\b\leq 1$ if  $h=0$. 
Further define the so-called \emph{replica symmetric formula} as the function
\be
SK(\b,h)=\log 2 + \frac{\b^2}{4}(1-q)^2+E\log\cosh(\b\sqrt{q}Z+h).
\Eq(1.6)
\ee
Bearing in mind that \eqv(1.4) is known, we adopt Talagrand's definition and say that
\begin{definition}
    \TH(1.def1)
The \emph{high-temperature region of the SK model} is the region of $(\b,h)$ where
\be
f(\b,h)=SK(\b,h).
\Eq(1.def1.1)
\ee
\end{definition}
The identity \eqv(1.def1.1) was originally established by Sherrington and Kirkpatrick by means of the replica method \cite{SK}. However, it was soon realised that  their results were flawed since they yield a negative entropy at low enough temperature. Revisiting the saddle point analysis that enters into the derivation of \eqv(1.6) in \cite{SK},  de Almeida and Thouless \cite{AT78} obtained that \eqv(1.def1.1) should hold if
\be
\b^2E\frac{1}{\cosh^4(\b\sqrt{q}Z+h)}\leq 1,
\Eq(1.7)
\ee
a condition hereafter referred to as the \emph{AT-condition}.
The region of $(\b,h)$ where the AT-condition is satisfied is called the \emph{AT-region}, and replacing the inequality by an equality in \eqv(1.7) defines the \emph{AT-line}. Mathematically, the validity of this condition was proved  for $h=0$ and $\b<1$  \cite{ALR}.
For $h>0$, only partial results are know: it was proved in  \cite{To} that the high-temperature region is entirely located inside
the AT-region and large subregions of the high-temperature region have been identified in   \cite{Ta12} (see Vol.~II Chap.~13, Theorem 13.6.2 for the explicit but somewhat inextricable description of this region) and \cite{JaTo}.

The  breakdown of the replica symmetric solution at low temperature motivated the search for alternatives to the replica method.
In pursuing this aim, Thouless, Anderson and Palmer (hereafter TAP) developed an extended mean field approach  \cite{TAP}.
Relying on a Bethe approximation, the fundamental self-consistency equation underlying mean field theories 
-- the so-called \emph{mean field equation} -- was derived: it consists of a system of $N$ equations in $N$ unknown 
thought of as local magnetisations, $m\equiv (m_i)_{1\leq i\leq N}\in [-1,1]^N$, given by
\be
m_i=\tanh\left(
h+\b \sum_{j=1}^{N}\frac{J_{ij}}{\sqrt N}m_j-\b^2(1-q_{\text{EA}}(m))m_i
\right), 
\quad 1\leq i\leq N.
\Eq(1.11)
\ee
Unlike classical mean field equations, a retroactive Onsager term $\b^2(1-q_{\text{EA}}(m))m_i$ is subtracted to account for the response of site $j$ to the local magnetisation at site $i$, where
\be
q_{\text{EA}}(m)=\frac{1}{N}\sum_{i=1}^Nm_i^2
\Eq(1.9)
\ee
is the Edwards-Anderson parameter. The associated free energy functional, $F_{N,\b,h}^{TAP}$, is then introduced in \cite{TAP} as a ``\emph{fait accompli}'' (derived from unpublished diagram expansion), and defined on $\R^N$ by
\be
\begin{split}
 &F_{N,\b,h}^{TAP}(m)  \equiv 
\\
& 
\frac{1}{N}\left\{
 \frac{\b}{2}\sum_{1\leq i,j\leq N}\frac{J_{ij}}{\sqrt N}m_im_j+ h\sum_{1\leq i\leq N}m_i+\frac{N\b^2}{4}\left(1-q_{\text{EA}}(m)\right)^2- \sum_{1\leq i\leq N}I(m_i)
 \right\} 
 \end{split}
\Eq(1.8)
\ee
where  $I$ is Cram\' er's entropy function, i.e., $I(x)=\infty$ for $|x|>1$ and
\be
I(x)=\frac{1+x}{2}\log\left(\frac{1+x}{2}\right)+\frac{1-x}{2}\log\left(\frac{1-x}{2}\right)\quad \text{for $|x|\leq 1$}.
\Eq(1.10)
\ee
Observing that the critical points of $F_{N,\b,h}^{TAP}(m)$ coincide, as it must, with the solutions of the \emph{TAP equations} \eqv(1.11),  one finally expects the  free energy in the infinite volume limit to be obtained by maximising $F_{N,\b,h}^{TAP}(m)$ over the space of local magnetisations, conditional on certain restrictions, and taking the limit $N\rightarrow\infty$. Indeed, since \eqv(1.11) and \eqv(1.8) both rely on approximations techniques, they must be accompanied by validity conditions and understood for large $N$, up to sub-leading corrections. 

The questions of justifying \eqv(1.8) and finding the conditions of its validity have received much attention in the physics literature.
The most influential contribution is undoubtedly that of Plefka \cite{Plef82} (see also the recent developments \cite{Plef02, Plef20}) who devised a method for deriving the free energy functional through the perturbative expansion of a certain function of the local magnetisations, called  Gibbs potential, which, when expended to second order, allows one to recover \eqv(1.8) under the main convergence condition 
\be
\b^2\frac{1}{N}\sum_{i=1}^N\left(1-m_i^2\right)^2<1.
\Eq(1.Plef)
\ee
This method has been examined  in a number of mathematical publications but has not been made rigorous to date
\cite{Nico18}, \cite{Adrien22}.

Another key issue raised by the extended mean field approach of \cite{TAP} it that of consistency, namely, that of  proving that the thermodynamic quantities calculated within this framework, i.e.~assuming \eqv(1.8),  coincide with  the predictions  based on the replica method \cite{SK}, \cite{P79} and rigorously established since \cite{Ta12}, \cite{PanBook}. Tightly related to the problem of finding the solutions of the TAP equations \eqv(1.11), this question has been studied extensively in theoretical physics in the early 2000s \cite{CGPM03}, \cite{CLPR04}, \cite{ABM04}, \cite{Plef02,Plef20} (see also references therein). From a mathematical viewpoint, a study of consistency for the free energy of mixed $p$-spin models was recently initiated  in \cite{CP} and further pursued in \cite{CPS1}, \cite{CPS2}. It is proved in \cite{CP} that  the infinite volume limit free energy can indeed be expressed as the supremum of the mixed $p$-spin version of the free energy functional \eqv(1.8)  constrained over magnetisations whose Edwards-Anderson parameter, $q_{\text{EA}}(m)$, is to the right of the support of the Parisi measure. This is done by linking the question of finding the maxima of the free energy functional to the mathematical theory of the Parisi solution, thus circumventing the hard problem of explicitly solving the TAP equations.

The latter problem was first addressed in the landmark paper \cite{EB14} and its follow-up \cite{EB19}. There, an iterative construction of solutions of  these equations for the SK model is  introduced and is shown to converge in the whole region of $(\beta,h)$ where the AT-condition \eqv(1.7) is satisfied. 

In this paper, we revisit the extended mean field theory of TAP from a field theory perspective via the Hubbard-Stratonovich transformation. Widely used in physics, this transformation linearises the Hamiltonian and, in doing so, introduces an auxiliary scalar field which, in the SK model, is complex. This feature makes the resulting integral representation of the partition function seem practically intractable, so that this otherwise very natural approach may have been deemed unrealistic. We show, focusing on the high-temperature region of the SK model in the sense of definition \thv(1.def1) and building on the properties of the iterative construction of solutions of the TAP equations obtained in \cite{EB14}, \cite{EB19}, that this integral representation allows us to express the free energy in the form of a variational formula. Unexpectedly, the free energy functional to be maximised,  $F_{N,\b,h}^{HT}$,  is not the TAP free energy functional \eqv(1.8),  but a smaller one, given by
\be
F_{N,\b,h}^{HT}(m)=F_{N,\b,h}^{TAP}(m)-\frac{\b^2}{4}\left(q-q_{\text{EA}}(m)\right)^2.
\Eq(7.31bis)
\ee
Note that when restricted to the subspace of magnetisations satisfying $q_{\text{EA}}(m)=q$, both functionals have the same critical points and take the same values at these points. Thus, as long as they have a common maximiser that lies in that subspace, the difference between the two will not affect the free energy.  It can, however, lead to different \emph{stability conditions}, i.e.~different conditions on $(\b,h)$ for a common critical point to be a global maximum. We return to this question in Section \thv(S1.5).

A first, concise formulation of the main result of this article is as follows.
\begin{theorem}
     \TH(main.theo0) 
There exists a region $\DD$ of $(\b,h)$, $h>0$, such that in the intersection of $\DD$ and the high-temperature region of Definition \thv(1.def1)
\be
f(\b,h)=\lim_{N\rightarrow\infty}\sup_{m\in [-1,1]^N} 
F_{N,\b,h}^{HT}\left(m\right)  
 \quad \P\textstyle{-a.s.}
\Eq(main.theo0.1)
\ee
\end{theorem}

The slightly cumbersome and unwieldy description of the explicit region $\DD$ for which we prove Theorem \thv(main.theo0) is deferred to Section \thv(S1.3). Nevertheless, it can be said at this stage that $\DD$ contains a large subregion of the AT-region  \cite{AT78}, in particular the high-temperature half-plane
\be
\b<1/2, \quad \forall h>0
\Eq(main.theo0.2)
\ee 
and the low-temperature and large-field region (recall that in  \cite{AT78} the field is the quantity $h/\b$)
\be
12\b e^{-\frac{1}{9}(h/\b)^2}<1.
\Eq(main.theo0.3)
\ee
For comparison,  in \cite{AT78} (see Figure 2 and Eq. (23)) the corresponding regions have the same shape but different constants, namely
\be
\b<1, \quad \forall h>0
\ee 
and, for large fields and low temperatures,
\be
\sfrac{4}{3 \sqrt{2\pi}}
\b e^{-\frac{1}{2}(h/\b)^2}<1.
\ee
Clearly, the region $\DD$  obtained is not optimal, but reflects the limitations of the techniques used in the analysis of $F_{N,\b,h}^{HT}$ (mainly random matrix techniques and Gaussian comparison techniques). Moreover, the treatment of the high-temperature subregion is kept brief and elementary, whereas the most difficult and interesting subregion, that of large fields and low temperatures, is the main focus of this paper.

We stress that the proof of Theorem  \thv(main.theo0) is not entirely \emph{ab initio} since it makes crucial use of the fact, 
presupposed in Definition \thv(1.def1), that the free energy converges almost surely to $SK(\b,h)$ in the high-temperature region. 
However, in contrast to previous work, it is not presumed that the TAP free energy formula is known.

Before presenting the strategy of the proof, we should emphasise that the study of the TAP approach remains topical in both theoretical physics and probability theory. Several questions, notably that of the consistency of the two approaches, \emph{Replicas} versus \emph{TAP}, have been actively debated in theoretical physics until recently   \cite{CGPM03},  \cite{ABM04}, \cite{Plef02}, \cite{Plef20}. From a mathematical perspective, this equivalence has recently been investigated in \cite{CP},  \cite{CPS1}, \cite{CPS2}.  The question of finding the solutions of the TAP equations \cite{ChTa},  \cite{EB14}, \cite{EB19},  reproving their derivation by Stein's method \cite{Chat} or by a dynamical approach \cite{ABSY21}, studying their stability and their numerical solutions \cite{GIK21}, questioning the meaning of Plefka's condition \cite{Adrien22} have also been tackled recently,  and an upper bound on the TAP free energy was obtained in \cite{Belius}, to mention only recent publications without claim to completeness.

\subsection{Structure of the proof of Theorem \thv(main.theo0) } 
    \TH(S1.2)
    
The major part of the proof of Theorem \thv(main.theo0) is concerned with the study of the properties of the function 
$F_{N,\b,h}^{HT}$ which are needed, in a final section, to establish the variational formula \eqv(main.theo0.1) via the Hubbard-Stratonovitch transformation.

The connection between this variational formula and the replica symmetric formula  \eqv(1.6) is established through a duality formula. Such a formula transforms an initial optimisation problem into another, so that the initial function to be optimised and its dual have the same critical points and take the same value at those points. The case of an initial function, such as  $F_{N,\b,h}^{HT}$, which decomposes into the sum of a quadratic form and a convex function, has been studied extensively (albeit in other contexts \cite{Eke}).  We draw on this in Section \thv(S2).

Knowing the critical points of $F_{N,\b,h}^{HT}$ is key to using duality. This is where  Bolthausen's iterative scheme  \cite{EB14}, \cite{EB19} comes in. For $h>0$, let  $q\equiv q(\b,h)$ be the unique solution of  \eqv(1.5) and consider the system of TAP equations  \eqv(1.11) with $q$ substituted for $q_{\text{EA}}$
\be
m_i=\tanh\left(
h+\b \sum_{j=1}^{N}\frac{J_{ij}}{\sqrt N}m_j-\b^2(1-q)m_i
\right), 
\quad 1\leq i\leq N.
\Eq(1.12)
\ee
Note that  \eqv(1.12) is nothing else but the critical point equation of the function  $F_{N,\b,h}^{HT}$. Let $m^{(k)}\equiv \bigl(m^{(k)}_i\bigr)_{1\leq i\leq N}$, $k\in\N$, be the sequence of random variables defined recursively by
\be
m^{(0)}={\bf{0}}, m^{(1)}=\sqrt q {\bf{1}}
\Eq(1.13)
\ee
where $\bf{0}$ and $\bf{1}$ are the vectors whose coordinates are all $0$ and all $1$, respectively and, for all $k\geq 1$, set
\be
m^{(k+1)}_i=\tanh\left(
h+\b \sum_{j=1}^{N}\frac{J_{ij}}{\sqrt N}m^{(k)}_j-\b^2(1-q)m^{(k-1)}_i
\right), 
\quad 1\leq i\leq N.
\Eq(1.14)
\ee
To simplify notation we keep the dependence of $m^{(k)}$ on $N$ implicit throughout the paper.

It is proved in \cite{EB14}, \cite{EB19} that under the  AT-condition \eqv(1.7), and thus by  \cite{To} in the whole high-temperature region of Definition \thv(1.def1), this iteration scheme is convergent when taking first the limit $N\rightarrow\infty$ and then $k\rightarrow\infty$. This is done via the explicit construction of a representation of the sequence $m^{(k)}$ for $N$ large. Moreover, this explicit representation of the solution makes it possible to express fairly general functions of $m^{(k)}$ in the limit of large $N$ and large $k$.

The precise statements of these results are given in  Section \thv(S3). 
Combined with the duality formula, they are used in Section \thv(S4) to prove the following statement.
\begin{theorem}
       \TH(main.theo4)
For all $(\b,h)$, $h>0$, satisfying the  AT-condition,
\be
\lim_{k\rightarrow\infty}\lim_{N\rightarrow\infty} F_{N,\b,h}^{HT}\left(m^{(k)}\right) = SK(\b,h) \quad \P\textstyle{-a.s.}
\Eq(1.theo1.1)
\ee
Thus, for all  $(\b,h)$, $h>0$, in the high-temperature region 
\be
f(\b,h)=\lim_{k\rightarrow\infty}\lim_{N\rightarrow\infty} F_{N,\b,h}^{HT}\left(m^{(k)}\right) \quad \P\textstyle{-a.s.}
\Eq(1.theo1.2)
\ee
Furthermore, \eqv(1.theo1.1), and hence \eqv(1.theo1.2), remain true if $F_{N,\b,h}^{HT}$ is replaced by $F_{N,\b,h}^{TAP}$.
\end{theorem}

We stress that  \eqv(1.theo1.2) merely follows by identifying the right-hand side of  \eqv(1.theo1.1) with the left-hand side of \eqv(1.theo1.2) via Definition  \thv(1.def1).

\begin{remark} Note that Theorem \thv(main.theo4) also holds for $F_{N,\b,h}^{TAP}$. As will be seen in Section \thv(S3), this  reflects the fact that $q_{\text{EA}}(m^{(k)})$ is concentrated near $q$, so that the quadratic term in \eqv(7.31bis) vanishes asymptotically. A similar result was obtained in Theorem 2 of \cite{CP}, which states that $F_{N,\b,h}^{TAP}$ evaluated at the vector of averaged local magnetisations
$
m=\left(\langle \s_1\rangle,\dots,\langle \s_N\rangle\right)
$,
where $\langle\cdot\rangle$ denotes the expectation with respect to the Gibbs measure, converges  in mean square to $SK(\b,h)$ under the assumption, characteristic of the high-temperature region, that the overlap is concentrated near $q$ (see e.g.~ \cite{Ta12} Vol.~II Chap.~13).  Moreover, under such a condition,  it was proved in \cite{ChTa} using a so-called cavity iteration that the solution of Bolthausen's iterative scheme converges to the local magnetizations.
\end{remark}

Theorem \thv(main.theo4) strongly suggests that the global maximum of $F_{N,\b,h}^{HT}$ is reached asymptotically at  $m^{(k)}$.  A global analysis of the function $F_{N,\b,h}^{HT}$ confirms that this is indeed the case in a subregion $\DD$ of the AT-region.  This is the content of the following theorem.

\begin{theorem}
     \TH(main.theo2) 
There exists a region $\DD$ of $(\b,h)$, $h>0$, such that in the intersection of $\DD$ and the AT-region
\be
\lim_{N\rightarrow\infty}\sup_{m\in [-1,1]^N} 
F_{N,\b,h}^{HT}\left(m\right)  
=\lim_{k\rightarrow\infty}\lim_{N\rightarrow\infty} F_{N,\b,h}^{HT}\left(m^{(k)}\right) 
\quad \P\textstyle{-a.s.}
\Eq(main.theo2.1)
\ee
Furthermore, with $\P$-probability one the supremum in \eqv(main.theo2.1) is uniquely attained, for all large enough $N$.
\end{theorem}

Before addressing the main issues of the proof of Theorem \thv(main.theo2), which occupy Sections \thv(S5) and \thv(S6), we state the variational formula from which the function $F_{N,\b,h}^{HT}$ emerges as the free energy functional of the SK model at high temperature.

\begin{theorem}
     \TH(main.theo1)
For all $(\b,h)$, $h>0$, in the intersection of the high-temperature region 
and the region $\DD$ of Theorem \thv(main.theo2), the free energy $F_{N,\b,h}$  of the SK model obeys
\be
\lim_{N\rightarrow\infty}\left|
F_{N,\b,h}
-\sup_{m\in [-1,1]^N}F_{N,\b,h}^{HT}(m)\right|=0 \quad \P-\text{a.s.}
\Eq(main.theo1.1)
\ee
If $\P$-almost sure convergence in Theorem \thv(main.theo2) is replaced by convergence in $\P$-probability, then \eqv(main.theo1.1) holds in $\P$-probability.
\end{theorem}

Theorem \thv(main.theo1) is proved in Section \thv(S7).

\subsection{The region $\DD$} 
    \TH(S1.3)
    
We now come to the detailed description of the region $\DD$ that arises from the proof of Theorem  \thv(main.theo2). When trying to establish a result like  \eqv(main.theo2.1), one naturally first checks whether $F_{N,\b,h}^{HT}$  is concave on its domain, $[-1,1]^N$, and, if not, whether it is locally concave in some subset of $[-1,1]^N$ containing $m^{(k)}$. More precisely, one looks for regions of $(\b,h)$ where, with $\P$-probability one, for all sufficiently large $N$, the Hessian of $F_{N,\b,h}^{HT}$ is strictly negative definite in as large a domain containing $m^{(k)}$  as possible. Such an analysis is carried out in Section \thv(S5). Two regions emerge that will lead to \eqv(main.theo0.2) and  \eqv(main.theo0.3), respectively. The first is the region 
\be
\DD^{(1)}
=
\bigl\{(\b,h) \mid h>0, \b<1/(1+\sqrt{q}) \bigr\}.
\Eq(1.theo0.2)
\ee
Clearly, $\DD^{(1)}$ is contained in the AT-region. Introducing a parameter $0\leq \varrho\leq 1$ and setting
\be
\vartheta(\varrho)\equiv 36(1-\varrho)+4(1-\varrho)^{1/2}+(1-\varrho)^{1/4}\left[4+\sqrt{12\left[|\ln(1-\varrho)|+2\right]}\right],
 \Eq(1.theo0.1')
\ee
the second region is $\varrho$-dependent and is defined by
\be
\DD^{(2)}_{\varrho}
=\bigl\{(\b,h) \mid  h>0, \varrho\leq q, \b\vartheta(\varrho)<1\bigr\}.
\Eq(1.theo0.1)
\ee

Based on the analysis of the Hessian, the following results are derived in Section \thv(S5.6).
\begin{theorem}
     \TH(main.theo3) 

\item(i)  For all $(\b,h)$ in $\DD^{(1)}$,  $\P$-almost surely 
\be
\begin{split}
& 
\lim_{N\rightarrow\infty}\sup_{m\in [-1,1]^N} F_{N,\b,h}^{HT}\left(m\right) 
\\
=& 
\lim_{\e\rightarrow 0}\lim_{N\rightarrow\infty}\sup_{m\in [-1,1]^N : |q_{\text{EA}}(m)-q|\leq q(1-q)\e} 
F_{N,\b,h}^{HT}\left(m\right)
\\
=&
\lim_{k\rightarrow\infty}\lim_{N\rightarrow\infty} F_{N,\b,h}^{HT}\left(m^{(k)}\right).
\end{split}
\Eq(main.theo3.2) 
\ee
Furthermore, with $\P$-probability one, for all but a finite number of indices $N$,  
$F_{N,\b,h}^{HT}$ is strictly convex on $[-1,1]^N$.

\item(ii) Let $\sqrt{3/4}\leq\varrho\leq 1$ be given. For all $(\b,h)$ in $\DD^{(2)}_{\varrho}$ satisfying the AT-condition, 
$\P$-almost surely
\be
\begin{split}
& 
\lim_{\e\rightarrow 0}\lim_{N\rightarrow\infty}\sup_{m\in [-1,1]^N : q_{\text{EA}}(m)\geq \varrho- \varrho(1-\varrho)\e} 
F_{N,\b,h}^{HT}\left(m\right)  
\\
=& 
\lim_{\e\rightarrow 0}\lim_{N\rightarrow\infty}\sup_{m\in [-1,1]^N : |q_{\text{EA}}(m)-q|\leq q(1-q)\e} 
F_{N,\b,h}^{HT}\left(m\right)
\\
=&
\lim_{k\rightarrow\infty}\lim_{N\rightarrow\infty} F_{N,\b,h}^{HT}\left(m^{(k)}\right).
\end{split}
\Eq(main.theo3.1)
\ee
The same statement is true if $0<\e\leq 1$ is held fixed. Furthermore, the supremum over the set
$
\{m\in [-1,1]^N : q_{\text{EA}}(m)\geq \varrho- \varrho(1-\varrho)\e\} 
$
is uniquely attained.
\end{theorem}

Under the assumptions of Theorem \thv(main.theo3), (ii), the supremum in \eqv(main.theo3.1) is not over the entire hypercube, but over a smaller, $\varrho$-dependent set where  $\sqrt{3/4}\leq\varrho<q$. In Section \thv(S6), we complement this result by giving conditions on $\varrho$ and $(\b,h)$ which guarantee that the supremum of $F_{N,\b,h}^{HT}$ over the set
$
\{m\in [-1,1]^N : q_{\text{EA}}(m)< \varrho\}
$
is strictly smaller than $SK(\b,h)$.
Specifically, define the region
\be
\DD^{(3)}
=
\left\{(\b,h) \mid  h/\b>2, \b^2(1-q)\leq 1, h\geq 4\right\}.
\Eq(6.theo1.1)
\ee
\begin{theorem}
     \TH(main.theo5) 
Let $\bar\varrho(\b,h)$ be the function defined in \eqv(6.theo1.3). For all $(\b,h)$ in $\DD^{(3)}$
\be
\limsup_{N\rightarrow\infty}\sup_{m\in [-1,1]^N : q_{\text{EA}}(m)\leq\bar\varrho(\b,h)} 
F_{N,\b,h}^{HT}\left(m\right)
<SK(\b,h)
 \quad \P\textstyle{-a.s.}
\Eq(1.theo00.1)
\ee
\end{theorem}
A detailed analysis of the function $\bar\varrho(\b,h)$ is carried out in Section \thv(S6).

We now come to the choice of  $\varrho$. Given that $\DD^{(2)}_{\varrho}$ increases as $\varrho$ increases from 0 to $q$, and that $\bar\varrho(\b,h)<q$  by definition (see \eqv(6.theo1.3)), the choice $\varrho=\bar\varrho(\b,h)$ in \eqv(1.theo0.1) is  allowed and optimal. The condition on $\rho$ of Theorem \thv(main.theo3), (ii), then gives a fourth and last region
\be
\DD^{(4)}
=
\left\{(\b,h) \mid \bar\varrho(\b,h)\geq  \sqrt{3/4}\right\}.
\Eq(main.theo6.2) 
\ee
In the light of the above, we arrive at the following extended version of Theorem \thv(main.theo2).
Set
\be
\DD=\DD^{(1)}\cup\left(\DD^{(2)}_{\bar\varrho(\b,h)}\cap\DD^{(3)}\cap\DD^{(4)}\right).
\Eq(main.theo6.1) 
\ee

\begin{theorem}[Theorem \thv(main.theo2) redux]
     \TH(main.theo6) 
Eq.~\eqv(main.theo2.1) holds for all $(\b,h)$ in the intersection of the AT-region and the region $\DD$ defined by \eqv(main.theo6.1) .
\end{theorem}

Equipped with the above result we now can state the full version of our main result.
\begin{theorem}[Theorem \thv(main.theo0) redux]
     \TH(main.theo7) 
Eq.~\eqv(main.theo0.1) holds for all $(\b,h)$ in the intersection of the high-temperature region of Definition \thv(1.def1) and the region $\DD$ defined by \eqv(main.theo6.1).
\end{theorem}

The next proposition gives an explicit characterisation of the region $\DD$ for sufficiently large $h/\b$ and $\b$, which justifies  the description given in  \eqv(main.theo0.2)-\eqv(main.theo0.3).
Set
\be
\begin{split}
\wt\DD^{(2)} & = \left\{(\b,h) \mid  12\b e^{-\frac{1}{9}(h/\b)^2}<1, 3\leq  {h}/{\b}\leq\b q/10\right\}.
\end{split}
\Eq(main.prop1.1) 
\ee

\begin{proposition}
     \TH(main.prop1) 
\be
\wt\DD=
\DD^{(1)}\cup\wt\DD^{(2)}
\subset(\DD\cap \{\text{AT-region}\}).
\Eq(main.prop1.1) 
\ee
\end{proposition}

The proof of Proposition \thv(main.prop1) is given in Section \thv(S6). We stress that no attempt was made  to optimise the constants in the definitions of the sets $\DD^{(2)}_{\varrho}$, $\wt\DD^{(2)}$, and $\DD^{(i)}$, $i=1,3,4$. This is because, 
as mentioned earlier, we do not expect the  region $\DD$ defined in \eqv(main.theo6.1) to be optimal due to technical artefacts.  However, it seems difficult to significantly improve the constant $1/9$ in \eqv(main.prop1.1) within our technical framework.
 
\subsection{Comments on stability} 
    \TH(S1.5)   
    
The issue of stability within the TAP approach has been extensively addressed in the physics literature \cite{BM79}, \cite{O82},  \cite{Plef82}, \cite{Plef02}, \cite{Plef20}, \cite{ABM04}. From these works, condition \eqv(1.Plef) emerges as the main stability condition. It is also believed to coincide with the AT-condition \eqv(1.7). Eq.~\eqv(1.Plef)  was originally derived in two different ways, as a convergence condition for Plefka's expansion \cite{Plef82} and as a divergence condition for the spin glass susceptibility in \cite{BM79}. Both conditions ultimately reduce to conditions on the eigenvalues of the Hessian matrix of $F_{N,\b,h}^{TAP}$, and are formulated as conditions on the \emph{empirical spectral measure} of the Hessian. We will not question the validity of these approaches and results here (see \cite{Adrien22} for recent, partial but rigorous results). We simply ask what would become of these results if we replaced $F_{N,\b,h}^{TAP}$ by $F_{N,\b,h}^{HT}$.

By \eqv(7.31bis), the Hessians of $NF_{N,\b,h}^{HT}$ and $NF_{N,\b,h}^{TAP}$, denoted by $\HH^{HT}$ and $\HH^{TAP}$ respectively, have spectral norm of order one as $N\rightarrow\infty$. Moreover, on the subspace of magnetisations satisfying $q_{\text{EA}}(m)= q$ (to which the iterative solution $m^{(k)}$ asymptotically belongs if the AT-condition \eqv(1.7) is satisfied),
these Hessians differ by a rank-one projector of non-null eigenvalue $q\b^2$. It is known that the \emph{extreme eigenvalues}  of certain Hermitian random matrices, such as Wigner matrices, can be strongly influenced by rank-one deformations, i.e., they can be detached from the spectrum for sufficiently large deformations (see, e.g., the survey paper \cite{CDM17}). Since in the present case the deformation is proportional to $q\b^2$, this effect is expected to be present at sufficiently low temperatures. The case of spectral measures is completely different. As a direct consequence of the so-called rank inequalities (see \cite{BS10}, Appendix A.6, and again \cite{CDM17}), the limiting behaviour of the spectral measure of a Hermitian random (or non-random) matrix is not modified by a finite-rank deformation. Accordingly, $\HH^{HT}$ and $\HH^{TAP}$  have the same limiting spectral measure, and will thus give the same condition \eqv(1.Plef) in the limit $N\rightarrow\infty$.

The proof of Theorem \thv(main.theo2) described in Section \thv(S1.3) is also a stability analysis. It differs from the approaches of  \cite{BM79} and \cite{Plef82} in that the functional $F_{N,\b,h}^{HT}$ is not examined in a single $m$, but globally over the whole hypercube. Indeed, as in the Laplace method for approximating integrals, we need to establish that the solution $m^{(k)}$ of the iterative scheme (1.22)-(1.23) is (a good ansatz for) the global maximum of $F_{N,\b,h}^{HT}$. To do this, it is not enough to know the nature of the Hessian $\HH^{HT}$ at just this point. This global control is achieved at the expense of the precision of the constants in \eqv(main.theo0.3), i.e., in \eqv(main.prop1.1). We also note that our analysis, via Theorem 1.6, centers on the largest eigenvalue of $\HH^{HT}$, not on the spectral measure as in \cite{BM79}, \cite{Plef82}.

The remainder of this paper is organised as follows.  
Section \thv(S2) introduces a key duality formula. 
Section \thv(S3) summarises the needed results on Bolthausen's iterative scheme and shows how they can be turned into 
$\P$-almost sure results. 
Section \thv(S4) contains the proof of Theorem  \thv(main.theo4) and 
Section \thv(S5) the proof of Theorem \thv(main.theo3). 
Theorem \thv(main.theo5) is a reformulation of Theorem \thv(6.theo1), which is stated and proved in Section \thv(S6). 
This section also contains the proofs of Theorems \thv(main.theo2), \thv(main.theo6) and Proposition  \thv(main.prop1).
Finally, in Section \thv(S7), the Hubbard-Stratonovitch transformation is used to prove Theorem  \thv(main.theo1), and the proofs of Theorems  \thv(main.theo0) and \thv(main.theo7) are given.

In the rest of the paper, $h>0, \b>0$ and  $q\equiv q(\b,h)$ is the unique solution of  \eqv(1.5).


\section{Preparatory tools: duality}
    \TH(S2)

The proof of Theorem \thv(main.theo0) hinges on a duality formula for non-convex functions known as the Clarke duality formula. The general formulation we use is that of \cite{Eke} (see  Section 4). This duality was first used in the context of spin glasses in the study of generalised Hopfield models to prove a so-called transfer principle \cite{BG98b}.

\subsection{Duality} 
    \TH(S2.1)
      
Let $A_N=(A_{ij})_{1\leq i,j\leq N}$ be the symmetric matrix with entries 
\be
A_{ij}=\frac{J_{ij}}{\sqrt N}-\b(1-q)\d_{i,j}   
\Eq(2.1)
\ee
where $\d_{i,j}$ is the Kronecker delta. It follows from know results (see, e.g., Theorem 1.2 of \cite{V14}) that there exists a subset $\O_0\subset\O$ with $\P\left(\O_0\right)=1$ such that for all $\o\in\O_0$ there exists $N_0(\o)<\infty$ such that for all $N>N_0(\o)$, $A_N$ is non-singular.  It is henceforth assumed that $N>N_0(\o)$: all results have to be understood in this sense.
Set $I^{*}(x)=\log\cosh(x)+\log 2$, $x\in\R$, and for ${\bf{h}}=(h_i)_{1\leq i\leq N}$ a given vector in $\R^N$, define the functions $\Psi_{N,\b,{\bf{h}}} : \R^N\rightarrow\R\cup\{\infty\}$ and $\Phi_{N,\b,{\bf{h}}} : \R^N\rightarrow\R$ by
\be
\begin{split}
\Psi_{N,\b,{\bf{h}}}(x)&= \frac{\b}{2}(x, A_Nx) +({\bf{h}},x)-\sum_{i=1}^{N} I(x_i),
\\
\Phi_{N,\b,{\bf{h}}}(x)&=-\frac{\b}{2}(x, A_Nx)+\sum_{i=1}^{N} I^{*}\left(\b (A_Nx)_i+h_i\right).
\end{split}
\Eq(2.2)
\ee

\begin{proposition}[Duality formula] 
    \TH(2.prop1)
$x$ is a critical point of $\Psi_{N,\b,{\bf{h}}}(x)$ if and only if $x$ is a critical point of $\Phi_{N,\b,{\bf{h}}}(x)$. These critical points are the solutions of 
the system of equations 
\be
x_i=\tanh\left(h_i+\b (A_Nx)_i\right), 
\quad 1\leq i\leq N,
 \Eq(2.prop1.1)
\ee
and at each critical point
\be
\Psi_{N,\b,{\bf{h}}}(x) = \Phi_{N,\b,{\bf{h}}}(x).
\Eq(2.prop1.2)
\ee   
\end{proposition}

By \eqv(2.2), the function $F_{N,\b,h}^{TAP}$ and $F_{N,\b,h}^{HT}$ defined in \eqv(1.8) and \eqv(7.31bis) can be written as
\bea
F_{N,\b,h}^{HT}(x)&=&\frac{1}{N}\left\{ \Psi_{N,\b,h{\bf{1}}}(x) + \frac{\b^2N}{4}\left(1-q^2\right)\right\},
\Eq(2.cor1.2bis)
\\
F_{N,\b,h}^{TAP}(x)&=&
\frac{1}{N}\left\{ \Psi_{N,\b,h{\bf{1}}}(x) + \frac{\b^2N}{4}\left[1-q^2+\left(q-q_{\text{EA}}(x)\right)^2\right]\right\} .
\Eq(2.cor1.2)
\eea

\begin{corollary}[Duality formula for the free energy functional] 
    \TH(2.cor1)
Taking ${\bf{h}}=h{\bf{1}}$  in \eqv(2.2) we have, for all solutions $x$  of the TAP equations \eqv(1.11) such that $q_{\text{EA}}(x)=q$, 
\be
F_{N,\b,h}^{TAP}(x)=F_{N,\b,h}^{HT}(x)=\frac{1}{N}\left\{ \Phi_{N,\b,h{\bf{1}}}(x) + \frac{\b^2N}{4}\left[1-q^2\right]\right\}.
\Eq(2.cor1.1)
\ee
\end{corollary}

\begin{proof} [Proof of Proposition \thv(2.prop1)]
This is a direct application of Theorem 2 in Section 4,  Chapter II of  \cite{Eke} (hereafter referred to as Theorem 2 of  \cite{Eke}),
whose notation and terminology we use. First note that $A_N$ being $\P$-almost surely non-singular, $\rm{Ker} A_N=\{\bf{0}\}$ $\P$-almost surely, where $\bf{0}$ is the vector whose coordinates are all $0$.  Next note that the functions 
\be
\textstyle
\II_N(x)\equiv\sum_{i=1}^{N} I(x_i), \quad
\II_N^{*}(x)\equiv\sum_{i=1}^{N} I^{*}\left(x_i\right)
\Eq(2.5)
\ee
form a pair of Legendre-Fenchel conjugates, and that both of them are proper, lower semicontinuous convex functions, so they are differentiable  in the interior of their domains, $\rm{int}(\rm{dom\,} \II_N^{*})$ and $\rm{int}(\rm{dom\,} \II_N)$ (the domain of a function $f:\R^N\rightarrow\R\cup\{\infty\}$ is the set ${\rm{dom}\,} f = \left\{x\in\R^N\mid f(x)<\infty\right\}$, with obvious modification under the global change of sign $f\mapsto -f$.) Clearly, condition (14) of Theorem 2 of  \cite{Eke} is satisfied. We can now conclude: the first claim of Proposition \thv(2.prop1) follows from the first two claims of Theorem 2 of \cite{Eke} combined, \eqv(2.prop1.2) is (15) therein, and  differentiation of any of the dual functions \eqv(2.2) yields \eqv(2.prop1.1). 
\end{proof}

\begin{proof} [Proof of Corollary \thv(2.cor1)] Take ${\bf{h}}=h{\bf{1}}$ in Proposition \thv(2.prop1). For this choice, the system of equations \eqv(2.prop1.1) reduces to the specialised TAP equations \eqv(1.12). The second equality in \eqv(2.cor1.1) then follows from \eqv(2.cor1.2bis) and \eqv(2.prop1.2). The first identity follows from \eqv(2.cor1.2) and the assumption that $q_{\text{EA}}(x)=q$.
\end{proof}

\begin{remark}[on maxima of the dual functions] 
As an immediate consequence of Legendre-Fenchel conjugacy, 
$
\Psi_{N,\b,{\bf{h}}}(x) \leq \Phi_{N,\b,{\bf{h}}}(x)
$
for all $x\in\R^N$. When $A_N$ is strictly positive definite, $\Phi_{N,\b,{\bf{h}}}$ is bounded from above and the set of critical points of $\Psi_{N,\b,{\bf{h}}}$ and  of $\Phi_{N,\b,{\bf{h}}}$ that are local maxima are in one-to-one correspondence. (This is the case, for example, with the Curie-Weiss model.) In the case where $A_N$ is not positive definite which interests us here, this is not true. The function $\Phi_{N,\b,{\bf{h}}}$ is unbounded. It has no maxima, only saddles, and the local maxima of $\Psi_{N,\b,{\bf{h}}}$  are saddles of the function $\Phi_{N,\b,{\bf{h}}}$. 
\end{remark}

\subsection{Dealing with approximate solutions of the TAP equations} 
    \TH(S2.2)
    
The duality formula for the free energy of Corollary \thv(2.cor1) is of little practical use if we only know approximate solutions of the TAP equations.  To deal with such a situation, let us first observe that equality in \eqv(2.prop1.2) can be achieved at any given ${\bar x}\in\R^N$ by using a modified magnetic field. Specifically, for ${\bar x}\in\R^N$ let ${\bf{\bar h}}\in\R^N$ be defined by
\be
{\bf{\bar h}}={\bf{h}}-\nabla \Psi_{N,\b,{\bf{h}}}(\bar x).
\Eq(2.4)
\ee
Here $\nabla$ is the gradient operator, that is to say, for $f:\R^N\rightarrow\R$, $\nabla f:\R^N\rightarrow\R^N$ is the vector of coordinates $\nabla f(x)=(\frac{\del}{\del x_1}f(x),\dots\frac{\del}{\del x_N}f(x))$.

\begin{lemma}
    \TH(2.lem1)
$
\displaystyle
\Psi_{N,\b,{\bf{\bar h}}}(\bar x) = \Phi_{N,\b,{\bf{\bar h}}}(\bar x).
$
\end{lemma}
\begin{proof} [Proof of Proposition \thv(2.prop1)] 
It is easy to check from the definition \eqv(2.2) that the choice  of ${\bf{\bar h}}$ in \eqv(2.4) guarantees that the function 
$\Psi_{N,\b,{\bf{\bar h}}}(x)$ has a critical point at $\bar x$, i.e., $\nabla \Psi_{N,\b,{\bf{\bar h}}}(\bar x)=0$.
The lemma then follows from an application of Proposition \thv(2.prop1).
\end{proof}

The next two results play the role of the duality formulas of Proposition \thv(2.prop1) and Corollary \thv(2.cor1), respectively, when only approximate solutions of the TAP equations are known.

\begin{proposition}
    \TH(2.prop2)
For all $\bar x\in \R^N$ such that $\sqrt{\frac{1}{N}\left\|\bar x\right\|_2^2}\leq \kappa$ for some constant $\kappa<\infty$,
\be
\frac{1}{N}\left|\Psi_{N,\b,{\bf{h}}}(\bar x)-\Phi_{N,\b,{\bf{h}}}(\bar x)\right|
\leq
\frac{1+\kappa}{\sqrt{N}}\left\|\nabla \Psi_{N,\b,{\bf{h}}}(\bar x)\right\|_2.
\Eq(2.prop2.2')
\ee 
\end{proposition}

Thus clearly, if $\bar x$ is an approximate solution of the system of equations \eqv(2.prop1.1) in the sense that
$
\frac{1}{N}\left\|\nabla \Psi_{N,\b,{\bf{h}}}(\bar x)\right\|_2^2\rightarrow 0
$
as $N\rightarrow\infty$, then,   normalised by $1/N$, \eqv(2.prop1.2) holds at $\bar x$ asymptotically, as $N\rightarrow\infty$. 
As an immediate corollary we have :

\begin{corollary}
    \TH(2.cor2)
Take ${\bf{h}}=h{\bf{1}}$ in \eqv(2.2). Then, under the assumptions and with the notations of Proposition \thv(2.prop2)
\be
\begin{split}
&
\left|\left\{ \frac{1}{N}\Phi_{N,\b,h{\bf{1}}}(\bar x)+ \frac{\b^2}{4}\left(1-q^2\right)\right\} -F_{N,\b,h}^{TAP}(\bar x)\right|
\\
\leq\,\, &\frac{1+\kappa}{\sqrt{N}}\left\|\nabla \Psi_{N,\b,{\bf{h}}}(\bar x)\right\|_2 +  \frac{\b^2}{4}\left(q-q_{\text{EA}}(\bar x)\right)^2.
\end{split}
\Eq(2.cor2.1)
\ee
The same result holds substituting $F_{N,\b,h}^{HT}$  for $F_{N,\b,h}^{TAP}$ and suppressing
the term $\frac{\b^2}{4}\left(q-q_{\text{EA}}(\bar x)\right)^2$ on the right-and side of \eqv(2.cor2.1).
\end{corollary}

\begin{proof} [Proof of Proposition \thv(2.prop2)] Using  Lemma \thv(2.lem1), we have
\be
\left|\Psi_{N,\b,{\bf{h}}}(\bar x)-\Phi_{N,\b,{\bf{h}}}(\bar x)\right|
\leq 
\left|\Psi_{N,\b,{\bf{h}}}(\bar x)-\Psi_{N,\b,{\bf{\bar h}}}(\bar x)\right|
+\left|\Phi_{N,\b,{\bf{h}}}(\bar x)-\Phi_{N,\b,{\bf{\bar h}}}(\bar x)\right|.
\Eq(2.prop2.3')
\ee
Consider the first term in the right-hand side of \eqv(2.prop2.3').
By \eqv(2.2) and \eqv(2.4), for all $x\in \R^N$
\bea
\nonumber
\left|\Psi_{N,\b,{\bf{\bar h}}}(x)-\Psi_{N,\b,{\bf{h}}}(x)\right|
=
\hspace{-6pt}&&\hspace{-6pt}
\left|({\bf{\bar h}}-{\bf{h}}, x)\right|
\\
\nonumber
=
\hspace{-6pt}&&\hspace{-6pt}       
\left|(\nabla \Psi_{N,\b,{\bf{h}}}(\bar x),x)\right|
\\
\leq 
\hspace{-6pt}&&\hspace{-6pt}
\textstyle
N\sqrt{\frac{1}{N}\left\|x\right\|_2^2} \sqrt{\frac{1}{N}\left\|\nabla \Psi_{N,\b,{\bf{h}}}(\bar x)\right\|_2^2}.
\Eq(2.prop2.4')
\eea
To bound the second term we write, recalling \eqv(2.5),
\be
\left|\Phi_{N,\b,{\bf{\bar h}}}(\bar x)-\Phi_{N,\b,{\bf{h}}}(\bar x)\right|
=
\left|
\II_N^{*}(\b A_N\bar x+{\bf{\bar h}})
-\II_N^{*}(\b A_N\bar x+{\bf{h}})
\right|.
\Eq(2.prop2.6')
\ee
Then, by the mean value theorem
\bea
\nonumber
\hspace{-6pt}&&\hspace{-6pt}
\left|
\II_N^{*}(\b A_N\bar x+{\bf{\bar h}})
-\II_N^{*}(\b A_N\bar x+{\bf{h}})
\right|
\\
\nonumber
\leq 
\hspace{-6pt}&&\hspace{-6pt}
\max_{0\leq \l\leq 1}
\left\|
\nabla\II_N^{*}
\left(
\b A_N\bar x+{\bf{h}}+(1-\l)({\bf{\bar h}}-{\bf{h}})
\right)
\right\|_2
\|{\bf{\bar h}}-{\bf{h}}\|_2
\\
\nonumber
= 
\hspace{-6pt}&&\hspace{-6pt}
\max_{0\leq \l\leq 1}
\left\{\sum_{i=1}^{N} \left[(I^{*})'\left(\b (A_N\bar x)_i+h_i+(1-\l)(\bar h_i-h_i)\right)\right]^2\right\}^{1/2}\|{\bf{\bar h}}-{\bf{h}}\|_2
\\
\leq 
\hspace{-6pt}&&\hspace{-6pt}
N \sqrt{\frac{1}{N}\left\|\nabla \Psi_{N,\b,{\bf{h}}}(\bar x)\right\|_2^2},
\Eq(2.prop2.7')
\eea
where we used \eqv(2.4) and the bound  $|(I^{*})'(z)|\leq 1$  $\forall z\in\R$  in the last line. Taking $x=\bar x$ in \eqv(2.prop2.4')  and inserting the resulting bound and \eqv(2.prop2.7') in 
 \eqv(2.prop2.3') establishes  \eqv(2.prop2.2') for all $\bar x$ such that $\sqrt{\frac{1}{N}\left\|\bar x\right\|_2^2}\leq C$ for some constant $0<C<\infty$. 
\end{proof}

\begin{proof} [Proof of Corollary \thv(2.cor2)]   
By Proposition \thv(2.prop2) and \eqv(2.cor1.2)
\be
\begin{split}
& \left|
\frac{1}{N}\left\{ \Phi_{N,\b,h{\bf{1}}}(\bar x) + \frac{\b^2N}{4}\left[1-q^2+\left(q-q_{\text{EA}}(\bar x)\right)^2\right]\right\} 
-F_{N,\b,h}^{TAP}(\bar x)
\right|
\\
\leq\,\, &\frac{1+\kappa}{\sqrt{N}}\left\|\nabla \Psi_{N,\b,{\bf{h}}}(\bar x)\right\|_2.
\end{split}
\Eq(2.cor2.3)
\ee
Clearly, \eqv(2.cor2.3) implies \eqv(2.cor2.1) while using \eqv(7.31bis) in \eqv(2.cor2.3) yields the claim for $F_{N,\b,h}^{HT}$.
\end{proof}


\section{Iterative solutions of the specialised TAP equations}
    \TH(S3)

This section recalls the results of \cite{EB14}, \cite{EB19}, which are central to this paper.
We mostly use the notation of  \cite{EB19}. 
In particular,  inner products and norms are rescaled: given two vectors $x,y\in\R^N$, we write
\be
\langle x, y\rangle=\frac{1}{N}(x,y) \quad\text{and}\quad \|x\|_{2,N}=\frac{1}{\sqrt N}\|x\|_{2}.
\Eq(3.1)
\ee

\subsection{Convergence of the iterative scheme} 
    \TH(S3.1)
    
The main result of \cite{EB14} is the following convergence theorem. 

\begin{theorem}[Theorem 2.1 of \cite{EB14}] 
    \TH(3.theo1)
Assume that $h>0$. If $\b>0$ is below the AT-line, i.e.~if
\be
\b^2E\frac{1}{\cosh^4(\b\sqrt{q}Z+h)}\leq 1,
\Eq(1.7bis)
\ee
then
\be
\lim_{k,k'\rightarrow\infty}\limsup_{N\rightarrow\infty}\E\left\|m^{(k)}-m^{(k')} \right\|_{2,N}^2=0.
\Eq(3.theo1.1)
\ee
If inequality in \eqv(1.7bis) is strict then there exists $0<\l(\b,h)<1$ and $C>0$ such that, for all $k$,
\be
\limsup_{N\rightarrow\infty}\E\left\|m^{(k+1)} -m^{(k)}\right\|_{2,N}^2=C\l^k(\b,h).
\Eq(3.theo1.2)
\ee
\end{theorem}

\subsection{Approximate solution of the iterative scheme} 
    \TH(S3.2)
    
The proof of Theorem \thv(3.theo1) relies on the construction of an explicit representation of a sequence of approximate solutions,  
$\bar m^{(k)}$, $k\geq 1$,  of the iterative scheme \eqv(1.13)-\eqv(1.14).  To express  $\bar m^{(k)}$, a number of notations and definitions have to be introduced. We stick as much as possible to those of  \cite{EB19}, which gives a technically simplified approach to the proofs of  \cite{EB14}, based on the symmetric representation \eqv(1.01). Denoting by $g$ and $g^t$, respectively, the (non-symmetric) $N\times N$ matrix with entries $g_{i,j}$ and its transpose, we write for simplicity
\be
{J_N}/{\sqrt N}\equiv \bar g = (g+g^T)/\sqrt 2.
\Eq(3.2)
\ee 
We now construct several sequences: 
\begin{itemize}
\item[(i)] of real numbers $\{\g_k\}_{k\geq 1}$ and $\{\varrho_k\}_{k\geq 1}$,
\item[(ii)] of random $N\times N$ matrices, $g^{(k)}$ and $\rho^{(k)}$, $k\geq 1$, and 
\item[(iii)] of random vectors, $\phi^{(k)}$, $\xi^{(k)}$, $\eta^{(k)}$ and $\zeta^{(k)}$ in $\R^N$.
\end{itemize}
Below, $Z$, $Z'$, $Z_1$, \emph{etc}.~are standard Gaussian random variables, always assumed independent when appearing in the same formula. We denote their joint expectation by $E$. Define 
\be
\g_1=E \tanh(h+\b Z), \quad \varrho_1=\sqrt{q} \g_1
\Eq(3.3)
\ee
and recursively,
\be
\varrho_k=\psi(\varrho_{k-1}), \quad 
\g_k=
\frac{\varrho_k-\sum_{j=1}^{k-1}\g_j^2}{\sqrt{q-\sum_{j=1}^{k-1}\g_j^2}},
\Eq(3.4)
\ee
where, setting 
$
\Tanh(x)\equiv\tanh(h+\b x)
$,
the function $\psi : [0,q] \rightarrow [0, q]$ is defined by
\be
\psi(t)=E\Tanh(\sqrt{t}Z+ \sqrt{q-t}Z')\Tanh(\sqrt{t}Z+ \sqrt{q-t}Z'').
\Eq(3.5)
\ee

We now define recursions for $g^{(k)}$ and $\phi^{(k)}$, as well as for the closely related vectors ${\bar h}^{(k)}$ and $\bar m^{(k)}$. 
For $k=1$,
\be
g^{(1)}=g,\,\,\,\bar m^{(1)}=\sqrt q \bf{1}
\Eq(3.6)
\ee
where  $\bf{1}$ is as in \eqv(1.13). Assume that $g^{(s)}$, $\phi^{(s)}$ and $\bar m^{(s)}$ are defined for $s\leq k$ and set 
\be
\xi^{(s)}=g^{(s)}\phi^{(s)},
\,\,\,
\eta^{(s)}= {g^{(s)}}^{T}\phi^{(s)},
\,\,\, \text{and}\,\,\,
\zeta^{(s)}=\frac{\xi^{(s)}+\eta^{(s)}}{\sqrt 2}=\overline{g^{(s)}}\phi^{(s)}.
\Eq(3.7)
\ee
Next write
$
\G^2_{k-1}=\sum_{j=1}^{k-1}\g_j^2
$,
set ${\bar h}^{(1)}_i=\tanh^{-1}(\sqrt{q})$ for all $1\leq i\leq N$ and for $k\geq 1$ set
\be
{\bar h}^{(k+1)}_i = h+\b\sum_{s=1}^{k-1}\g_s\zeta^{(s)}_i+\b\sqrt{q-\G^2_{k-1}}\zeta^{(k)}_i,\quad 1\leq i\leq N,
\Eq(3.8)
\ee
\be
\bar m^{(k+1)}_i = \tanh\left({\bar h}^{(k+1)}_i\right),\quad 1\leq i\leq N.
\Eq(3.8bis)
\ee
Finally, defining the vectors $\phi^{(k)}$ as
\be
\phi^{(k+1)}=\frac{
\bar m^{(k+1)}-\sum_{s=1}^{k}\langle\bar m^{(k+1)},\phi^{(s)}\rangle\phi^{(s)}
}
{
\left\|
\bar m^{(k+1)}-\sum_{s=1}^{k}\langle\bar m^{(k+1)},\phi^{(s)}\rangle\phi^{(s)}
\right\|_{2,N}
}
\Eq(3.9)
\ee
the matrix $g^{(k)}$ is defined recursively through
\be
g^{(k+1)}=g^{(k)}-\rho^{(k)}
\Eq(3.10)
\ee
where
\be
\rho^{(k)}=\xi^{(k)}\otimes\phi^{(k)} + \phi^{(k)}\otimes\eta^{(k)} -\langle\phi^{(k)},\xi^{(k)}\rangle\left(\phi^{(k)}\otimes\phi^{(k)}\right),
\Eq(3.11)
\ee
and where, given two vectors $x,y\in\R^N$, $x \otimes  y$ the $N\times N$ denotes the matrix with entries
\be
(x \otimes  y)_{i,j}=\frac{x_iy_j}{N}.
\Eq(3.12)
\ee

All of the above objects are well defined and their properties are well understood.  We refer the reader to \cite{EB19}  for more details. 

We now specify in which sense the vector $\bar m^{(k)}$ of coordinates $\bar m^{(k)}_i$ defined in \eqv(3.8bis) is an approximation of the vector $m^{(k)}$ of coordinates $m^{(k)}_i$, $1\leq i\leq N$. For this we use additional notations.

We write $X_N\simeq Y_N$ if $X_N$ and $Y_N$ are two random variables, possibly depending  on extra parameters (such as $\b$, $h$, $k$), if there exists a constant $C>0$,  possibly depending on these parameters, but not on $N$, such that
\be
\P\left(|X_N- Y_N|\geq t\right)\leq Ce^{-t^2N/C}.
\Eq(3.13)
\ee
If $X^N=\left(X^N_i\right)_{i\leq N}$ and $Y^N=\left(Y^N_i\right)_{i\leq N}$ are two sequences of random vectors in $\R^N$ we write $X^N\approx Y^N$ if
\be
\frac{1}{N}\sum_{i=1}^N\left|X^N_i-Y^N_i\right|\simeq 0.
\Eq(3.14)
\ee

Let ${h}^{(k+1)}$ be the vector in $\R^N$ defined through
\be
h^{(k+1)} = h+\b \bar g m^{(k)}-\b^2(1-q)m^{(k-1)}
\Eq(4.32)
\ee
and denote by ${\bar h}^{(k+1)}$ the vector of coordinates ${\bar h}^{(k+1)}_i $, $1\leq i\leq N$ (see \eqv(3.8)).

\begin{lemma}
    \TH(3.lem1)
For all $\b>0$ and all $k\in\N$
\be
{\bar h}^{(k+1)} \approx {h}^{(k+1)},
\Eq(3.lem1.1)
\ee
\be
 \bar m^{(k)} \approx m^{(k)}.
\Eq(3.lem1.2)
\ee
\end{lemma}

\begin{proof} [Proof of Lemma  \thv(3.lem1)] 
These results are proved in \cite{EB14} by explicitly constructing the iterates of the scheme \eqv(1.14). This construction uses the matrix ${J_N}/{\sqrt N}$, while that of \cite{EB19} which we have adopted, uses the matrix $g$ from the representation \eqv(3.2). This leads to slightly different objects. In order to prove the lemma, the iterative method of \cite{EB14} must therefore be adapted to the present setting. We will not give the simple but lengthy details of this adaptation.  Let us only point out that in \cite{EB14}, the analogue of the sequence $\bar m^{(k)}$ is given by the right-hand side of (1.4). It is formulated more precisely as $\hat m^{(k)}$, defined above (5.2) (see also $\bar m^{(k)}$ above (5.10)). Then \eqv(3.lem1.2) is obtained by combining Remark 5.2 and Remark 5.4 of \cite{EB14}, and one checks that this statement is in substance deduced from \eqv(3.lem1.1). By repeating the iteration of \cite{EB14} using $g$ instead of ${J_N}/{\sqrt N}$, one arrives at an expression similar to  (1.4) in \cite{EB14}.  The gain is that one now has a structurally simple expression for the matrix $g^{(k)}$, from which the term $\sum_{j}g_{i,j}^{(k-1)}m_j^{(k-1)}$ in the right-hand side of (1.4) in \cite{EB14} can easily be shown to be Gaussian, conditional on the sigma algebra $ \GG_{k-2}=\left\{\xi^{(s)},\zeta^{(s)}\mid s\leq k-2\right\}$. Its variance can be calculated. For finite $N$ it still depends on $\GG_{k-2}$ in a complicated way, but by a SLLN it is proved to be non-random in the limit $N\rightarrow\infty$ and given by $\sqrt{q-\G^2_{k-1}}$.
\end{proof}

This section concludes with two important structural results from \cite{EB14} and \cite{EB19}.
\begin{lemma}[Lemma 2 of \cite{EB19}]
    \TH(3.lem2)
  
\hfill\break \noindent a) $\{\varrho_k\}$ is an increasing sequence. $\lim_{k\rightarrow\infty}\varrho_k=q$ if and only if \eqv(1.7bis) is satisfied. If inequality in \eqv(1.7bis) is strict, this convergence is exponentially fast.

\noindent b) $\G^2_{k-1}=\sum_{j=1}^{k-1}\g_j^2<\varrho_k<q$ holds for all $k$ and $\sum_{j=1}^{\infty}\g_j^2=q$ holds if and only if \eqv(1.7bis) is satisfied.
\end{lemma}

The following result is stated for $\bar m^{(k)}$ as Proposition 6 of \cite{EB19}, and for $m^{(k)}$ as Proposition 2.5 of \cite{EB14}.
(In \cite{EB14}, $\phi^{(k)}$ is defined as in \eqv(3.9) substituting  $m^{(k)}$ for  $\bar m^{(k)}$. The sequences $\g_j$ and $\varrho_j$ are defined in the same way in both papers.)

\begin{proposition}[Proposition 6 of \cite{EB19} \& Proposition 2.5 of \cite{EB14}]
    \TH(3.prop1)
a) For any $j<k$, 
\be
\left\< m^{(k)},\phi^{(j)}\right\>\simeq\g_j.
\Eq(3.prop1.1)
\ee
b) For any $k\in\N$
\be
\bigl\| m^{(k)}\bigr\|^2_{2,N}\simeq q,
\Eq(3.prop1.2)
\ee
and for $j<k$
\be
\left\< m^{(k)},m^{(j)}\right\>\simeq\varrho_j.
\Eq(3.prop1.3)
\ee
The proposition holds unchanged if $\bar m^{(k)}$ is  substituted for  $m^{(k)}$.
\end{proposition}

\subsection{Almost sure convergence results} 
    \TH(S3.4)
       
The formulations of Lemma \thv(3.lem1) and Proposition \thv(3.prop1) are particularly well suited  to prove convergence results in 
mean of order $p$. As the next lemma shows, they can easily be reformulated as almost sure convergence results. Note, however, that regardless of the chosen notion of convergence, the limits in $N$ and $k$ cannot be interchanged, as can be seen from \eqv(3.lem3'.2).

\begin{lemma}
    \TH(3.lem3)

Let $X^{(k)}_N$ and $Y^{(k)}_N$ be two sequences of random variables depending on a parameter $k\in\N$. Assume that $X^{(k)}_N\simeq Y^{(k)}_N$ in the sense of \eqv(3.13), namely, with a constant $C\equiv C(k)>0$ depending a priori on $k$ but not on $N$. Then, there exists a subset $\O^*\subset\O$ with $\P\left(\O^*\right)=1$, that does not depend on $k$ and such that on $\O^*$, the following holds: for all  $k \geq 1$
\be
\lim_{N\rightarrow\infty}\left|X^{(k)}_N-Y^{(k)}_N\right|=0,
\Eq(3.lem3.1)
\ee
and
\be
\lim_{k\rightarrow\infty}\lim_{N\rightarrow\infty}\left|X^{(k)}_N-Y^{(k)}_N\right|=0.
\Eq(3.lem3.1')
\ee
The limits in \eqv(3.lem3.1') are iterated, as opposed to joint. If the constant $C(k)$ depends on other parameters (e.g., $\b$, $h$), then $\O^*$ also depends on these parameters. The lemma \thv(3.lem3) holds in the case of sequences $X^{(k_1,\dots,k_m)}_N$ and $Y^{(k_1,\dots,k_m)}_N$ that depend on finitely many parameters  $k_1<\dots<k_m$ in $\N$, with $C=C(k_1,\dots,k_m)>0$. 
\end{lemma}

\begin{proof}[Proof of Lemma \thv(3.lem3)]
Given $\varepsilon>0$, define the collections of sets 
\be
\begin{split}
\O^{(k)}_N(\varepsilon) &=\left\{\o\in\O \,:\, \left|X^{(k)}_N(\o)-Y^{(k)}_N(\o)\right|\leq \varepsilon\right\}, \quad \forall N\geq 1, k\geq 1,
\\
\O^{(k)}(\varepsilon)&=\textstyle{\bigcup_{N^*\geq 1}\bigcap_{N\geq N^*}}\O^{(k)}_N(\varepsilon), \quad\forall  k\geq 1,
\\
\O^{(k)}_0 & =\cap_{\varepsilon>0}\O^{(k)}(\varepsilon).
\end{split}
\Eq(3.lem3'.3)
\ee
Further introduce the quantity
\be
\d_{k,N}^2\equiv \frac{2C(k)}{N}\left(\log\bigl( k\sqrt{|\log C(k)|}\bigr) +\log N\right)
\Eq(3.lem3'.2)
\ee
and note that by  \eqv(3.13) (with $C\equiv C(k)$), 
\be
\P\left(\Bigl(\O^{(k)}_N(\d_{k,N})\Bigr)^c\right)\leq \frac{1}{N^2k^2}.
\Eq(3.lem3'.1')
\ee

Let now $k$ be fixed. Then, $\d_{k,N}$ is a decreasing function of $N$ that decays to zero as $N\uparrow\infty$.  Hence, for all $\varepsilon>0$, there exists $N(k,\varepsilon)$ such that $\d_{k,N}<\varepsilon$ for all $N\geq N(k,\varepsilon)$ and, for all $N^*\geq N(k,\varepsilon)$,
\be
\P\left({\textstyle\bigcup_{N\geq N^*}\bigl(\O^{(k)}_N(\varepsilon)\bigr)^c}\right)
\leq 
\sum_{N>N^*}\frac{1}{N^2k^2}<\frac{1}{N^*k^2}<\infty,
\Eq(3.lem3'.4)
\ee
where we used  \eqv(3.lem3'.1'). By Borel-Cantelli lemma, for all $\varepsilon>0$
\be
\P\left(\left(\O^{(k)}(\varepsilon)\right)^c\right)
=\lim_{{N^*\rightarrow\infty}\atop{N^*\geq N(k,\varepsilon)}}\P\left({\textstyle\bigcup_{N\geq N^*}\bigl(\O^{(k)}_N(\varepsilon)\bigr)^c}\right)
=0.
\Eq(3.lem3'.5)
\ee
From this and the monotony of $\O^{(k)}(\varepsilon)$ it then follows in a standard way that
\be
\P\bigl(\O^{(k)}_0\bigr)=\P\left(\lim_{N\rightarrow\infty}\left|X^{(k)}_N(\o)-Y^{(k)}_N(\o)\right|=0\right)=1.
 \Eq(3.lem3'.6)
\ee
Since the above holds true for any given $k\geq 1$,
\be
\P\left({\textstyle\bigcap_{k\geq 1}}\O^{(k)}_0\right)=\P\left(\forall{k\geq 1}\lim_{N\rightarrow\infty}\left|X^{(k)}_N(\o)-Y^{(k)}_N(\o)\right|=0\right)=1,
\Eq(3.lem3'.7)
\ee
which proves  \eqv(3.lem3.1).

In order to prove \eqv(3.lem3.1') we must take the additional  limit $k\rightarrow\infty$. Set
\be
\O_0=\textstyle{\bigcup_{k^*\geq 1}\bigcap_{k\geq k^*}}\O^{(k)}_0
\Eq(3.lem3'.8')
\ee
and observe that
$\cap_{k\geq 1}\O^{(k)}_0\subseteq \O_0$. Thus, by \eqv(3.lem3'.7)
\be
1=\P\left({\textstyle\bigcap_{k\geq 1}}\O^{(k)}_0\right)\leq \P\left(\O_0\right)\leq 1,
\Eq(3.lem3'.10')
\ee
and so,
\be
\P\bigl(\O_0\bigr)=\P\left( \lim_{k\rightarrow\infty}\lim_{N\rightarrow\infty}\left|X^{(k)}_N(\o)-Y^{(k)}_N(\o)\right|=0\right)=1.
 \Eq(3.lem3'.12')
\ee
This proves \eqv(3.lem3.1'). (Note that alternatively, we could have taken this second $k\rightarrow\infty$ limit by summing \eqv(3.lem3'.4) over $k$ and using Borel-Cantelli lemma.) It is clear from the proof that  the limits in \eqv(3.lem3'.12') cannot be interchanged. Taking $\O^*=\bigl(\cap_{k\geq 1}\O^{(k)}_0\bigr)\textstyle{\bigcap}\O_0=\cap_{k\geq 1}\O^{(k)}_0$ completes the proof of the lemma in the case of sequences that depend on a single parameter, $k\in\N$. The extension of the proof to the case of sequences depending on  finitely many parameters  $k_1<\dots<k_m$ is straightforward.
\end{proof}

\begin{theorem}[Almost sure version of Theorem \thv(3.theo1)]
    \TH(3.theo2)
Under the assumptions of Theorem \thv(3.theo1), there exists a subset $\O^*(\b,h)\subset\O$ with $\P\left(\O^*(\b,h)\right)=1$ such that on $\O^*(\b,h)$, 
\be
\lim_{k,k'\rightarrow\infty}\limsup_{N\rightarrow\infty}\left\|m^{(k)}-m^{(k')} \right\|_{2,N}^2=0.
\Eq(3.theo2.1)
\ee
\end{theorem}

\begin{proof}[Proof of Theorem \thv(3.theo2)]  This is a simple modification of the proof of Theorem \thv(3.theo1). In view of Lemma \thv(3.lem3), it follows from (b) of Proposition \thv(3.prop1) that there exists a subset $\O'(\b,h)\subset\O$ with $\P\left(\O'(\b,h)\right)=1$ such that on $\O'(\b,h)$, for all $k\geq 1$
\be
\lim_{N\rightarrow\infty}\bigl\|m^{(k)}\bigr\|^2_{2,N}=q.
\Eq(3.cor3.2)
\ee
Similarly, there exists a subset $\O''(\b,h)\subset\O$ with 
$\P\left(\O''(\b,h)\right)=1$ such that on $\O''(\b,h)$, for all $k>k'\geq 1$
\be
\lim_{N\rightarrow\infty}\bigl\<m^{(k)},m^{(k')}\bigr\>=\varrho_{k'}.
\Eq(3.cor3.3)
\ee
Set $\O^*(\b,h)=\O'(\b,h)\cap \O''(\b,h)$ and write
\be
\bigl\|m^{(k)}-m^{(k')} \bigr\|_{2,N}^2=\bigl\|m^{(k)}\bigr\|_{2,N}^2+\bigl\|m^{(k')} \bigr\|_{2,N}^2-2\bigl\<m^{(k)},m^{(k')}\bigr\>.
\Eq(3.cor3.4)
\ee
Using \eqv(3.cor3.2) and \eqv(3.cor3.3) to first take the limit $N\rightarrow\infty$ in \eqv(3.cor3.4), and using (a) of Lemma \thv(3.lem2) together with \eqv(3.lem3.1') of Lemma \thv(3.lem3) to next take the limits $k',k\rightarrow\infty$, the theorem follows.
\end{proof}

We conclude this section by stating the almost sure versions of three technical lemmata from \cite{EB19} needed in Section \thv(S4). Making use of Lemma \thv(3.lem3), their proofs are \emph{mutatis mutandis} those of their original versions. They are omitted.

\begin{lemma}[Almost sure version of  Lemma 13 of \cite{EB19}]
\TH(A.lem1bis)
Under the assumptions of Lemma 13 of \cite{EB19}, there exists a subset $\O^*(\b,h)\subset\O$ with $\P\left(\O^*(\b,h)\right)=1$ such that on $\O^*(\b,h)$, 
\be
\lim_{k\rightarrow\infty}\lim_{N\rightarrow\infty}\left\| g^{(k)}\bar m^{(k)}\right\|^2_{2,N}=0.
\Eq(A.lem1bis.1)
\ee
\end{lemma}

\begin{lemma}[Almost sure version of Lemma 16 of \cite{EB19}]
\TH(A.lem3bis)
There exists a subset $\O^*(\b,h)\subset\O$ with $\P\left(\O^*(\b,h)\right)=1$ such that on $\O^*(\b,h)$, the following holds:
for all $n>2$
\be
\textstyle
\lim_{N\rightarrow\infty}
\left\langle\bar m^{(n)},\zeta^{(n-1)}\right\rangle
=
\b(1-q)\sqrt{q-\sum_{j=1}^{n-2}\g_j^2},
\Eq(A.lem3bis.1)
\ee
and for $1\leq m\leq n-2$
\be
\lim_{N\rightarrow\infty}
\left\langle\bar m^{(n)},\zeta^{(m)}\right\rangle
=
\b\g_m(1-q).
\Eq(A.lem3bis.2)
\ee
\end{lemma}

\begin{lemma}[Almost sure version of  Lemma 14 of \cite{EB19}]
\TH(A.lem2bis)
Under the assumptions of Lemma 14 of \cite{EB19}, there exists a subset $\O^*(\b,h)\subset\O$ with $\P\left(\O^*(\b,h)\right)=1$ such that on $\O^*(\b,h)$, for all $k\geq 2$
\be
\lim_{N\rightarrow\infty}\frac{1}{N}\sum_{i=1}^Nf\left({\bar h}_i^{(k)}\right)=Ef(h+\b\sqrt{q} Z).
\Eq(A.lem2bis.1)
\ee
\end{lemma}


\section{Proof of Theorem \thv(main.theo4)}
    \TH(S4)
 \subsection{Proof of Theorem \thv(main.theo4)} 
    \TH(S4.1)

Theorem \thv(main.theo4) is an immediate consequence of the following two lemmata. Let $m^{(k)}$, $k\geq 1$,
be defined through the iterative scheme \eqv(1.13)-\eqv(1.14).

\begin{lemma}
\TH(4.lem4)
For all $(\b,h)$, $h>0$, satisfying the  AT-condition,
\be
\lim_{k\rightarrow\infty}\lim_{N\rightarrow\infty}
\left|\left\{ \frac{1}{N}\Phi_{N,\b,h{\bf{1}}}\left(m^{(k)}\right)+ \frac{\b^2}{4}\left(1-q^2\right)\right\} -F_{N,\b,h}^{HT}\left(m^{(k)}\right)\right|
=0\quad \P-\text{a.s.}
\Eq(4.lem4.1)
\ee
The same result holds with  $F_{N,\b,h}^{TAP}$ substituted for $F_{N,\b,h}^{HT}$.
\end{lemma}

\begin{lemma}
\TH(4.lem3)
For all $(\b,h)$, $h>0$, satisfying the  AT-condition,
\be
\lim_{k\rightarrow\infty}\lim_{N\rightarrow\infty}
\left\{\frac{1}{N}\Phi_{N,\b,h{\bf{1}}}\left(m^{(k)}\right) 
+ \frac{\b^2}{4}\left(1-q^2\right)\right\}
=
SK(\b,h)
\quad \P-\text{a.s.}
\Eq(4.lem3.1)
\ee
\end{lemma}

Before proving these two lemmata, we state a well known bound on the spectral radius $r({J_N}/{\sqrt N})$ of ${J_N}/{\sqrt N}$ (defined above \eqv(9.2.1)), which will be needed repeatedly.

\begin{theorem} [Geman \cite{Ge}]
\TH(4.theo1)
$\P$-almost surely, $\displaystyle\lim_{N\rightarrow\infty}r\Bigl(\sfrac{J_N}{\sqrt N}\Bigr)= 2$.
\end{theorem}

We are now ready to prove Lemma \thv(4.lem4). In the sequel, the notation \eqv(3.1) is used without reminder. 

\begin{proof} [Proof of Lemma \thv(4.lem4)]
We first prove the lemma for $F_{N,\b,h}^{TAP}$. For this we use  Corollary \thv(2.cor2) with $\bar x=m^{(k)}$. By definition of $m^{(k)}$, $q_{\text{EA}}\left(m^{(k)}\right)\leq 1$.  Hence \eqv(2.cor2.1) holds with $\kappa=1$ and
\be
\begin{split}
&
\left|\left\{ \frac{1}{N}\Phi_{N,\b,h{\bf{1}}}\left(m^{(k)}\right)+ \frac{\b^2}{4}\left(1-q^2\right)\right\} - F_{N,\b,h}^{TAP}\left(m^{(k)}\right)\right|
\\
\leq\,\, & 
2\sqrt{\left\|\nabla \Psi_{N,\b,h{\bf{1}}}\left(m^{(k)}\right)\right\|_{2,N}^2}+  \frac{\b}{4}\left(q-q_{\text{EA}}\left(m^{(k)}\right)\right)^2.
\end{split}
\Eq(4.lem4.2)
\ee
It remains to prove that  in the AT-region, passing  first to the limit $N\rightarrow\infty$ and then $k\rightarrow\infty$,
both terms on the right-hand side of \eqv(4.lem4.1) vanish $\P$-almost surely. 

Proceeding as in the proof  of Theorem \thv(3.theo2) (see \eqv(3.cor3.2)),
\be
\lim_{k\rightarrow\infty}\lim_{N\rightarrow\infty}\frac{\b}{4}\left(q-q_{\text{EA}}\left(m^{(k)}\right) \right)^2=0 \quad \P-{\text{a.s.}}
\Eq(4.lem4.4)
\ee
Let us now establish that in the AT-region,
\be
\lim_{k\rightarrow\infty}\lim_{N\rightarrow\infty}\sqrt{\left\|\nabla \Psi_{N,\b,h{\bf{1}}}\left(m^{(k)}\right)\right\|_{2,N}^2} = 0  
\quad \P-{\text{a.s.}}
\Eq(4.lem4.5)
\ee
By \eqv(2.2),
\be
\frac{\del}{\del x_i}\Psi_{N,\b,{\bf{h}}}(x)=\b (A_Nx)_i +h- I'(x_i)=0, \quad 1\leq i\leq N,
\Eq(4.lem4.6)
\ee
where $I'(x)=(I^*(x))^{-1}= \tanh^{-1}(x)$. Thus
\be
\left\|\nabla \Psi_{N,\b,h{\bf{1}}}\left(m^{(k)}\right)\right\|_{2,N}^2 
=
\frac{1}{N}\sum_{i=1}^{N}\left[
\b \left(A_Nm^{(k)}\right)_i +h-  \tanh^{-1}\bigl(m^{(k)}_i\bigr)
\right]^2
\Eq(4.lem4.7)
\ee
Now, since $m^{(k)}$ obeys \eqv(1.14) we have, using \eqv(2.1),  
\be
\tanh^{-1}\bigl(m^{(k)}_i\bigr)=h+\b \left(\sfrac{J_N}{\sqrt N}m^{(k-1)}\right)_i-\b^2(1-q)m^{(k-2)}_i
\quad 1\leq i\leq N.
\Eq(4.lem4.8)
\ee
Inserting \eqv(4.lem4.8) in \eqv(4.lem4.7) yields
\bea
\nonumber
\hspace{-6pt}&&\hspace{-6pt}
\left\|\nabla \Psi_{N,\b,h{\bf{1}}}\left(m^{(k)}\right)\right\|_{2,N}^2 
\\
\nonumber
\hspace{-6pt}&=&\hspace{-6pt}
\frac{1}{N}\sum_{i=1}^{N}\left[
\b\left(\sfrac{J}{\sqrt N}\left(m^{(k)}-m^{(k-1)}\right)\right)_i-\b^2(1-q)\left(m^{(k)}-m^{(k-2)}\right)_i
\right]^2
\\
\hspace{-6pt}&\leq&\hspace{-6pt}
2\b r^2\bigl(\sfrac{J}{\sqrt N}\bigr)\left\|m^{(k)}-m^{(k-1)}\right\|^2_{2,N}
+
2\b^2(1-q)\left\|m^{(k)}-m^{(k-2)}\right\|^2_{2,N},
\Eq(4.lem4.9)
\eea
where we used the Courant-Fisher minimax principle in the last line. Eq.~\eqv(4.lem4.5) now readily follows from Theorem \thv(4.theo1), Theorem \thv(3.theo2) and the continuity of $x\rightarrow\sqrt x$ on $\R^+$.

Taking the limits of both sides of \eqv(4.lem4.2), it follows from \eqv(4.lem4.4) and \eqv(4.lem4.5) that \eqv(4.lem4.1) holds with $F_{N,\b,h}^{HT}$ replaced by $F_{N,\b,h}^{TAP}$. That  \eqv(4.lem4.1)  holds for the function $F_{N,\b,h}^{HT}$ itself follows from \eqv(4.lem4.5) by virtue of the last claim of Corollary \thv(2.cor2). Lemma \thv(4.lem4) is proven.
\end{proof}

\begin{proof} [Proof of Lemma \thv(4.lem3)] 
Consider the first term in the left-hand side of \eqv(4.lem3.1). By \eqv(2.2),
\be
\frac{1}{N}\Phi_{N,\b,h{\bf{1}}}\left(m^{(k)}\right)
=-\frac{\b}{2N}\left(m^{(k)}, A_Nm^{(k)}\right)+\frac{1}{N}\sum_{i=1}^{N}I^{*}\left(\b \left(A_Nm^{(k)}\right)_i+h\right).
\Eq(4.8')
\ee
In view of \eqv(1.6) and \eqv(4.8'), Lemma \thv(4.lem3) will be proven if we can establish the following two claims hold $\P$-almost surely:
\bea
\lim_{k\rightarrow\infty}\lim_{N\rightarrow\infty}
\frac{\b}{2N}\left(m^{(k)}, A_Nm^{(k)}\right)
\hspace{-6pt}&=&\hspace{-6pt}
\frac{\b^2}{2} q(1-q),
\Eq(4.lem3.3)
\\
\lim_{k\rightarrow\infty}\lim_{N\rightarrow\infty}
\frac{1}{N}\sum_{i=1}^{N} I^{*}\left(\b \left(A_Nm^{(k)}\right)_i+h\right)
\hspace{-6pt}&=&\hspace{-6pt}
\log 2 + E\log\cosh(\b\sqrt{q}Z+h).\quad
\Eq(4.lem3.4)
\eea

We first prove \eqv(4.lem3.3). Recall the definition of the symmetrized matrix $\bar g$  from \eqv(3.2). Using the notations \eqv(3.1) and  the definition \eqv(2.1), we have
\be
\left\langle m^{(k)}, A_N m^{(k)}\right\rangle=\left\langle m^{(k)}, \bar g m^{(k)}\right\rangle-\b(1-q)\left\|m^{(k)}\right\|^2_{2,N}.
\Eq(4.9')
\ee
Proceeding as in the proof of \eqv(4.lem4.4),
\be
\lim_{k\rightarrow\infty}\lim_{N\rightarrow\infty}\b(1-q)\left\|m^{(k)}\right\|^2_{2,N}=\b q(1-q)
\quad \P-\text{a.s.}
\Eq(4.lem3.5)
\ee
To deal with the first term on the right-hand side of \eqv(4.9'), we decompose this term into
\be
\left\langle m^{(k)}, \bar g m^{(k)}\right\rangle=T^{(k)}_{N,1}+T^{(k)}_{N,2}+T^{(k)}_{N,3}+T^{(k)}_{N,4}
\Eq(4.12')
\ee
where
\be
T^{(k)}_{N,1}= \left\langle m^{(k)}-\bar m^{(k)}, \bar g \left(m^{(k)}-\bar m^{(k)}\right)\right\rangle,
\Eq(4.13')
\ee
\be
T^{(k)}_{N,2} = 2 \left\langle m^{(k)}-\bar m^{(k)}, \bar g m^{(k)}\right\rangle,
\Eq(4.14')
\ee
\be
T^{(k)}_{N,3} = \left\langle \bar m^{(k)}, \overline{g^{(k)}}\bar m^{(k)}\right\rangle,
\Eq(4.16')
\ee
\be
T^{(k)}_{N,4}= \left\langle \bar m^{(k)}, (\bar g - \overline{g^{(k)}})\bar m^{(k)}\right\rangle
\Eq(4.15')
\ee
where $\overline{g^{(k)}}$ is the symmetrized matrix defined in \eqv(3.7). By Cauchy-Schwarz's inequality, the bound $\left\|m^{(k)}\right\|^2_{2,N}\leq 1$ and the Courant-Fisher minimax principle we have
\bea
\bigl| T^{(k)}_{N,1}\bigr|
\hspace{-6pt}&\leq &\hspace{-6pt}
r(\bar g) \left\| m^{(k)}-\bar m^{(k)}\right\|^2_{2,N},
\\
\bigl| T^{(k)}_{N,2}\bigr|
\hspace{-6pt}&\leq &\hspace{-6pt}
2\left\| m^{(k)}-\bar m^{(k)}\right\|_{2,N}\left\| \bar g m^{(k)}\right\|_{2,N}
\leq 2r^2(\bar g)\left\| m^{(k)}-\bar m^{(k)}\right\|_{2,N},
\\
\bigl|T^{(k)}_{N,3}\bigr|
\hspace{-6pt}&\leq &\hspace{-6pt}
\left\|\bar m^{(k)}\right\|_{2,N}\left\|  \overline{g^{(k)}}\bar m^{(k)}\right\|_{2,N}
\leq 
\left\| g^{(k)}\bar m^{(k)}\right\|_{2,N}.
\Eq(4.lem3.6)
\eea
On the one hand, by \eqv(A.lem1bis.1) of Lemma \thv(A.lem1bis),
$
\lim_{k\rightarrow\infty}\lim_{N\rightarrow\infty}T^{(k)}_{N,3}=0
$
$\P$-a.s..
On the other hand, by \eqv(3.lem1.2) of Lemma \thv(3.lem1) and Lemma \thv(3.lem3).
\be
\lim_{k\rightarrow\infty}\lim_{N\rightarrow\infty}  \left\| m^{(k)}-\bar m^{(k)}\right\|^2_{2,N}=0\quad \P-\text{a.s.},
\Eq(4.lem3.7)
\ee
and by this and Theorem \thv(4.theo1), 
$
\lim_{k\rightarrow\infty}\lim_{N\rightarrow\infty} T^{(k)}_{N,i}=0
$
$\P$-a.s. for $i=1,2$. 

It remains to deal with $T^{(k)}_{N,4}$. By \eqv(3.10) and \eqv(3.11)
\bea
\Eq(4.21'')
T^{(k)}_4
\hspace{-6pt}&= &\hspace{-6pt}
\sum_{l=1}^{k-1}\left\langle \bar m^{(k)}, \overline{\rho^{(k)}}\bar m^{(k)}\right\rangle
\\
\hspace{-6pt}&= &\hspace{-6pt}
\sum_{l=1}^{k-1}\left\{
2\left\langle\bar m^{(k)},\zeta^{(l)}\right\rangle
\left\langle\bar m^{(k)},\phi^{(l)}\right\rangle
-
\left\langle\phi^{(l)},\zeta^{(l)}\right\rangle
\left\langle\bar m^{(k)},\phi^{(l)}\right\rangle^2
\right\}.
\Eq(4.21''')
\eea
We treat the two terms in the brackets separately, starting with the first.
For $k\geq 2$, set
\be
a^{(k,l)}\equiv
\begin{cases} 
\b(1-q)\sqrt{q-\G^2_{k-2}}&\mbox{if }\, l=k-1, \\
\b(1-q)\g_{l}&\mbox{if }\, 1\leq l\leq k-2.
\end{cases}
\Eq(4.22')
\ee
By Lemma \thv(A.lem3bis), there exists a subset $\O_1(\b,h)\subset\O$ with $\P\left(\O_1(\b,h)\right)=1$ such that on $\O_1(\b,h)$,
for all $1\leq l\leq k-1$ and $k\geq 2$
\be
\lim_{N\rightarrow\infty} \left\langle\bar m^{(k)},\zeta^{(l)}\right\rangle=a^{(k,l)},
\Eq(4.lem3.8)
\ee
whereas by \eqv(3.prop1.1) of Proposition \thv(3.prop1) and Lemma \thv(3.lem3),  there exists a subset $\O_2(\b,h)\subset\O$ with $\P\left(\O_2(\b,h)\right)=1$ such that on $\O_2(\b,h)$, for all $1\leq l\leq k-1$ and $k\geq 2$
\be
\lim_{N\rightarrow\infty} \left\langle\bar m^{(k)},\phi^{(l)}\right\rangle=\g_{l}.
\Eq(4.lem3.9)
\ee
Thus, on $\O_1(\b,h)\cap \O_2(\b,h)$,  for all $1\leq l\leq k-1$ and $k\geq 2$
\be
\lim_{N\rightarrow\infty}
\left\langle\bar m^{(k)},\zeta^{(l)}\right\rangle\left\langle\bar m^{(k)},\phi^{(l)}\right\rangle
=\g_la^{(k,l)}.
\Eq(4.lem3.10)
\ee
We now turn to the second term in the brackets in \eqv(4.21'''). 
Note that
$\left\|\phi^{(l)}\right\|^2_{2,N}= 1$
and
$\left\|m^{(k)}\right\|^2_{2,N}\leq 1$
so that
$
\left|\left\langle\bar m^{(k)},\phi^{(l)}\right\rangle\right|\leq 1
$,
while by  Lemma 11 of \cite{EB19} $\left\langle\phi^{(l)},\zeta^{(l)}\right\rangle$ is Gaussian with  mean zero and variance $1/N$.
From this it follows that there exists a subset $\O_3(\b,h)\subset\O$ with $\P\left(\O_3(\b,h)\right)=1$ such that on $\O_3(\b,h)$, for all $1\leq l\leq k-1$ and $k\geq 2$
\be
\lim_{N\rightarrow\infty} \left\langle\phi^{(l)},\zeta^{(l)}\right\rangle \left\langle\bar m^{(k)},\phi^{(l)}\right\rangle^2 =0.
\Eq(4.lem3.11)
\ee
Plugging \eqv(4.lem3.10) and \eqv(4.lem3.11) in \eqv(4.21''') we obtain that on $\cap_{i=1}^3\O_i(\b,h)$, for all $1\leq l\leq k-1$ and $k\geq 2$
\be
\lim_{N\rightarrow\infty}T^{(k)}_{N,4}
=
2\sum_{l=1}^{k-1}\g_la^{(k,l)}
=
2\b(1-q)\g_{k-1}\sqrt{q-\G^2_{k-2}}
+
2\b(1-q)\G^2_{k-2},
\Eq(4.lem3.12)
\ee
where $\G^2_{k-2}$ is defined below \eqv(3.7). Using (b) of Lemma \thv(3.lem2) and reasoning as in the proof of \eqv(3.lem3.1') to pass to the limit $k\rightarrow\infty$, we get that on $\cap_{i=1}^3\O_i(\b,h)$
\be
\lim_{k\rightarrow\infty}\lim_{N\rightarrow\infty}T^{(k)}_{N,4}
=
2\b q(1-q).
\Eq(4.lem3.13)
\ee
Inserting the above results in \eqv(4.12'), we obtain that
\be
\lim_{k\rightarrow\infty}\lim_{N\rightarrow\infty} \left\langle m^{(k)}, \bar g m^{(k)}\right\rangle
=
2\b q(1-q)
\quad \P-\text{a.s.}
\Eq(4.lem3.14)
\ee
Finally,  \eqv(4.lem3.3) follows from \eqv(4.lem3.14), \eqv(4.lem3.5) and \eqv(4.9').

It remains to prove \eqv(4.lem3.4). From the definition of $I^{*}$, \eqv(2.1) and \eqv(4.32), we have
\be
\frac{1}{N}\sum_{i=1}^{N}  I^{*}\left(\b \left(A_Nm^{(k)}\right)_i+h\right)
=\log 2 + \frac{1}{N}\sum_{i=1}^{N}\log\cosh\bigl({h}^{(k+1)}_i\bigr).
\Eq(4.33')
\ee
Recalling  the definition of ${\bar h}^{(k+1)}_i$ from \eqv(3.8), we decompose the last term in \eqv(4.33') into 
\be
\frac{1}{N}\sum_{i=1}^{N}\log\cosh\bigl({h}^{(k+1)}_i\bigr)
= {\overline T}^{(k)}_{N,1}+ {\overline T}^{(k)}_{N,2}+ E\log\cosh(\b\sqrt{q}Z+h)
\Eq(4.34')
\ee
where
\bea
 {\overline T}^{(k)}_{N,1}
\hspace{-6pt}&= &\hspace{-6pt}
\frac{1}{N}\sum_{i=1}^{N}\left[\log\cosh\bigl({h}^{(k+1)}_i\bigr)-\log\cosh\bigl({\bar h}^{(k+1)}_i\bigr)\right],
\Eq(4.35')
\\
 {\overline T}^{(k)}_{N,2}
\hspace{-6pt}&=&\hspace{-6pt}
\frac{1}{N}\sum_{i=1}^{N}\log\cosh\bigl({\bar h}^{(k+1)}_i\bigr)-E\log\cosh(\b\sqrt{q}Z+h).
\Eq(4.36')
\eea
Now, by Lemma \thv(A.lem2bis),  there exists a subset $\O_4(\b,h)\subset\O$ with $\P\left(\O_4(\b,h)\right)=1$ such that on $\O_4(\b,h)$, for all $k\geq 1$
\be
\lim_{N\rightarrow\infty} {\overline T}^{(k)}_{N,2}=0.
\Eq(4.lem3.15)
\ee
Turning to ${\overline T}^{(k)}_{N,1}$, observe that 
\be
\textstyle
\bigl|{\overline T}^{(k)}_{N,1}\bigr|
\leq 
\frac{1}{N}\sum_{i=1}^{N}\bigl|{h}^{(k+1)}_i-{\bar h}^{(k+1)}_i\bigr|.  
\Eq(4.38'')
\ee
It then follows from \eqv(4.38''),  \eqv(3.lem1.1) of Lemma \thv(3.lem1) and the definition \eqv(3.14) that
\be
\textstyle
\bigl|{\overline T}^{(k)}_{N,1}\bigr|
\leq 
\frac{1}{N}\sum_{i=1}^{N}\bigl|{h}^{(k+1)}_i-{\bar h}^{(k+1)}_i\bigr|  \simeq 0,
\Eq(4.lem3.16)
\ee
and so, by Lemma \thv(3.lem3) there exists a subset $\O_5(\b,h)\subset\O$ with $\P\left(\O_5(\b,h)\right)=1$ such that on $\O_5(\b,h)$, 
for all $k\geq 1$
\be
\lim_{N\rightarrow\infty} {\overline T}^{(k)}_{N,1}=0.
\Eq(4.lem3.17)
\ee
Inserting \eqv(4.lem3.15), and \eqv(4.lem3.17) in \eqv(4.34') and again reasoning as in the proof of \eqv(3.lem3.1') to pass to the limit $k\rightarrow\infty$, the claim of \eqv(4.lem3.4) follows. The proof of Lemma \thv(4.lem3) is done.
\end{proof}
 
The proof of Theorem \thv(main.theo4) is complete.


\section{Analysis of the Hessian of $\Psi_{N,\b,{\bf{h}}}$}
    \TH(S5)
    
\subsection{Main results and strategy} 
    \TH(S5.01) 
Let $\HH_N(m)\equiv\HH_{N,\b,{\bf{h}}}(m)$ denote the Hessian matrix of $\Psi_{N,\b,{\bf{h}}}$ at $m$. Given $\sqrt{3/4}\leq\varrho\leq 1$, consider the set
\be
\BB^c_{N,\e}(\varrho)=\left\{m\in [-1,1]^N : q_{\text{EA}}(m)\geq \varrho-\varrho(1-\varrho)\e\right\}.
\Eq(5.theo4.1)
\ee
The main result of this section establishes that  $\HH_N(m)$ is strictly negative definite on $\BB^c_{N,\e}(\varrho)$ for all $(\b,h)$ 
in $\DD^{(1)}\cap\DD^{(2)}_{\varrho}$, where $\DD^{(1)}$ and $\DD^{(2)}_{\varrho}$ are defined in \eqv(1.theo0.2) and \eqv(1.theo0.1), respectively.

\begin{theorem}
\TH(5.theo4)
\item{(i)} Let $\sqrt{3/4}\leq\varrho$ be given. For all $(\b,h)$ in $\DD^{(2)}_{\varrho}$
\be
\P\left(\bigcup_{N_0}\bigcap_{N\geq N_0}
\left\{
\sup_{\e\in[0,1]}\sup_{m\in \BB^c_{N,\e}(\varrho)}\l_{max}(\HH_N(m))<0
\right\}
\right)=1,
\Eq(5.theo4.2)
\ee
\item{(ii)} For all $(\b,h)$ in $\DD^{(1)}$
\be
\P\left(\bigcup_{N_0}\bigcap_{N\geq N_0}
\left\{
\sup_{m\in [-1,1]^N}\l_{max}(\HH_N(m))<0
\right\}
\right)=1.
\Eq(5.theo4.2bis)
\ee
\end{theorem}

The proofs of the two items of Theorem \thv(5.theo4) follow very different strategies. Item (ii) is based on elementary arguments and is proved at the very end of the section, while the proof of item (i) occupies most of it. 
 
First, we set up the necessary matrix notation. Let $M_N$ be an $N\times N$ real symmetric matrix. We write $M_N<0$ (resp.,  $M_N>0$) when  $M_N$ is strictly negative (resp., positive) definite. The notation $M_N=\diag(\mu_1, \dots, \mu_N)$ indicates that  $M_N$ is diagonal with diagonal entries $M_{ii}=\mu_i$. We denote by $\l_{max}(M_N)$ and $\l_{min}(M_N)$ the largest and smallest eigenvalues of  $M_N$, and by $r(M_N)$ its spectral radius, i.e.~the largest absolute value of its eigenvalues. The operator norm of $M_N$ is denoted by $\|M_N\|$ and defined by
\be
\|M_N\|=\sup_{x:\|x\|_2=1}\|M_Nx\|_2.
\Eq(9.2.1)
\ee
Finally, we recall that eigenvalues, spectral radius and operator norm are related through
\be
\textstyle
\sup_{x:\|x\|_2=1}|(x,M_Nx)|
=r(M_N)
=\sqrt{r(M^2_N)}
=\|M_N\|.
\Eq(9.2.2)
\ee

The central idea behind the proof of  item (i) of Theorem \thv(5.theo4) is to replace the condition on the negative definiteness of the Hessian matrix $\HH_N(m)$,  which, as we will see, is of additive form, by a condition on the spectrum of a matrix, which takes the form of a product  and is easier to study. To this end, we go back to the function  $\Psi_{N,\b,{\bf{h}}}$ and, taking second  partial derivatives, we get
\be
\HH_N(m)=\b \frac{J_N}{\sqrt N} -B_N(m)
\Eq(5.1)
\ee
where $B_N(m)=\diag(b_1(m), \dots, b_N(m))$ and
\be
b_i(m)=\b^2(1-q)+ \frac{1}{1-m_i^2}, \quad 1\leq i\leq  N.
\Eq(5.2)
\ee
Since $b_i(m)>0$ for each $1\leq i\leq  N$, $B_N(m)$ is strictly positive definite and has rank $N$. 
Denoting by $B^{1/2}_N(m)$ its strictly unique positive definite square root and by $B^{-1/2}_N(m)$ its inverse, \eqv(5.1) may be written as
\be
\HH_N(m)=B^{1/2}_N(m)\left(\b C_N(m) -I_N\right)B^{1/2}_N(m)
\Eq(5.3)
\ee
where $I_N$ is the identity matrix in $\R^N$ and
\be
C_N(m)\equiv B^{-1/2}_N(m)\frac{J_N}{\sqrt N}B^{-1/2}_N(m).
\Eq(5.4)
\ee
By definition, $\HH_N(m)$ is strictly negative definite if $(x, \HH_N(m) x)<0$ for all non-zero $x\in\R^N$, and this is true if and only if all eigenvalues of $\HH_N(m)$ are strictly negative. Since $ \rank B^{1/2}_N(m) =\rank B_N(m)= N$, it follows from \eqv(5.3) that $\HH_N(m)<0$ if and only if 
$
\b C_N(m) -I <0
$.
Thus, a necessary and sufficient condition for $\HH_N(m)$ to be strictly negative definite is
\be
\b \l_{max}(C_N(m))< 1.
\Eq(5.5)
\ee
Theorem \thv(5.theo4) will then be deduced from the following proposition.
Let $\SS_{N,\e}(\varrho)$ denote the spherical shell 
\be
\SS_{N,\e}(\varrho)=\left\{m\in [-1,1]^N : |q_{\text{EA}}(m)-\varrho|\leq \varrho(1-\varrho)\e\right\}\subset\BB^c_{N,\e}(\varrho).
\Eq(5.prop2.15)
\ee
For $0\leq\varrho\leq q$, recalling the definition of $\vartheta(\varrho)$ from \eqv(1.theo0.1'), set
\bea
f_1(\varrho)
\hspace{-6pt}&=&\hspace{-6pt}\vartheta(\varrho)
\Eq(5.prop1.11)
\\
f_2(\b,q)
\hspace{-6pt}&=&\hspace{-6pt}
2\b-\left(\b^2(1-q)+ 1\right).
\Eq(5.prop1.11bis)
\eea

\begin{proposition}
\TH(5.prop1)
\item{(i)}  For all $(\b,h)$, $h>0$, all $\varrho\geq \sqrt{3/4}$ and all $N$ large enough
\be
\P\left(\sup_{\e\in[0,1]}\sup_{m\in \SS_{N,\e}(\varrho)}
\l_{max}(C_N(m))\geq
f_1(\varrho)+16\sqrt{\sfrac{\log N}{\sqrt N}}
\right)
\leq
6e^{-\sqrt{N}}.
\Eq(5.prop1.1)
\ee
\item{(ii)}  For all $(\b,h)$, $h>0$, and  all $N$ large enough
\be
\P\left(
\sup_{m\in [-1,1]^N}
\l_{max}(\HH_N(m))\geq  f_2(\b,q)+\b N^{-1/4}\right)\leq 2e^{-\sqrt{N}/4}.
\Eq(5.prop1.22)
\ee
\end{proposition}

Again, the bulk of the proof of Proposition \thv(5.prop1) is devoted to proving assertion (i), which will itself be deduced from an analogous statement for  the operator norm $\|C_N(m)\|$. From \eqv(5.1) and \eqv(5.4) we see, by comparing the condition $\HH_N(m)<0$ and \eqv(5.5), that we have turned a condition on the spectrum of a deterministic  full-rank perturbation of a standard Gaussian Wigner random matrix, $J_N/\sqrt{N}$, into a condition on the spectrum of the sole Wigner-type Gaussian random matrix $C_N$ whose entries, $C_{ij}(m)$,  are also independent centred Gaussians, but now have non-identical variances. More specifically, setting
\be
a_i(m)\equiv b^{-1}_i(m), \quad 1\leq i\leq N,
\Eq(5.prop2.32)
\ee 
we have
\be
C_{ij}(m)=J_{ij}v_{ij}(m), \quad v_{ij}(m)\equiv\sqrt{\frac{a_i(m)a_j(m)}{N}}.
\Eq(5.theo3.8)
\ee
While there are few tools available to deal with non-finite rank perturbations of Wigner random matrices (see \cite{CDMFF}, \cite{CaPe}), the question of finding bounds on the norm or largest eigenvalue of non-homogeneous Wigner-type random matrices  such as $C_N(m)$ has recently witnessed significant developments \cite{BvH}, \cite{LvHY}. The proof of Proposition \thv(5.prop1) is based on results of \cite{BvH} which, for the convenience of the reader, we state below in a version specialised to the matrices \eqv(5.theo3.8).

\begin{theorem}[Theorem 1.1 and Corollary 3.9 of \cite{BvH}]
\TH(5.theo3)
Given $m\in[-1,1]^N$ set 
\be
\s(m)=\max_{i}\sqrt{\sum_{j}v_{ij}^2(m)}, \quad {\sigma_\star(m)}=\max_{ij}|v_{ij}(m)|.
\Eq(5.theo3.0)
\ee
Then,  for any $0<\varepsilon\leq 1/2$
\be
\E\|C_N(m)\|\leq (1+\varepsilon)\left\{2\s(m)+\frac{6}{\sqrt{\log(1+\varepsilon)}}{\sigma_\star(m)}\sqrt{\log N}\right\}.
\Eq(5.theo3.1)
\ee
In addition, for any $0<\varepsilon\leq 1/2$ and $t\geq 0$
\be 
\P\left(
\|C_N(m)\|
\geq 
\E\|C_N(m)\|+t\sqrt{N}\s_\star(m) 
\right)
\leq
e^{-Nt^2/4}.
\Eq(5.theo3.2)
\ee
\end{theorem}
The bound \eqv(5.theo3.1) is expected to be best when the coefficients $v_{ij}(m)$ are not too inhomogeneous (see Lemma 3.14 and its corollary in Section 3.5 of \cite{BvH}). However, since we want to bound the supremum of $\|C_N(m)\|$ over $\SS_{N,\e}(\varrho)$,  we also have to deal with matrices with highly inhomogeneous coefficients.  To explain the difficulty we face, let us first state a useful lemma.
Set 
\be
\varrho^{\pm}_{\e}\equiv \varrho\pm \varrho(1-\varrho)\e,
\Eq(5.lem6.0)
\ee
and for $y\in[0,1]$
\be
c_0(y)\equiv \left(\b^2(1-q)+ 1/y\right)^{-1}.
\Eq(5.lem6.00)
\ee
\begin{lemma}
\TH(5.lem6)
\bea
(1-\varrho^{+}_{\e})c_0(1)
\leq \hspace{-6pt} &
\displaystyle\sup_{m\in\SS_{N,\e}(\varrho)}\frac{1}{N}\sum_{j}a_j(m)
&\hspace{-6pt}
\leq
c_0(1-\varrho^{-}_{\e})
\Eq(5.lem6.1)
\\
c_0(1-\varrho^{-}_{\e})
\leq
\hspace{-6pt}&
\displaystyle\sup_{m\in\SS_{N,\e}(\varrho)} \max_{i}a_i(m)
&\hspace{-8pt}\leq 
c_0(1)\leq 1.
\Eq(5.lem6.2)
\eea
\end{lemma}
If we were to apply Theorem \thv(5.theo3) to bound $\sup_{m\in\SS_{N,\e}(\varrho)}\|C_N(m)\|$, we would have to replace 
$\s(m)$ and ${\sigma_\star(m)}$ by their supremum, which, according to Lemma \thv(5.lem6), gives
\be
\textstyle
\sup_{m\in\SS_{N,\e}(\varrho)}\s(m)
\leq
\sqrt{
c_0(1)c_0(1-\varrho^{-}_{\e})
},
\quad
\sup_{m\in\SS_{N,\e}(\varrho)}{\sigma_\star(m)}\leq c_0(1)/\sqrt{N}.
\Eq(5.lem6.4)
\ee
Thus, the deviation term that comes from  Gaussian concentration in \eqv(5.theo3.2) would outweigh the mean value of the operator norm. Since $t$ must be chosen large enough to control the supremum over $\SS_{N,\e}(\varrho)$, we see that this approach cannot provide a useful bound on the operator norm. Instead of applying Theorem \thv(5.theo3) directly to $C_N(m)$,  we will introduce upper and lower thresholds on $a_j(m)$ and decompose the matrix into a sum of three terms,
$
C_N(m)=\underline{C}_N(m)+C^{\circ}_N(m)+\overline{C}_N(m)
$,
depending on the size of these coefficients. Theorem \thv(5.theo3) is then be applied to a matrix $C^{\circ}_N(m)$ with ``tamed coefficient'', that are neither too large nor too small compared to the average $\frac{1}{N}\sum_{j}a_j(m)$.

\subsection{Decomposition of the matrix $C_N(m)$}
    \TH(S5.0)

Given $0<\theta_{\star}<\theta\leq 1$ to be chosen later, define the sets
\bea
\L\equiv\L(m,\theta)\hspace{-6pt}&=&\hspace{-6pt}\left\{1\leq i\leq N \,\Big|\,1-m^2_i>\theta\right\},
\Eq(5.prop2.4)
\\
\L_{\star}\equiv\L_{\star}(m,\theta_{\star})
\hspace{-6pt}&=&\hspace{-6pt}\{1\leq i\leq N \mid 1-m^2_i>\theta_{\star}\}.
\Eq(5.prop2.5)
\eea
Note that $\L(m,\theta)\subset\L_{\star}(m,\theta_{\star})$. Using $\theta_{\star}$ and $\L_{\star}$, we define the modified coefficients
\be
{\tilde a}_i(m)=
\begin{cases}
a_i(m) 
& \text{if}\,\,\, i\in\L_{\star},
\\
c_0(\theta_{\star})
& \text{if}\,\,\, i\in\L^c_{\star},
\end{cases},
\quad 1\leq i\leq N.
\Eq(5.prop2.35)
\ee 
Unlike $a_i(m)$, ${\tilde a}_i(m)$  is bounded from below and its derivative is bounded on $[-1,1]$. By analogy with \eqv(5.2), \eqv(5.prop2.32) and \eqv(5.4) we set ${\wt B}_N(m)=\diag({\tilde b}_1(m), \dots, {\tilde b}_N(m))$ where 
${\tilde b}_i(m)={\tilde a}^{-1}_i(m)$, $1\leq i\leq N$,  and 
\be
{\wt C}_N(m)\equiv {\wt  B}^{-1/2}_N(m)\frac{J_N}{\sqrt N}{\wt B}^{-1/2}_N(m).
\Eq(5.prop2.34)
\ee
As usual, we denote the entries of $ {\wt C}_N(m)$ by $\wt{C}_{ij}(m)$.
The set $\L$ is then used to define the matrix $\overline{C}_N(m)$ with entries
\be
\overline{C}_{ij}(m)\equiv
\begin{cases}
{\wt C}_{ij}(m) & \text{for all}\,\,\, (i,j)\in\L\times\L,\\
0 &  \text{else}.
\end{cases}
\Eq(5.prop2.8)
\ee
If we also set
\bea
\underline{C}_N(m)\hspace{-6pt}&\equiv&\hspace{-6pt}C_N(m)-{\wt C}_N(m),
\Eq(5.prop2.6)
\\
C^{\circ}_N(m)\hspace{-6pt}&\equiv&\hspace{-6pt}{\wt C}_N(m)-\overline{C}_N(m),
\Eq(5.prop2.7)
\eea
we obtain the decomposition 
\be
C_N(m)=\underline{C}_N(m)+C^{\circ}_N(m)+\overline{C}_N(m).
\Eq(5.prop2.7')
\ee

By the triangle inequality 
\be
\left\|C_N(m)\right\|
\leq \left\|\underline{C}_N(m)\right\|+\left\|C^{\circ}_N(m)\right\|+\left\|\overline{C}_N(m)\right\|.
\Eq(5.prop2.9)
\ee
We begin by establishing a priori bounds  on the operator norm of $\underline{C}_N(m)$ and $\overline{C}_N(m)$. To do so, we use the following notations. Given an $N\times N$ matrix $M_N\equiv\left(M_{ij}\right)_{1\leq i,j\leq N}$ and a subset $U$ of 
$\{1,\dots,N\}$, we denote by $M_{U}$  the $N\times N$ matrix of entries 
$
\left(M_{U}\right)_{ij}=M_{ij}
$
for all $(i,j)\in U\times U$ and 
$
\left(M_{U}\right)_{ij}=0
$
else. In this way,
\be
\overline{C}_N(m)=\bigl({\wt C}(m)\bigr)_{\L}=\bigl({\wt  B}^{-1/2}(m)\bigr)_{\L}\frac{J_{\L}}{\sqrt N}\bigl({\wt  B}^{-1/2}(m)\bigr)_{\L},
\Eq(5.prop2.10)
\ee
and by the submultiplicativity property of matrix norms 
\be
\left\|\overline{C}_N(m)\right\|
\leq 
\left\|\bigl({\wt  B}^{-1/2}(m)\bigr)_{\L}\right\|^2\left\|\frac{J_{\L}}{\sqrt N}\right\|
\leq
c_0(1)
\left\|\frac{J_{\L}}{\sqrt N}\right\|,
\Eq(5.prop2.11)
\ee
where we used that  by \eqv(5.prop2.35), the inclusion $\L(m,\theta)\subset\L_{\star}(m,\theta_{\star})$
and  \eqv(5.lem6.2),
\be
\textstyle
\bigl\|\bigl({\wt  B}^{-1/2}(m)\bigr)_{\L}\bigr\|^2\leq \sup_{i\in\L}{\tilde a}_i(m)=\sup_{i\in\L}a_i(m)\leq 
c_0(1).
\Eq(5.prop2.110)
\ee
Turning to  $\underline{C}_N(m)$, we set
\bea
\Delta_N(m)\hspace{-6pt}&\equiv&\hspace{-6pt} B^{-1/2}_N(m)-{\wt B}^{-1/2}_N(m)
\\
\hspace{-6pt}&=&\hspace{-6pt}\diag\left(\bigl(\sqrt{a_i(m)}-\sqrt{{\tilde a}_i(m)}\bigr), 1\leq i\leq N\right)
=(\Delta(m))_{\L^c_{\star}}
\eea
where the last equality follows from \eqv(5.prop2.35).
Thus, by \eqv(5.prop2.6), \eqv(5.4) and \eqv(5.prop2.34)
\be
\begin{split}
\underline{C}_N(m) 
& 
= 
(\Delta(m))_{\L^c_{\star}}\frac{J_N}{\sqrt N} {\wt B}^{-1/2}_N(m)
+{\wt B}^{-1/2}_N(m)\frac{J_N}{\sqrt N} (\Delta(m))_{\L^c_{\star}}
\\
&
+(\Delta(m))_{\L^c_{\star}}\frac{J_N}{\sqrt N}(\Delta(m))_{\L^c_{\star}}.
\end{split}
\Eq(5.prop2.12)
\ee
By the triangle identity and the submultiplicativity property, bounding the matrix $\underline{C}_N(m)$ reduces to bounding each of the matrices that appear on the right-hand side of \eqv(5.prop2.12). Proceeding as in \eqv(5.prop2.110) to bound $\bigl\|{\wt B}^{-1/2}_N(m)\|$, and observing that since $\sqrt{a_i(m)}\leq \sqrt{{\tilde a}_i(m)}$ for all $i\in\L^c_{\star}$,
\bea
\bigl\|(\Delta(m))_{\L^c_{\star}}\bigr\|
\hspace{-6pt}&\leq &\hspace{-6pt}
\textstyle
\sup_{i\in\L^c_{\star}}
\bigl|\sqrt{a_i(m)}-\sqrt{{\tilde a}_i(m)}\bigr|
\leq\sqrt{
c_0(\theta_{\star})
},\quad
\Eq(5.prop2.130)
\eea
we get
\be
\left\|\underline{C}_N(m)\right\|
\leq 3\sqrt{c_0(1)c_0(\theta_{\star})}\left\|\frac{J_N}{\sqrt N}\right\|.
\Eq(5.prop2.13)
\ee
Inserting \eqv(5.prop2.11) and \eqv(5.prop2.13) in \eqv(5.prop2.9) and taking the supremum over $m\in\SS_{N,\e}(\varrho)$,
obtain
\be
\begin{split}
 &\sup_{m\in\SS_{N,\e}(\varrho)}\left\|C_N(m)\right\|  
\\
\leq  &
 \sup_{m\in\SS_{N,\e}(\varrho)}
c_0(1)\left\|\frac{J_{\L}}{\sqrt N}\right\|
+\sup_{m\in\SS_{N,\e}(\varrho)}\left\|C^{\circ}_N(m)\right\| 
+3\sqrt{c_0(1)c_0(\theta_{\star})}\left\|\frac{J_N}{\sqrt N}\right\|.
\end{split}
\Eq(5.prop2.16)
\ee
Proposition \thv(5.prop1)  then follows directly from the next three propositions, which give tail probability bounds for each operator norm in  \eqv(5.prop2.16). For $0\leq x\leq 1$, define the function
\be
\JJ(x)=-\left\{x\log x+(1-x)\log(1-x)\right\}
\Eq(5.prop2.17')
\ee
and set
\be
\textstyle
\bar{L}= \sup_{m\in\SS_{N,\e}(\varrho)}|\L(m,\theta)|.
\Eq(5.prop2.17)
\ee 
Note that if $\bar{L}=0$ then the first term in the right-hand side of \eqv(5.prop2.16) drops out. 

\begin{proposition}
\TH(5.prop2)
Set  $x=\bar{L}/{N}$ if $1\leq\bar{L} \leq N/2$ and $x=1/2$ if $N/2<\bar{L} \leq N$. Then, for all $N$ and all $\l>0$
\be
\P\left(
\sup_{m\in\SS_{N,\e}(\varrho)}\left\|\frac{J_{\L}}{\sqrt N}\right\|
\geq 
2\sqrt{x}+2\sqrt{\JJ\left(x\right)+\l}
\right)
\leq 
2e^{-N\l}.
\Eq(5.prop2.1)
\ee
\end{proposition}

\begin{proposition}
\TH(5.prop4)
For all $N$ and all $\l>0$
\be
\P\left(
\sup_{m\in\SS_{N,\e}(\varrho)}\left\|\frac{J_N-J_{\L}}{\sqrt{N}}\right\|
\geq 
4+2\sqrt{\log 2}+\l
\right)
\leq
4e^{-N\l^2/4}.
\Eq(5.prop4.1)
\ee
\end{proposition}

Recall the notation \eqv(5.lem6.00) and for the sake of brevity set
\be
\varrho^{\pm}_{\e,\tilde\varepsilon} = \varrho\pm [\varrho(1-\varrho)\e+ \varrho\tilde\varepsilon].
\Eq(5.prop3.0)
\ee

\begin{proposition}
\TH(5.prop3)
For all $0<\tilde\varepsilon\leq 1$, $\e\geq 0$, $0<\theta\leq\theta_{\star}\leq 1$, $\b,h>0$ and all $N>80$
\be
\P\left( 
\sup_{m\in\SS_{N,\e}(\varrho)}\left\|C^{\circ}_N(m)\right\| 
\geq 
c_0(1)r(\varrho,\e,\tilde\varepsilon, \theta, \theta_{\star}, \b)
\right)
\\
\leq 2e^{-\sqrt N/4}
\Eq(5.prop3.1)
\ee
where 
\be
\begin{split}
r(\varrho,\e,\tilde\varepsilon, \theta, \theta_{\star}, \b)
& \equiv
2\sqrt{c_0(1-\varrho^{-}_{\e,\tilde\varepsilon})+c_0(\theta_{\star})}
+
\sqrt{2c_0(\theta)\log\left(2\pi e
\frac{\varrho^+_{\e,2\tilde\varepsilon}}{\varrho\tilde\varepsilon}
\right)}
\\
 & 
+12\left(\frac{c_0(\theta_{\star})}{\theta_{\star}}\right)^{3/2}
\sqrt{\varrho\tilde\varepsilon\frac{1-\theta_{\star}}{\theta_{\star}}}
+
15\sqrt{\frac{\log N}{c_0(1)\sqrt N}}.
\end{split}
\Eq(5.lem1.19)
\ee
\end{proposition}

In complement to Theorem \thv(5.theo3) we state below a classical tail probability bound on the operator norm of $\|J_N/\sqrt{N}\|$.
\begin{proposition}
\TH(5.prop5)
For all $t\geq 0$,
$
\P\bigl(\|J_N/\sqrt{N}\|\geq 2+t\bigr) \leq 2e^{-Nt^2/4}.
$
\end{proposition}

\begin{proof}
This follows from the one-sided concentration bound for $\l_{max}(J_N/\sqrt{N})$ stated below (3.5) in \cite{L07} and the fact that by symmetry of the distribution of the spectrum of $J_N/\sqrt{N}$, the same bound holds for $-\l_{min}(J_N/\sqrt{N})$.
\end{proof}

\subsection{Proof of Proposition \thv(5.prop2) and Proposition \thv(5.prop4)} 
    \TH(S5.1)

\begin{proof}[Proof of Proposition \thv(5.prop2)]
Using \eqv(5.prop2.17), we break  $\SS_{N,\e}(\varrho)$ into 
$\SS_{N,\e}(\varrho)=\cup_{\underline{L}\leq \ell\leq \bar{L}}\EE_{N,\ell,\e}(\varrho)$,
\be
\EE_{N,\ell,\e}(\varrho)=\left\{m\in [-1,1]^N : |q_{\text{EA}}(m)-\varrho|\leq \varrho(1-\varrho)\e\,\,\, \text{and}\,\,\,  |\L(m,\theta)|=\ell \right\}.
\nonumber
\ee
It is worth making the construction of these sets explicit. Given $m\in\SS_{N,\e}(\varrho)$, we use $\theta$ to construct the set $\L(m,\theta)$ defined in \eqv(5.prop2.4). 
To each $m$ corresponds a unique $\L(m,\theta)$. We then define  $\EE_{N,\ell,\e}(\varrho)$ as the set of all $m\in\SS_{N,\e}(\varrho)$ such that $\L(m,\theta)$ has given cardinality, 
$
|\L(m,\theta)|=\ell
$, 
$\underline{L}\leq \ell\leq \bar{L}$. 
Clearly, the sets $\EE_{N,\ell,\e}(\varrho)$ form a disjoint covering of $\SS_{N,\e}(\varrho)$, and so,
\be
\begin{split}
&\P\left(
\sup_{m\in\SS_{N,\e}(\varrho)}\left\|\frac{J_{\L(m,\theta)}}{\sqrt N}\right\|
\geq 
\sqrt{\frac{\bar{L}}{N}}(2+t)
\right)
\\
\leq &
 \sum_{\underline{L}\leq \ell\leq\bar{L}}
\P\left(
 \sup_{m\in\EE_{N,\ell,\e}(\varrho)}\left\|\frac{J_{\L(m,\theta)}}{\sqrt N}\right\|
\geq 
\sqrt{\frac{\bar{L}}{N}}(2+t)
\right).
\end{split}
\Eq(5.prop2.20)
\ee
Since $\left\|J_{\L(m,\theta)}\right\|$ only depends on $m$ through the set $\L(m,\theta)$,
\be
\textstyle
\sup_{m\in\EE_{N,\ell,\e}(\varrho)}\left\|J_{\L(m,\theta)}\right\|
\leq 
\sup_{\L\in\{1,\dots,N\}:|\L|=\ell}\left\|J_{\L}\right\|
\Eq(5.prop2.20')
\ee
where the last $\sup$ is over non ordered sets. Thus,  \eqv(5.prop2.20) is bounded above by
\bea
\sum_{\underline{L}\leq \ell\leq\bar{L}}\binom{N}{\ell}
\P\left(
\sqrt{\frac{ \ell}{N}}\left\|\frac{J_{\ell}}{\sqrt \ell}\right\|
\geq \sqrt{\frac{\bar{L}}{N}}(2+t)\right)
\leq 
\sum_{\underline{L}\leq \ell\leq\bar{L}} \binom{N}{\ell}2e^{-\frac{1}{4}\bar{L}t^2}
\Eq(5.prop2.22)
\eea
where we used Proposition \thv(5.prop5) in the last line.
We now assume that $\underline{L}\geq 0$ is arbitrary. In that case we extend the summation range in \eqv(5.prop2.22) to 
$0\leq \ell\leq\bar{L}$. Because of the symmetry of the binomial coefficient with respect to $\ell$ and  $N-\ell$, and the fact that it is strictly increasing for $\ell\leq N/2$, we handle the resulting sum differently if $\bar{L}\leq N/2$ or  $\bar{L}> N/2$.
In the first case, we use the well known bound
$
\sum_{0\leq \ell\leq \bar{L}}\binom{N}{\ell}\leq e^{N\JJ(\bar{L}/N)}
$,
valid for all $0\leq \ell\leq \bar{L}$ with $\bar{L}/N\leq 1/2$.
If on the contrary $\bar{L}>N/2$, we simply write 
$
\sum_{0\leq \ell\leq \bar{L}}\binom{N}{\ell}\leq 2^N=e^{N\JJ(1/2)}
$.
Inserting these bounds in \eqv(5.prop2.22), \eqv(5.prop2.1) is obtained by choosing $t$ in Proposition \thv(5.prop5) such that 
$\frac{1}{4}({\bar{L}}/{N})t^2=\JJ(x)+\l$, where $x=\bar{L}/N$ if $1\leq \bar{L}\leq N/2$ and $x=1/2$ else.
\end{proof}

\begin{proof}[Proof of Proposition \thv(5.prop4)]
The proposition follows from the bound
$
\|J_N-J_{\L}\|/\sqrt{N}\leq(\|J_N\|+\|J_{\L}\|)/\sqrt{N}
$,
using Proposition \thv(5.prop5) to bound $\|J_N\|/\sqrt{N}$ and proceeding as in the proof of Proposition \thv(5.prop2) to bound
 $\|J_{\L}\|/\sqrt{N}$.
\end{proof}

\subsection{Proof of Proposition \thv(5.prop3)} 
    \TH(S5.3)
 
The first step is to replace the supremum over $\SS_{N,\e}(\varrho)$ by the supremum over a discrete set, 
$\NN_{N,\e,\tilde\varepsilon}(\varrho)$, defined as follows. Given $0<\tilde\varepsilon\leq 1$, let
\be
\NN_{N,\e,\tilde\varepsilon}(\varrho)=
\left\{m\in \left(\sqrt{\varrho\tilde\varepsilon}\,\Z\cap[-1,1]\right)^N
: |q_{\text{EA}}(m)-\varrho|\leq \varrho(1-\varrho)\e+\varrho\tilde\varepsilon
\right\}.
\Eq(5.prop3.2)
\ee
For every $m$ in $\SS_{N,\e}(\varrho)$ there exists $m_0$ in $\NN_{N,\e,\tilde\varepsilon}(\varrho)$ such that $|m_i-m_{0,i}|\leq \sqrt{\varrho\tilde\varepsilon}$ for all $1\leq i\leq N$ 
(this means that $\NN_{N,\e,\tilde\varepsilon}(\varrho)$ is a $\varrho\tilde\varepsilon$-net of $\SS_{N,\e}(\varrho)$ for the supremum norm).

The next lemma provides a bound on the size of $\NN_{N,\e,\tilde\varepsilon}(\varrho)$. 

\begin{lemma}
\TH(5.lem2)
For large enough $N$
\be
\left|\NN_{N,\e,\tilde\varepsilon}(\varrho)\right|
\leq
\left(2\pi e\right)^{N/2}\left(\frac{1+(1-\varrho)\e+2\tilde\varepsilon}{\tilde\varepsilon}\right)^{N/2}.
\Eq(5.lem2.1)
\ee
\end{lemma}
\begin{proof}
Denote by $\WW_{N}$ the lattice of side length $\sqrt{\varrho\tilde\varepsilon}$
and by $\BB_{r}=\{\|m\|^2_2\leq r\}$ the ball of radius $\sqrt r$ centered at zero.
Then $\left|\NN_{N,\e,\tilde\varepsilon}(\varrho)\right|$ is bounded above by the number of points of the lattice $\WW_{N}$ that lie in 
$[-1,1]^N\cap\BB_{\varrho N(1+(1-\varrho)\e+\tilde\varepsilon)}$. Let us surround each point of the lattice  $\WW_{N}$ by a cube of side length 
$\sqrt{\varrho\tilde\varepsilon}$. Note that the diagonal of this cube has length $qN\tilde\varepsilon$. Clearly,  $\left|\NN_{N,\e,\tilde\varepsilon}(\varrho)\right|$ 
is smaller than the number of cubes which have non empty intersection with $[-1,1]^N\cap\BB_{\varrho N(1+(1-\varrho)\e+\tilde\varepsilon)}$. 
Thus, if $\VV_{N}$ is the volume of 
$[-(1+\sqrt{\varrho\tilde\varepsilon}),1+\sqrt{\varrho\tilde\varepsilon}]^N\cap\BB_{\varrho N(1+(1-\varrho)\e+2\tilde\varepsilon)}$, 
\be
\left|\NN_{N,\e,\tilde\varepsilon}(\varrho)\right|\leq \VV_{N}\left(\sqrt{\varrho\tilde\varepsilon}\right)^{-N}.
\Eq(5.lem2.2)
\ee
It remains to estimate $\VV_{N}$. Recalling \eqv(5.prop3.0), we have
\bea
\VV_{N}
&=&\int_{-(1+\sqrt{\varrho\tilde\varepsilon})}^{1+\sqrt{\varrho\tilde\varepsilon}}dm_1
\dots
\int_{-(1+\sqrt{\varrho\tilde\varepsilon})}^{1+\sqrt{\varrho\tilde\varepsilon}}dm_N
\1_{
\left\{
\|m\|^2_2\leq 
\varrho^{+}_{\e,2\tilde\varepsilon}N
\right\}}
\Eq(5.lem2.3)
\\
&\leq &
e^{\frac{N}{2}}\prod_{i=1}^N\left(
\int_{-(1+\sqrt{\varrho\tilde\varepsilon})}^{1+\sqrt{\varrho\tilde\varepsilon}}dm_i
e^{-m_i^2/\left(2\varrho^{+}_{\e,2\tilde\varepsilon}\right)
}
\right)
\leq
\left({2\pi e}
\varrho^{+}_{\e,2\tilde\varepsilon}
\right)^{N/2}.
\Eq(5.lem2.5)
\eea
(This bound is rough but it is hard to substantially improve it.)
Inserting \eqv(5.lem2.5) in  \eqv(5.lem2.2) proves \eqv(5.lem2.1). \end{proof}

The next lemma will enable us to replace the supremum of the operator norm over $\SS_{N,\e}(\varrho)$ by its supremum over  
$\NN_{N,\e,\tilde\varepsilon}(\varrho)$.
Set
\be
\chi(\b,\theta_{\star})
\equiv
\sqrt{\left(c_0(\theta_{\star})/\theta_{\star}\right)^{3}c_0(1)}\sqrt{\frac{1-\theta_{\star}}{\theta_{\star}}}.
\Eq(5.lem4.2)
\ee

\begin{lemma}
\TH(5.lem4)
\be
\sup_{m\in\SS_{N,\e}(\varrho)}\left\|C^{\circ}_N(m)\right\|
\leq 
\sup_{m_0\in\NN_{N,\e,\tilde\varepsilon}(\varrho)}\left\|C^{\circ}_N(m_0)\right\| 
+2\sqrt{\varrho\tilde\varepsilon}\chi(\b,\theta_{\star})\sup_{m\in\SS_{N,\e}(\varrho)}\left\|\frac{J_N-J_{\L}}{\sqrt{N}}\right\|.
\Eq(5.lem4.1)
\ee
\end{lemma}

\begin{proof}[Proof of Lemma \thv(5.lem4)] 
For simplicity of notation we write 
$
D_{N,\L}\equiv(J_N-J_{\L})/\sqrt{N}
$
throughout the proof.
Recall that $C^{\circ}_N(m)$ defined in \eqv(5.prop2.7) is an $N\times N$ matrix. By \eqv(9.2.2)
\be
\sup_{m\in\SS_{N,\e}(\varrho)}\left\|C^{\circ}_N(m)\right\|
=\sup_{x:\|x\|_2=1}\sup_{m\in\SS_{N,\e}(\varrho)}\left|\left(x,C^{\circ}_N(m)x\right)\right|.
\Eq(5.prop3.5)
\ee
Given a point $x$ on the sphere $\|x\|_2=1$ in $\R^N$, let $m\in\SS_{N,\e}(\varrho)$ be such that 
\be
|(x,C^{\circ}_N(m)x)|=\sup_{m\in\SS_{N,\e}(\varrho)}|(x,C^{\circ}_N(m)x)|
\ee
and pick a point $m_0\in\NN_{N,\e,\tilde\varepsilon}(\varrho)$ 
such that $\sup_{1\leq i\leq N}|m_i-m_{0,i}|\leq \sqrt{\varrho\tilde\varepsilon}$. 
For $1\leq 1\leq N$, set 
\be
u_i(m)\equiv x_i\sqrt{{\tilde a}_i(m)}
\Eq(5.prop3.6)
\ee
and denote by $u(m)$ the vector $u(m)=(u_1(m),\dots,u_N(m))$.
By \eqv(5.prop2.34), \eqv(5.prop2.7) and \eqv(5.prop2.10)
\bea
(x,C^{\circ}_N(m)x) \hspace{-6pt}&=&\hspace{-6pt} \left(u(m), D_{N,\L} u(m)\right)
\Eq(5.prop3.7)
\\
\hspace{-6pt}&=&\hspace{-6pt}\left(u(m_0), D_{N,\L} u(m_0)\right)
+Q_1(m,m_0)+Q_2(m,m_0)
\quad
\Eq(5.prop3.8)
\\
\hspace{-6pt}&=&\hspace{-6pt}(x,C^{\circ}_N(m_0)x) +Q_1(m,m_0)+Q_2(m,m_0)
\Eq(5.prop3.8')
\eea
where
\be
\begin{split}
Q_1(m,m_0)&= \left(u(m),D_{N,\L}(u(m)-u(m_0))\right),
\\
Q_2(m,m_0)&=\left((u(m)-u(m_0)),D_{N,\L}u(m_0)\right).
\end{split}
\Eq(5.prop3.9)
\ee
We begin by bounding $Q_2(m,m_0)$. By Cauchy-Schwarz inequality
\be
Q_2(m,m_0)
\leq\|u(m)-u(m_0)\|_2\sqrt{\left(u(m_0) D_{N,\L}^2u(m_0)\right)}.
\Eq(5.prop3.10)
\ee
Consider the first factor in \eqv(5.prop3.10). Using that $\|x\|_2=1$, we have
\bea
\|u(m)-u(m_0)\|_2^2
\hspace{-6pt}&\leq &\hspace{-6pt}\sum_{i=1}^Nx_i^2\left(\sqrt{{\tilde a}_i(m)}-\sqrt{{\tilde a}_i(m_0)}\right)^2
\Eq(5.prop3.11)
\\
\hspace{-6pt}&\leq &\hspace{-6pt}\sup_{1\leq i\leq N}\left(\sqrt{{\tilde a}_i(m)}-\sqrt{{\tilde a}_i(m_0)}\right)^2.
\Eq(5.prop3.12)
\eea
The reason for the definition \eqv(5.prop2.35) of ${\tilde a}_i(m)$ now becomes clear. Setting $g(m_i)=\sqrt{{\tilde a}_i(m)}$, 
$g(m_i)-g(m_{0,i})=0$ for all $i\in\L^c_{\star}(m,\theta_{\star})\cap\L^c_{\star}(m_0,\theta_{\star})$. In all other cases, $g(m_i)$ has bounded derivative on $[-1,1]$ and it follows from the mean value theorem that
\be
g(m_i)-g(m_{0,i})
\leq \sup_{\hat m_i\in[0,1-\theta_{\star}]}g'(\hat m_i)|m_i-m_{0,i}|
\leq \sqrt{\varrho\tilde\varepsilon/c_0(1)}
\chi(\b,\theta_{\star})
\ee
for $\chi(\b,\theta_{\star})$ as in \eqv(5.lem4.2). Combined with \eqv(5.prop3.12), this yields
\be
\|u(m)-u(m_0)\|_2
\leq 
 \sqrt{
 \varrho\tilde\varepsilon/c_0(1)
 }
 \chi(\b,\theta_{\star}).
\Eq(5.prop3.32)
\ee

It remains to bound the last factor in \eqv(5.prop3.10). Introducing the Rayleigh quotient $Q(v)$,
\be
Q(v)=(v,v)^{-1}\left(v  D_{N,\L}^2 v\right),
\ee
we have
\be
\left(u(m_0)  D_{N,\L}^2 u(m_0)\right)=\|u(m_0)\|_2^2 Q(u(m_0)).
\Eq(5.prop3.18)
\ee
Then, by \eqv(9.2.2)
\be
Q(u(m_0))
\leq 
\sup_{v:\|v\|_2=1}\left(v  D_{N,\L}^2 v\right)
=
r\left(D_{N,\L}^2 \right)
=
\left\|D_{N,\L}\right\|^2 .
\Eq(5.prop3.20)
\ee
Proceeding as in \eqv(5.prop3.11)-\eqv(5.prop3.12) 
and using the rough bound 
$
{\tilde a}_i(m_0)\leq c_0(1)
$
we have
\be
\|u(m_0)\|_2^2\leq \max_{1\leq i\leq N}{\tilde a}_i(m_0)\leq c_0(1),
\Eq(5.prop3.19)
\ee
and so, plugging \eqv(5.prop3.20) and \eqv(5.prop3.19) in \eqv(5.prop3.18),
\be
\left(u(m_0)  D_{N,\L}^2 u(m_0)\right)
\leq 
c_0(1)\left\|D_{N,\L}\right\|^2.
\Eq(5.prop3.21)
\ee
Finally, inserting \eqv(5.prop3.32) and \eqv(5.prop3.21) in \eqv(5.prop3.10), we obtain
\be
Q_2(m,m_0)
\leq
\sqrt{\varrho\tilde\varepsilon}\chi(\b,\theta_{\star})\left\|D_{N,\L}\right\|
\leq
\sqrt{\varrho\tilde\varepsilon}\chi(\b,\theta_{\star})\sup_{m\in\SS_{N,\e}(\varrho)}\left\|D_{N,\L}\right\|.
\Eq(5.prop3.22)
\ee
Bounding the term $Q_1(m,m_0)$ in \eqv(5.prop3.9) in the same way, it follows from \eqv(5.prop3.8') that
\be
\left|(x,C^{\circ}_N(m)x)\right| 
\leq
\left|(x,C^{\circ}_N(m_0)x)\right| 
+2\sqrt{\varrho\tilde\varepsilon}\chi(\b,\theta_{\star})\sup_{m\in\SS_{N,\e}(\varrho)}\left\|D_{N,\L}\right\|.
\Eq(5.prop3.23)
\ee
From this and our choices of $m$ and $m_0$ (see the paragraph above \eqv(5.prop3.6)), we get
\be
\hspace{-5pt}
\sup_{m\in\SS_{N,\e}(\varrho)}\left|(x,C^{\circ}_N(m)x)\right|
\leq
\sup_{m_0\in\NN_{N,\e,\tilde\varepsilon}(\varrho)}\left|(x,C^{\circ}_N(m_0)x)\right| 
+2\sqrt{\varrho\tilde\varepsilon}\chi(\b,\theta_{\star})\sup_{m\in\SS_{N,\e}(\varrho)}\left\|D_{N,\L}\right\|.
\nonumber
\Eq(5.prop3.24)
\ee
Since this bound holds for any given point $x$ on the sphere $\|x\|_2=1$, taking the supremum over $x$ on both sides and recalling the identity \eqv(5.prop3.5), we arrive at
\be
\sup_{m\in\SS_{N,\e}(\varrho)}\left\|C^{\circ}_N(m)\right\|
\leq 
\sup_{m_0\in\NN_{N,\e,\tilde\varepsilon}(\varrho)}\left\|C^{\circ}_N(m_0)\right\| 
+2\sqrt{\varrho\tilde\varepsilon}\chi(\b,\theta_{\star})\sup_{m\in\SS_{N,\e}(\varrho)}\left\|D_{N,\L}\right\|.
\Eq(5.prop3.25)
\ee
The proof of Lemma \thv(5.lem4) is done.
\end{proof}

By Lemma \thv(5.lem4) and Proposition \thv(5.prop4) with $\l=N^{-1/4}$, we have for all $t'>0$ and $N>80$
\be
\begin{split}
& 
\P\left( 
\sup_{m\in\SS_{N,\e}(\varrho)}\left\|C^{\circ}_N(m)\right\| 
\geq t' \right)
\\
\leq
&\sum_{m_0\in\NN_{N,\e,\tilde\varepsilon}(\varrho)}\P\left( 
\left\|C^{\circ}_N(m_0)\right\| 
+12\sqrt{\varrho\tilde\varepsilon}\chi(\b,\theta_{\star})\geq t' \right)
+e^{-\sqrt N/4}.
\end{split}
\Eq(5.prop3.27)
\ee
We are thus left with proving an upper bound on the tail probability of  $\left\|C^{\circ}_N(m_0)\right\|$. 
Set
\be
{\tilde\sigma}(m)=\sqrt{
\frac{c_0(1)}{N}\sum_{1\leq j\leq N }{\tilde a}_j(m)}, \quad
{\tilde\sigma_\star}
= \sqrt{\frac{c_0(1)c_0(\theta)}{N}
},
\quad
{\tilde\sigma_0}=\sup_{m_0\in\NN_{N,\e,\tilde\varepsilon}(\varrho)}{\tilde\sigma}(m_0).
\Eq(5.lem1.11)
\ee
\begin{lemma}
\TH(5.lem1)
For all $m\in\NN_{N,\e,\tilde\varepsilon}(\varrho)$, all $N$ and all $t>0$
\be
\P\left(
\left\|C^{\circ}_N(m)\right\| 
\geq 
2{\tilde\sigma_0}+\sqrt{N}{\tilde\sigma_\star}t+14\sqrt{N^{-1/2}{\log N}}
\right)
\leq
e^{-Nt^2/4}.
\Eq(5.lem1.1)
\ee
\end{lemma}

\begin{proof}
Pick any $m\in\NN_{N,\e,\tilde\varepsilon}(\varrho)$.
Setting 
\be
\tilde v_{ij}(m)\equiv\sqrt{\frac{{\tilde a}_i(m){\tilde a}_j(m)}{N}},
\Eq(5.lem1.2)
\ee
the lemma follows from an application of Theorem \thv(5.theo3) to the matrix $C^{\circ}_N(m)$ of entries
\be
C^{\circ}_{ij}(m)=
\begin{cases}
\tilde v_{ij}(m)J_{ij}
& \text{if}\,\,\,  (i,j)\in(\{1,\dots,N\}\times\L^c)\cup(\L^c\times\{1,\dots,N\}).\\
0 &  \text{if}\,\,\,  (i,j)\in\L\times\L,
\end{cases}
\Eq(5.lem1.3)
\ee
(i.e., we replace \eqv(5.theo3.8) with \eqv(5.lem1.2)-\eqv(5.lem1.3)). To bound the quantities  $\s(m)$  
and ${\sigma_\star}(m)$ defined in \eqv(5.theo3.0), recall the bound \eqv(5.prop3.19) on 
$\max_{1\leq i\leq N}{\tilde a}_i(m)$ and observe that on $\L^c$
\be
{\tilde a}_j(m)\leq 
\begin{cases}
c_0(\theta)
& \text{if}\,\,\, j\in\L^c\cap\L_{\star},
\\
c_0(\theta_{\star})
& \text{if}\,\,\, j\in\L^c\cap\L^c_{\star},
\end{cases}
\Eq(5.lem1.6)
\ee
where $\theta_{\star}< \theta$ by assumption.
Then
\bea
\s^2(m)
\leq 
\max\left\{
\frac{c_0(\theta)}{N}\sum_{1\leq j\leq N}{\tilde a}_j(m),
\frac{c_0(1)}{N}\sum_{j\in \L^c}{\tilde a}_j(m)
\right\}
\leq{\tilde\sigma}^2(m),
\Eq(5.lem1.4)
\eea
and
\be
{\sigma_\star}^2(m)
= \frac{1}{N}\max_{1\leq i\leq N}{\tilde a}_i(m)\max_{j\in\L^c}{\tilde a}_j(m)
\leq
{\tilde\sigma^2_\star}.
\Eq(5.lem1.5)
\ee
Hence, by Theorem \thv(5.theo3) with $\varepsilon=1/\sqrt{N}$,  we have that for all $t>0$
\be
\P\left(
\left\|C^{\circ}_N(m)\right\| 
\geq 
2{\tilde\sigma}(m)+\sqrt{N}{\tilde\sigma_\star}t
+14\sqrt{\sfrac{\log N}{\sqrt N}}
\right)
\leq
e^{-Nt^2/4}.
\Eq(5.lem1.10)
\ee
Given the definition of ${\tilde\sigma_0}$ in  \eqv(5.lem1.11), this implies the lemma.
\end{proof}

Combining \eqv(5.prop3.27) and   \eqv(5.lem1.1), and using Lemma \thv(5.lem2) to bound the sum over 
$\NN_{N,\e,\tilde\varepsilon}(\varrho)$ we obtain, choosing
$
t^2=2\log\left(2\pi e
{\varrho^+_{\e,2\tilde\varepsilon}}/{(\varrho\tilde\varepsilon)}
\right)+\frac{1}{4\sqrt N},
$
\be
\begin{split}
\P\Biggl( 
\sup_{m\in\SS_{N,\e}(\varrho)}
&
\left\|C^{\circ}_N(m)\right\| 
\geq 
2{\tilde\sigma_0}
+\sqrt{N}{\tilde\sigma_\star}\sqrt{2\log\left(2\pi e
{\varrho^+_{\e,2\tilde\varepsilon}}/{(\varrho\tilde\varepsilon)}
\right)}
\\
&
\quad\quad\quad+12\sqrt{\varrho\tilde\varepsilon}\chi(\b,\theta_{\star})
+15\sqrt{\sfrac{\log N}{\sqrt N}}
\Biggr)
\leq \bigl(1+\OO\bigl(\sfrac{1}{\sqrt{N}}\bigr)\bigr)e^{-\sqrt N/4}.
\end{split}
\Eq(5.lem1.16)
\ee

All that remains is to bound ${\tilde\sigma_0}$. Recalling the notation \eqv(5.prop3.0), we have:

\begin{lemma}
\TH(5.lem7)
$
{\tilde\sigma_0}\leq \sqrt{c_0(1)}\sqrt{c_0(1-\varrho^{-}_{\e,\tilde\varepsilon})+c_0(\theta_{\star})}.
$
\end{lemma}
We prove successively Lemma \thv(5.lem6) and Lemma \thv(5.lem7).

\begin{proof}[Proof of Lemma \thv(5.lem6)]
Note that the function $f(m)=\frac{1}{N}\sum_{j}a_j(m)$  is strictly concave on
$[-1,1]^N$.
Writing
\be
\textstyle
\sup_{m\in\SS_{N,\e}(\varrho)}f(m)
=\sup_{\rho': |\rho'-\varrho|\leq \varrho(1-\varrho)\e}\sup_{m : q_{\text{EA}}(m)=\rho'}f(m)
\ee
and using Lagrange multipliers, one readily gets that the last constrained supremum is attained at 
points such that $m^2_i=\rho'$, $1\leq i\leq N$, yielding the upper bound of \eqv(5.lem6.1).
Similarly, one proves that the constrained infimum is attained at points of the form
$m_i=1$ for all $i\in I$ and $m_i=0$ else, where $I\subset\{1,\dots,N\}$ is any subset of cardinality $|I|=\rho'N$. This yields the lower bound.
Since each $a_i(m)$ is maximized at $m_i=0$, the upper bound of \eqv(5.lem6.2) is attained at any point $m$ that contains at least one zero coordinate. The lower bound follows from the choice  $m^2_i=\varrho(1+\e)$, $1\leq i\leq N$. 
\end{proof}

\begin{proof}[Proof of Lemma \thv(5.lem7)] By  \eqv(5.prop2.35)
\be
\sup_{m_0\in\NN_{N,\e,\tilde\varepsilon}(\varrho)}\frac{1}{N}\sum_{1\leq j\leq N }{\tilde a}_j(m)
\leq
\sup_{m_0\in\NN_{N,\e,\tilde\varepsilon}(\varrho)}\frac{1}{N}\sum_{1\leq j\leq N}a_j(m) +c_0(\theta_{\star}).
\Eq(5.lem7.5)
\ee
The last sum is bounded above by the upper bound of  \eqv(5.lem6.1) with $\varrho^{-}_{\e}$ replaced by $\varrho^{-}_{\e,\tilde\varepsilon}$ (see  \eqv(5.prop3.0)). The lemma then follows from the definition \eqv(5.lem1.11) of ${\tilde\sigma_0}$.
\end{proof}

Proposition \thv(5.prop3) now follows from \eqv(5.lem1.16), Lemma \thv(5.lem7), the definitions \eqv(5.lem1.11) and \eqv(5.lem4.2) of ${\tilde\sigma_0}$, ${\sigma_\star}$ and $\chi(\b,\theta_{\star})$, and the bound $c_0(y)\leq y$.

\subsection{Conclusion of the proof of Proposition \thv(5.prop1) and proof of Theorem \thv(5.theo4)} 
    \TH(S5.4)

\begin{proof}[Proof of Proposition \thv(5.prop1)]

We prove the two items of the proposition separately.

\smallskip
\paragraph{Proof of item (i).} 
Since $\l_{max}(C_N(m))\leq \left\|C_N(m)\right\|$, it suffices to prove  \eqv(5.prop1.1) with $\left\|C_N(m)\right\|$ substituted for  $\l_{max}(C_N(m))$.

First we have to specify the parameters $\theta$, $\theta_{\star}$ and $\tilde\varepsilon$  (see \eqv(5.prop2.4), \eqv(5.prop2.5) and \eqv(5.prop3.2)).  A natural idea is to choose $\theta$ and $\theta_{\star}$ such that on the sets $\L(m,\theta)$ and $\L^c_{\star}(m,\theta_{\star})$, the coefficients $a_i(m)$ are respectively larger and smaller than the average $\frac{1}{N}\sum_{j}a_j(m)$. In view of \eqv(5.lem6.1), this prompts us to choose 
\be
\theta=(1-\varrho^{-}_{\e})^{\alpha},\quad
\theta_{\star}=(1-\varrho^{-}_{\e})^{\kappa},
\Eq(5.prop1.3)
\ee
where $0<\alpha<1$ and $\kappa\geq 1$ are constants to be chosen. Eq.~\eqv(5.lem1.19) then leads us to take
\be
\tilde\varepsilon=(1-\varrho^{-}_{\e})^{\tilde\kappa}
\Eq(5.prop1.4)
\ee
for some $\tilde\kappa>\kappa$. Equipped with these choices, we now use Propositions \thv(5.prop2), \thv(5.prop3) and \thv(5.prop5) to bound each of the three terms appearing on the right-hand side of \eqv(5.prop2.16). The first of these terms is treated using Proposition \thv(5.prop2). By our choice of $\theta$, $\bar{L}$ in \eqv(5.prop2.17) is bounded by
\be
\bar{L}\leq \frac{1-\varrho^{-}_{\e}}{\theta}N= (1-\varrho^{-}_{\e})^{1-\alpha}N.
\Eq(5.prop1.5)
\ee
We want to guarantee that $\bar{L}/N\leq 1/2$. For this it suffices to assume that
\be
\varrho^{-}_{\e}\geq 1-\left(\sfrac{1}{2}\right)^{\frac{1}{1-\a}}.
\Eq(5.prop1.6)
\ee
It then follows from \eqv(5.prop2.1) with $\l=N^{-1/4}$ and the
classical bound $\JJ(x)\leq 2\ln 2\sqrt{x(1-x)}$, $0\leq x\leq 1/2$, that
\be
\P\left(
\sup_{m\in\SS_{N,\e}(\varrho)}\left\|\frac{J_{\L}}{\sqrt N}\right\|
\geq 
2\sqrt{(1-\varrho^{-}_{\e})^{1-\alpha}\left(1+2\ln 2\right)}+2N^{-1/4}
\right)
\leq 
2e^{-\sqrt{N}}.
\Eq(5.prop1.7)
\ee
The second term on the right-hand side of \eqv(5.prop2.16) is treated using Proposition \thv(5.prop3). It follows from \eqv(5.prop1.3) that the function \eqv(5.lem1.19) is bounded above by
\be
\begin{split}
&
r(\varrho,\e,\tilde\varepsilon, \theta, \theta_{\star}, \b) 
\\
& \leq 
2\sqrt{(1-\varrho^{-}_{\e})\left[ 1+q(1-\varrho^{-}_{\e})^{\tilde\kappa-1}+(1-\varrho^{-}_{\e})^{\kappa-1}\right]}
\\
&+\sqrt{2(1-\varrho^{-}_{\e})^{\alpha}\left\{
{\tilde\kappa}\left|\ln (1-\varrho^{-}_{\e})\right|+
\log\left(2\pi e(1+(1-\varrho)\e+2\tilde\varepsilon)\right)\right\}}
\\
 & +12\sqrt{\varrho(1-\varrho^{-}_{\e})^{\tilde\kappa-\kappa}}
+15\sqrt{\sfrac{\log N}{\sqrt N}}.
\end{split}
\Eq(5.prop1.8)
\ee
Finally, to deal with the third and last term we use that, by Proposition \thv(5.prop5) with $t=N^{-1/4}$,
\be
\P\left(
3\sqrt{\theta_{\star}}
\left\|\frac{J_N}{\sqrt N}\right\|
\geq 
3\sqrt{(1-\varrho^{-}_{\e})^{\kappa}}
(2+N^{-1/4})
\right)
\leq 
2e^{-\sqrt{N}}.
\Eq(5.prop1.9)
\ee

Collecting these results yields a bound on the tail probability of  $\sup_{m\in \SS_{N,\e}(\varrho)}\left\|C_N(m)\right\|$
that still depends on $\a$, $\kappa$ and $\tilde\kappa$. It is clear that to minimise the contribution of terms containing $\a$, one should choose $\alpha=1/2$. How to optimise the choice of $\kappa$ and $\tilde\kappa$ is less obvious. We take $\kappa=2$ and $\tilde\kappa=4$.

It remains to deal with the supremum over $\e$ in \eqv(5.prop1.1). 
For this we note that $\SS_{N,\e}(\varrho)\subseteq\SS_{N,1}(\varrho)$ for all $\e\in[0,1]$ and use that
$
1-\varrho^{-}_{\e}=(1-\varrho)(1+\varrho\e)\leq 1-\varrho^2 \leq 2(1-\varrho)
$  
to make our bounds uniform in $\e$. The last inequality implies in particular that if $\varrho\geq \sqrt{3/4}$ then \eqv(5.prop1.6) is verified for all $\e\leq 1$ and $\a=1/2$. Item (i) of Proposition \thv(5.prop1) now readily follows. 

\paragraph{Proof of item (ii).}
Returning to the Hessian \eqv(5.1)-\eqv(5.2) (and remembering  the matrix notation from the paragraph above \eqv(5.1)),
it follows from the Courant-Fisher minimax principle that the largest eigenvalue $\l_{max}(\HH(m))$ of $\HH(m)$ satisfies
\be
\l_{max}(\HH_N(m))\leq \b\l_{max}\left(\frac{J_N}{\sqrt N}\right)-\l_{min}(B_N(m)).
\Eq(5.prop1.20)
\ee
On the one hand, for all $m\in [-1,1]^N$, $B_N(m)$ is strictly positive definite and obeys
\be
\l_{min}(B_N(m))\geq \b^2(1-q)+ 1.
\Eq(5.prop1.21)
\ee
On the other hand, by Proposition \thv(5.prop5), 
$
\P\left(\l_{max}\bigl(\frac{J_N}{\sqrt N}\bigr)\geq 2+\frac{1}{N^{1/4}}\right) \leq 2e^{-\sqrt{N}/4}
$.
Combining \eqv(5.prop1.20) and \eqv(5.prop1.21) proves \eqv(5.prop1.22).

The proof of Proposition \thv(5.prop1) is now complete. \end{proof}

\begin{proof}[Proof of Theorem \thv(5.theo4)]
\hfill\break

\vspace{-10pt}
\paragraph{Proof of item (i).} 
Since the set $\BB^c_{N,\e}(\varrho)$ in \eqv(5.theo4.1) is increasing with $\e$, we may assume that $\e>0$.
The strategy of the proof is  to cover $\BB^c_{N,\e}(\varrho)$ with a collection of spherical shells and proceed as in the proof of
item  (i) of Proposition \thv(5.prop1) to deal with each of them.
Set 
$
K=\frac{1-\varrho^{-}_{\e}}{2\varrho(1-\varrho)\e}
$,
$
\rho_{k} = \varrho^{-}_{\e}+(2k+1)\varrho(1-\varrho)\e
$
and for 
$
0\leq k\leq K-1
$
\be
\SS_{N,k,\e}
=\left\{m\in [-1,1]^N :  |q_{\text{EA}}(m)-\rho_{k}|\leq \varrho(1-\varrho)\e\right\}.
\Eq(5.theo4.5)
\ee
Then
\be
\BB^c_{N,\e}(\varrho)=\bigcup_{k=0}^{K-1}\SS_{N,k,\e}.
\Eq(5.theo4.6)
\ee
We now claim that for each shell $\SS_{N,k,\e}$, $0\leq k\leq K-1$, under the assumptions and with the notation of item (i) of Proposition \thv(5.prop1), for all $0< \e\leq 1$ and all $N$ large enough
\be
\P\left(\sup_{m\in \SS_{N,k,\e}}
\l_{max}(C_N(m))\geq
f_1(\varrho)+16\sqrt{\sfrac{\log N}{\sqrt N}}
\right)
\leq
6e^{-\sqrt{N}}.
\Eq(5.theo4.7)
\ee
Note that for $k=0$, $\SS_{N,0,\e}=\SS_{N,\e}(\varrho)$ so that \eqv(5.theo4.7) follows from \eqv(5.prop1.1).
For $k>1$ the proof of  \eqv(5.theo4.7) is a simple rerun of the proof of the case $k=0$, replacing $\varrho$, $\varrho^{\pm}_{\e}$ and $\varrho^{\pm}_{\e,\tilde\varepsilon}$ (see \eqv(5.lem6.0) and \eqv(5.prop3.0)) where needed with $\rho_{k}$, $\rho^{\pm}_{k,\e}\equiv\rho_{k}\pm \varrho(1-\varrho)\e$ and $\rho^{\pm}_{k,\e,\tilde\varepsilon}\equiv\rho_{k}\pm [\varrho(1-\varrho)\e+\rho_{k}\tilde\varepsilon]$. In particular, Proposition \thv(5.prop3) is modified as follows: replacing $\SS_{N,\e}(\varrho)$ with 
$\SS_{N,k,\e}$ in \eqv(5.prop3.1),  the quantity $r(\varrho,\e,\tilde\varepsilon, \theta, \theta_{\star}, \b)$ must be replaced with
\be
\begin{split}
\hspace{-4pt}r_k(\rho_{k},\e,\tilde\varepsilon, \theta, \theta_{\star}, \b) & \equiv 
2\sqrt{1-\rho^{-}_{k,\e,\tilde\varepsilon} + \theta_{\star}}
+\sqrt{2\theta\log\left(2\pi e\frac{1+(1-\rho_{k})\e+2\tilde\varepsilon}{\tilde\varepsilon}\right)}
\\
 & +12\sqrt{\rho_{k}\tilde\varepsilon}\sqrt{\frac{1-\theta_{\star}}{\theta_{\star}}}
+15\sqrt{\sfrac{\log N}{\sqrt N}}.
\end{split}
\Eq(5.theo4.8)
\ee
We then proceeds as in the proof of item (i)  of Proposition \thv(5.prop1), replacing 
$\varrho^{-}_{\e}$ with $\rho^{-}_{k,\e}$ in \eqv(5.prop1.3)-\eqv(5.prop1.4), 
$r(\varrho,\e,\tilde\varepsilon, \theta, \theta_{\star}, \b)$ with $r_k(\rho_{k},\e,\tilde\varepsilon, \theta, \theta_{\star}, \b)$
in \eqv(5.prop1.8), and bounding quantities of the form $\rho_{k}\tilde\varepsilon$ by $\tilde\varepsilon$. Doing this, we obtain \eqv(5.theo4.7) with $f_1(\varrho)$ replaced by $f_1(\rho_{k})$. The proof of \eqv(5.theo4.7) is now completed using the following two facts: (1) the central radius $\rho_{k}$ of the spherical shell $\SS_{N,k,\e}$ increases from $\varrho$ to $1-\varrho(1-\varrho)\e$ as $k$ increases from $0$ to $K-1$, and (2) $f_1(\rho)$ is a decreasing function of $\rho$ on $[0,1]$.

By \eqv(5.theo4.6) and \eqv(5.theo4.7), we get that under the assumptions and with the notation of item (i) of Proposition \thv(5.prop1), for all $0< \e\leq 1$ and all $N$ large enough
\be
\P\left(\sup_{m\in \BB^c_{N,\e}(\varrho)}
\l_{max}(C_N(m))\geq
f_1(\varrho)+16\sqrt{\sfrac{\log N}{\sqrt N}}
\right)
\leq
6Ke^{-\sqrt{N}}.
\Eq(5.theo4.9)
\ee
Thus, for all $\varrho\geq \sqrt{3/4}$ and $(\b,h)$, $h>0$,  such that $\b f_1(\varrho)<1$, it follows from \eqv(5.theo4.9)  
Borel-Cantelli lemma that
\be
\P\left(\bigcup_{N_0}\bigcap_{N\geq N_0}
\left\{
\sup_{\e\in[0,1]}\sup_{m\in \BB^c_{N,\e}(\varrho)}\b \l_{max}(C_N(m))<1
\right\}
\right)=1.
\Eq(5.theo4.10)
\ee
Now, the condition $\b f_1(\varrho)<1$ is nothing but \eqv(1.theo0.1),  and so, by \eqv(5.5), \eqv(5.theo4.10) proves item (i) of Theorem \thv(5.theo4).

\smallskip
\paragraph{Proof of item (ii).} If $(\b,h)$, $h>0$, are such that \eqv(1.theo0.2) is satisfied,  then the function $f_2(\b,q)$ of \eqv(5.prop1.11bis) obeys $f_2(\b,q)<0$. Item (ii) of Theorem \thv(5.theo4) in this case follows from \eqv(5.prop1.22) of Proposition \thv(5.prop1) and Borel-Cantelli lemma. 

The proof of Theorem \thv(5.theo4) is complete.
\end{proof}

\subsection{Proof of Theorem \thv(main.theo3)} 
\TH(S5.6)
Recall from \eqv(2.cor1.2bis) that $F_{N,\b,h}^{HT}(m)$ can be written as
\be
F_{N,\b,h}^{HT}(m)=\frac{1}{N}\left\{ \Psi_{N,\b,h{\bf{1}}}(x) + \frac{\b^2N}{4}\left(1-q^2\right)\right\}.
\Eq(main.theo3-proof.1)
\ee
We first prove Theorem \thv(main.theo3) for $(\b,h)$ in the intersection of $\DD^{(2)}_{\varrho}$ and the AT-region. 
Let $\sqrt{3/4}\leq\varrho\leq q\leq 1$ be given. Under the assumptions and with the notation of Theorem \thv(5.theo4), (i), 
there exists a subset $ \O_1(\b,h)\subset\O$ of full measure such that on  $ \O_1(\b,h)$,  for all  $\e\in[0,1]$ and all but a finite number of indices $N$, the Hessian of $F_{N,\b,h}^{HT}$ at $m$ is strictly negative definite for all $m\in\BB^c_{N,\e}(\varrho)$. We claim that this implies the following lemma.

\begin{lemma}
   \TH(S5.6.lem1) 
On  $\O_1(\b,h)$, for all large enough $N$, $F_{N,\b,h}^{HT}$ has at most  one critical point in $\BB^c_{N,\e}(\varrho)$, which  must be a maximum.
\end{lemma}

\begin{proof}
To prove this, we first establish that given $0<\e\leq1$  and $0<\varrho\leq 1$, for all sufficiently large $N$, $\BB^c_{N,\e}(\varrho)$ is a closed, bounded and path-connected subset of $\R^N$. Obviously, as the intersection of two closed and bounded sets, $\BB^c_{N,\e}(\varrho)$ is closed and bounded. This leaves us to prove that it is path-connected, namely, that given any two points $m,m'\in\BB^c_{N,\e}(\varrho)$, there exists a continuous function $\g_{mm'}$ from $[0,1]$ into $\BB^c_{N,\e}(\varrho)$ with endpoints $\g_{mm'}(0)=m$ and $\g_{mm'}(1)=m'$.
We do this in three steps.

\noindent\emph{Step 1:} Denote by $V_N=\{-1,1\}^N\ni v=(v_1,\dots,v_N)$ the set of vertices of the hypercube $[-1,1]^N$.  
Under the above assumptions on $\e$ and $\varrho$, $\varrho^{-}_{\e}<1$ (see \eqv(5.lem6.0)), and so,
$V_N\subset \BB^c_{N,\e}(\varrho)$. An edge of $[-1,1]^N$ is a path $\g_{vv'}$ connecting two vertices $v,v'\in V_N$ that differ in exactly one coordinate, say the $i{\text{th}}$ coordinate, described by the function 
$\g_{vv'}: [0,1]\rightarrow [-1,1]^N$, $s\mapsto \g_{vv'}(s)=(v_1,\dots,v_{i-1},2s-1,v_{i+1}\dots v_{N})$. That is, all coordinates are kept fixed except the $i{\text{th}}$, which varies linearly from $-1$ to $+1$.
Along this edge, 
\be
\textstyle
q_{\text{EA}}(\g_{vv'}(s))=\big((2s-1)^2+N-1\bigr)N^{-1}\geq 1-N^{-1}.
\Eq(main.theo3-proof.2)
\ee
Thus $\g_{vv'}\subset \BB^c_{N,\e}(\varrho)$ for all $N$ such that  
$
\varrho^{-}_{\e}\leq 1-N^{-1}
$.
Let an \emph{edge path} $\g_{vv'}$ be a path connecting two given vertices $v,v'\in V_N$ through a sequence of adjacent edges (i.e., any two consecutive edges of the path share a common vertex). Since for all sufficiently large $N$ every edge of the path maps $[0,1]$ into $\BB^c_{N,\e}(\varrho)$, so does  $\g_{vv'}$ itself.

\smallskip
\noindent\emph{Step 2:} Given any $m\in\BB^c_{N,\e}(\varrho)$, let $v\in V_N$ be the vertex of coordinates $v_i=1$ if $m_i>0$, $v_i=-1$ if $m_i<0$ and $v_i=1$ otherwise, $1\leq i\leq N$. Note that $v$ minimises the Euclidean distance from $m$ to $V_N$,
Without loss of generality we can assume that $m_i\geq 0$, so that $v_i=1$ for all $1\leq i\leq N$. Set $\CC_N(m)=\times_{i=1}^N[m_{i},1]$  and let $\g_{mv}: [0,1]\rightarrow \CC_N(m)$ be any path confined to $\CC_N(m)$ that connects $m$ and $v$. 
Clearly, $\CC_N(m)\subset\BB^c_{N,\e}(\varrho)$ since for each $m'\in\CC_N(m)$ 
\be
\textstyle
q_{\text{EA}}(m')=N^{-1}\sum_{i=1}^N (m'_i)^2\geq N^{-1}\sum_{i=1}^N (m_i)^2\geq \varrho^{-}_{\e}.
\Eq(main.theo3-proof.3)
\ee

\smallskip
\noindent\emph{Step 3:} 
Now consider any two points $m,m'\in\BB^c_{N,\e}(\varrho)$.  Let $v$ and $v'$ be any vertices of $V_N$ that minimise the Euclidean distance of $m$ and $m'$ to $V_N$, respectively. Consider the path $\g_{m,m'}=\g_{m,v}\cup\g_{v,v'}\cup\g_{v',m'}$ where  $\g_{v,v'}$ is an edge path,  and $\g_{mv}: [0,1]\rightarrow \CC_N(m)$ and  $\g_{v',m'}: [0,1]\rightarrow \CC_N(m')$ are arbitrary paths confined to $\CC_N(m)$ and $\CC_N(m')$, respectively. By steps 1 and 2 above, $\g_{m,m'}\subset\BB^c_{N,\e}(\varrho)$ for all large enough $N$. This proves our claim that $\BB^c_{N,\e}(\varrho)$ is path-connected for all sufficiently large $N$.

Now let us assume that $N$ is large enough for $\BB^c_{N,\e}(\varrho)$ to be path-connected.  This implies that on the set $\O_1(\b,h)$ (see the paragraph below \eqv(main.theo3-proof.1)), $F_{N,\b,h}^{HT}$ has at most one maximum in $\BB^c_{N,\e}(\varrho)$.  Indeed, if there are two distinct local maxima, then along any path in $\BB^c_{N,\e}(\varrho)$ connecting these two points there must exist a point $m''$ and a vector $v$ such that the second order directional derivative at $m''$ along $v$, $(v,\HH_N(m'')v)$, is greater than or equal to zero, which  is a contradiction on $ \O_1(\b,h)$.
\end{proof}

Recall at this point that, by assumption, $(\b,h)$ lies in the AT-region. We know from Section \thv(S3) that in the AT-region the solution $m^{(k)}$ of Bolthausen's iterative scheme  \eqv(1.13)-\eqv(1.14) provides an approximate solution to the critical point equation for $F_{N,\b,h}^{HT}$. To use this result we proceed as in Section \thv(S2) and compare the function  $F_{N,\b,h}^{HT}$ to the modified function $F_{N,\b,\bf{\bar h}}^{HT}$ whose magnetic field,  $\bf{\bar h}$, is chosen as in \eqv(2.4) with $\bar x=m^{(k)}$. This choice of ${\bf{\bar h}}$ ensures that  $F_{N,\b,\bf{\bar h}}^{HT}$ has a critical point at $m^{(k)}$, i.e., 
$\nabla F_{N,\b,\bf{\bar h}}^{HT}\bigl(m^{(k)}\bigr)=0$. Since the modified and original functions differ only by a linear term, they have the same Hessian. This critical point is therefore unique and a maximum on $\O_1(\b,h)$ for all $N$ large enough.

We also know from Section \thv(S3) that in the AT-region, if $0<\e\leq 1$ and $\varrho\leq q$, then
\be
m^{(k)}\in\SS_{N,\e}(q)\subset \BB^c_{N,\e}(\varrho)
\Eq(main.theo3-proof.4)
\ee
for all $N$ large enough, where $\SS_{N,\e}(q)$ is the spherical shell  \eqv(5.prop2.15). More precisely, we saw in the proof of Theorem \thv(3.theo2) (see \eqv(3.cor3.2)) that there exists a subset $\O'(\b,h)\subset\O$ with $\P\left(\O'(\b,h)\right)=1$ such that on $\O'(\b,h)$, for all $k\geq 1$,
$
\lim_{N\rightarrow\infty}\bigl\|m^{(k)}\bigr\|^2_{2,N}=q.
$
This implies that on $\O'(\b,h)$,  for all $\varrho\leq q$ and  all $k\geq 1$,
$
|q_{\text{EA}}(m^{(k)})-q| \leq q(1-q)\e_N
$
for some $\e_N$ (possibly depending on $k, \b, h$) with the property that $\e_N\downarrow 0$ as $N\uparrow\infty$.
Without loss of generality, we can assume that $N$ is large enough so that $\e_N<\e$.

Next, reasoning as in \eqv(2.prop2.4'), we have that for all $m\in [-1,1]^N$ 
\bea
\left|
F_{N,\b,h}^{HT}(m)-F_{N,\b,\bf{\bar h}}^{HT}(m)
\right|
\leq 
\hspace{-6pt}&&\hspace{-6pt}
\textstyle
\sqrt{\left\|m\right\|_{2,N}^2} 
\sqrt{\left\|\nabla \Psi_{N,\b,h{\bf{1}}}\left(m^{(k)}\right)\right\|_{2,N}^2}.
\Eq(main.theo3-proof.5)
\eea
It then follows from \eqv(4.lem4.5) and the bound $\left\|m\right\|_{2,N}\leq 1$ that
\be
\lim_{k\rightarrow\infty}\lim_{N\rightarrow\infty}
\sup_{m\in [-1,1]^N}
\left|
F_{N,\b,h}^{HT}(m)-F_{N,\b,\bf{\bar h}}^{HT}(m)
\right|
= 0  
\quad \P\textstyle{-a.s.},
\Eq(main.theo3-proof.6)
\ee 
that is, $F_{N,\b,h}^{HT}$ is uniformly well approximated by $F_{N,\b,\bf{\bar h}}^{HT}$  asymptotically, $\P$-a.s.. 
Thus, 
\be
\lim_{N\rightarrow\infty}\sup_{m\in\BB^c_{N,\e}(\varrho)} F_{N,\b,h}^{HT}(m)
=
\lim_{k\rightarrow\infty}\lim_{N\rightarrow\infty}F_{N,\b,h}^{HT}\left(m^{(k)}\right)
\quad \P\textstyle{-a.s.}.
\Eq(main.theo3-proof.7)
\ee
By \eqv(1.theo1.1) of Theorem \thv(main.theo4), \eqv(main.theo3-proof.7) implies
\be
\lim_{N\rightarrow\infty}\sup_{m\in\BB^c_{N,\e}(\varrho)}F_{N,\b,h}^{HT}(m)= SK(\b,h)\quad \P\textstyle{-a.s.}.
\Eq(1.theo0'.10)
\ee
By \eqv(main.theo3-proof.4), $\BB^c_{N,\e}(\varrho)$ can be replaced by $\SS_{N,\e}(q)$ in the above.
So far, $0<\e\leq 1$ is arbitrary. Passing to the limit  $\e\rightarrow 0$ in  \eqv(1.theo0'.10), we get
\be
\lim_{\e\rightarrow 0}\lim_{N\rightarrow\infty}\sup_{m\in \sup_{m\in\BB^c_{N,\e}(\varrho)}} F_{N,\b,h}^{HT}\left(m\right)
= SK(\b,h) \quad \P\textstyle{-a.s.}
\Eq(main.theo3-proof.8)
\ee
This concludes the proof of Theorem \thv(main.theo3), (ii).

The case where  $(\b,h)$ is in $\DD^{(1)}$ is simpler. Indeed, under the assumptions and with the notation of Theorem \thv(5.theo4), (ii),  there exists a subset $\O_2(\b,h)\subset\O$ of full measure such that on  $\O_2(\b,h)$,  for all but a finite number of indices $N$, the Hessian of $F_{N,\b,h}^{HT}$ at $m$ is strictly negative definite in the entire hypercube  $[-1,1]^N$. Thus, on $\O_2(\b,h)$, for all sufficiently large $N$, the function $F_{N,\b,h}^{HT}$ is strictly concave on the convex domain $[-1,1]^N$. 
From here on, the proof is a repeat of the proof of Theorem \thv(main.theo3), (ii). The proof of Theorem  \thv(main.theo3) is now complete.
\endproof


\section{
Proof of Theorem \thv(main.theo5)
}
\TH(S6)

Recall that $\DD^{(3)}$ is defined in \eqv(6.theo1.1).
Set
\be
r(\b,h)
\equiv
h^{-1}\left(\b^2q\left(1-q\right)+E2|\b\sqrt{q}Z+h|e^{-2|\b\sqrt{q}Z+h|}\right)
\Eq(6.theo1.2)
\ee
and let $\bar\varrho(\b,h)$ be the function defined on $\DD^{(3)}$ by
\be
\bar\varrho\equiv\bar\varrho(\b,h)
=
\begin{cases}
\left[1-\left(1+\sfrac{1}{h-1}\right)r(\b,h)\right]^2 & \text{if $\left[1-\left(1+\sfrac{1}{h-1}\right)r(\b,h)\right]^2<q$,} \\
q\left[1-\left(1+\sfrac{1}{h-1}\right)r(\b,h)\right]^2 &  \text{else}.
\end{cases}
\Eq(6.theo1.3)
\ee
Clearly,  $\bar\varrho(\b,h)<q$ for all $(\b,h)$ in $\DD^{(3)}$.
Given $\varrho> 0$ set
\be
\BB_{N}(\varrho)=\left\{m\in [-1,1]^N : q_{\text{EA}}(m)< \varrho\right\}.
\Eq(6.theo1.0)
\ee
Theorem \thv(main.theo5) is a reformulation of the following result.

\begin{theorem}
    \TH(6.theo1)
For all $(\b,h)$ in $\DD^{(3)}$
\be
\P\left(\bigcup_{N_0}\bigcap_{N\geq N_0}
\left\{
\sup_{m\in \BB_{N}(\bar\varrho)}F_{N,\b,h}^{HT}(m)< SK(\b,h) 
\right\}
\right)=1.
\Eq(6.theo1.4)
\ee
\end{theorem}

The domain $\DD^{(3)}$ is not the most general possible, but is chosen to satisfy two conditions: it allows to easily bound $r(\b,h)$
and it contains a large part of $\DD^{(2)}_{\bar\varrho}$, especially the  low temperature part. The following two lemmata, which provide bounds on $r(\b,h)$ and $1-q$, and their accompanying remarks, elaborate on these observations. 
 
\begin{lemma}
\TH(6.lem8)
 For all $(\b,h)$, $h>0$, and all $0<\eta<1$, if $h\leq2\eta\b^2q$
\be
\left(1-\sfrac{3}{4(1-\eta^2)}\right)
\sfrac{\b\sqrt{q}}{h}
\leq
r(\b,h) \sqrt{\sfrac{\pi}{2}}e^{\frac{1}{2}\left(\frac{h}{\b\sqrt{q}}\right)^2}
\leq
\left(1+\sfrac{\eta}{1-\eta^2}\right)\sfrac{2\b\sqrt{q}}{h},
\Eq(6.lem8.1)
\ee
\be
\sfrac{1}{4}\sqrt{\sfrac{1}{[\b\sqrt{q}(1+\eta)]^2+1}}
\leq
(1-q) \sqrt{\sfrac{\pi}{2}}e^{\frac{1}{2}\left(\frac{h}{\b\sqrt{q}}\right)^2} 
\leq 
\sfrac{2}{\b\sqrt{q}(1-\eta^2)}.
\Eq(6.lem8.2)
\ee
\end{lemma}

The next lemma is given for the sake of completeness and stated without proof.
For $\eta>0$ set
\be
f(\a_1,\a_2)=\a_1+\a_2\sqrt{\sfrac{2}{\pi}}\sfrac{1}{\b\sqrt{q}}e^{-\frac{1}{2}(2\eta \b\sqrt{q} )^2},
\Eq(6.lem9.0)
\ee
\begin{lemma}
\TH(6.lem9)
For all $(\b,h)$, $h>0$, and all $\eta>0$, if $h\geq 2(1+\eta)\b^2q$, 
\be
f(2-3/[4(1+\eta)],1/\eta)
\leq
r(\b,h)e^{2\left(h-\b^2q\right)}
\leq
f(2,3/[4(1+\eta)]),
\Eq(6.lem9.1)
\ee
\be
f(1,1/(4\eta))
\leq
(1-q)e^{2\left(h-\b^2q\right)}
\leq 
f(4,2/(1+\eta)).
\Eq(6.lem9.2)
\ee
\end{lemma}

\begin{remark}
Under the assumptions of Lemma \thv(6.lem8) and Lemma \thv(6.lem9), respectively, $r(\b,h)$ and $1-q$ have a common leading exponential decay. The conditions on $h/\b$ and $h$ entering the definition of $\DD^{(3)}$ serve to control this decay from above.
In Lemma \thv(6.lem9), the pre-factors modulating the exponential decay are not sharp enough to tell which of $\bar\varrho$ or $q$ is larger (this is due to the rough bounds of Lemma \thv(6.lem6)). On the contrary, in Lemma \thv(6.lem8), these pre-factors  guarantee that $\bar \varrho<q$ if $\b^2q/h$ is large enough, a fact already clear from \eqv(6.theo1.2).
\end{remark}

\begin{remark} 
Recall that  in the physics literature the magnetic field is the quantity $h'=h/\b$. 
Under the assumptions of  Lemma \thv(6.lem8), namely if $h'\leq2\eta\b q$, it follows from \eqv(6.lem8.2) 
that when the field $h'$ is large, $q$ is close to one and 
$
\b^2(1-q)\sim\b \exp\bigl(-\frac{1}{2}h'^2\bigr)
$.
Comparing to formula (23) of \cite{AT78}, we see that up to a constant pre-factor, $\b^2(1-q)$ has the same behaviour as 
$
\b^2E\left[\cosh^{-4}(\b\sqrt{q}Z+\b h')\right]
$
for large fields. Thus, under these assumptions on $(\b, h')$, the condition $\b^2(1-q)<1$ is analogous to the AT-condition \eqv(1.7). This sheds light on the domain $\DD^{(3)}$.
\end{remark}

We now turn to the proof of Theorem \thv(6.theo1). It hinges on two key propositions.
Set
\be
\begin{split}
\psi_{\b,h}(\rho) 
\,=\,& 
\sqrt{\rho}h\left(1-\sqrt{\sfrac{\rho}{q}}\right) 
+\frac{\b^2}{2}\left(1-q\right)\left(q-\rho\right)
\\
\,+\,& 
E\left[\log\cosh\left(\sqrt{\sfrac{\rho}{q}}\left(\b\sqrt{q}Z+h\right)\right)-\log\cosh(\b\sqrt{q}Z+h)\right].
\end{split}
\Eq(6.choice1.4)
\ee

\begin{proposition}
    \TH(6.prop1)
Let $0< \varrho\leq q$ be given.There exists a subset $\wt\O_N\subseteq \O$ with $\P\bigl(\wt\O_N\bigr)\geq 1-e^{-N}$ such that on $\wt\O_N$, 
\be
\sup_{m\in \BB_{N}(\varrho)}F_{N,\b,h}^{HT}(m)
\leq  
SK(\b,h)+\sup_{0\leq \rho<\varrho}\psi_{\b,h}(\rho) + \OO\Bigl(\sqrt{\sfrac{\log N}{N}}\Bigr).
\Eq(6.prop1.1)
\ee
(See \eqv(6.prop1.22) for the precise form of the error term.)
\end{proposition}

\begin{proposition}
    \TH(6.prop2)
For all $(\b,h)$ in $\DD^{(3)}$ and for $\bar\varrho$ defined in \eqv(6.theo1.3), 
$
\sup_{0\leq \rho<\bar\varrho}\psi_{\b,h}(\rho)<0
$.
\end{proposition}

\begin{proof}[Proof of Theorem \thv(6.theo1)] 
Given Proposition \thv(6.prop2) and taking $\rho=\bar\varrho$ in Proposition \thv(6.prop1), the theorem follows from  Borel-Cantelli lemma.
\end{proof}

In the rest of this section,  we first prove Proposition \thv(6.prop1) and Proposition \thv(6.prop2), with Lemma \thv(6.lem8) being proved at the very end, as well as Proposition \thv(main.prop1), and Theorems \thv(main.theo2) and \thv(main.theo6)  from Section \thv(S1).

\begin{proof}[Proof of Proposition \thv(6.prop1)] 
By \eqv(2.cor1.2bis) and \eqv(2.2), we can write
\be
F_{N,\b,h}^{HT}(m)= f(m)+\frac{\b^2}{4}(1-q)^2,
\Eq(6.prop1.3)
\ee
where
\be
f(m) = \frac{1}{N}\left\{\frac{\b}{2}(m,\sfrac{J}{\sqrt N} m) +h({\bf{1}},m)-\sum_{i=1}^{N} I(m_i)\right\}
+\frac{\b^2}{2}\left(1-q\right)\left(q-q_{\text{EA}}(m)\right).
\Eq(6.prop1.2)
\ee
Our task is thus to bound $\sup_{m\in \BB_{N}(\varrho)}f(m)$. To this end, we first replace this quantity by its expectation using a classical Gaussian concentration inequality, and then apply Gaussian comparison techniques to linearise the quadratic form. This is the content of the next two lemmata.

\begin{lemma}
    \TH(6.lem1)
\be  
\P\left(  
\sup_{m\in \BB_{N}(\varrho)}f(m) \geq  \E\sup_{m\in \BB_{N}(\varrho)}f(m)
+\frac{\b\varrho}{\sqrt N}
\right)
\leq e^{-N}.
\Eq(6.lem1.1)
\ee
\end{lemma}
Let ${\bf{Z}}=(Z_i)_{1\leq i\leq N}$ be a standard Gaussian random vector in $\R^N$ and, denoting by ${\bf{E}}$ the expectation with respect to ${\bf{Z}}$, set
\be
\bar{f}(m)=\frac{1}{N}\left\{\b({\bf{Z}},m)\frac{\|m\|_2^1}{\sqrt N}+h({\bf{1}},m) -\sum_{i=1}^{N} I(m_i)\right\}
+\frac{\b^2}{2}\left(1-q\right)\left(q-q_{\text{EA}}(m)\right).
\Eq(6.lem2.0)
\ee

\begin{lemma}
    \TH(6.lem2)
$    
\displaystyle
\quad
\E\sup_{m\in \BB_{N}(\varrho)}f(m)\leq{\bf{E}}\sup_{m\in \BB_{N}(\varrho)} \bar{f}(m).
$
\end{lemma}

\begin{proof}[Proof of Lemma \thv(6.lem1)] 
Given $m\in \BB_{N}(\varrho)$, let $f(\cdot,m): \R^{N(N-1)/2}\rightarrow \R$ be the function that assigns to 
$x=(x_{ij})_{1\leq i<j\leq N}\in\R^{N(N-1)/2}$ the value
\be
f(x,m)= \frac{1}{N}\left\{\b\sum_{1\leq i<j\leq N}x_{ij}\frac{m_im_j}{\sqrt N}+h({\bf{1}},m)-\sum_{i=1}^{N} I(m_i)\right\}
+\frac{\b^2}{2}\left(1-q\right)\left(q-q_{\text{EA}}(m)\right).
\ee
By Cauchy-Schwarz's inequality,
\bea
f(x,m)-f(y,m)
&\hspace{-6pt} \leq \hspace{-6pt} &
\frac{\b q_{\text{EA}}(m)}{\sqrt{2}N}\|x-y\|_2
\leq 
\frac{\b\varrho}{\sqrt{2}N}\|x-y\|_2.
\eea
Thus, $\sup_{m\in \BB_{N}(\varrho)} f(x,m)$ is Lipschitz with constant $L\equiv {\b\varrho}/{\sqrt{2}N}$ and by Tsirelson-Ibragimov-Sudakov concentration inequality (see \cite{BLM}, Theorem 5.6), for all $t>0$ 
\be
\P\left(\sup_{m\in \BB_{N}(\varrho)} f(m)-\E\sup_{m\in \BB_{N}(\varrho)} f(m)\geq t\right)
\leq 
e^{-\frac{1}{2}\left(\frac{t}{L}\right)^2}.
\Eq(6.lem1.2)
\ee Choosing $t=\b\varrho/\sqrt{N}$ then yields the claim of Lemma  \thv(6.lem1).
\end{proof}

\begin{proof}[Proof of Lemma \thv(6.lem2)]  This is a straightforward application of Sudakov-Fernique Gaussian comparison inequality (see \cite{AdTa}, Theorem 2.2.3). 
Since 
\be
\E f(m)={\bf{E}}\bar{f}(m)=\frac{1}{N}\left\{h({\bf{1}},m)-\sum_{1\leq i\leq N} I(m_i)\right\}
+\frac{\b^2}{2}\left(1-q\right)\left(q-q_{\text{EA}}(m)\right),
\ee 
we only have to check that for all $m\in\BB_{N}(\varrho)$ and all $\bar{m}\in \BB_{N}(\varrho)$,
\be
\E\left[f(m)-f(\bar{m})\right]^2
\leq 
{\bf{E}}\left[\bar{f}(m)-\bar{f}(\bar{m})\right]^2.
\Eq(6.lem2.2)
\ee
Setting
$
\varphi(m)=\frac{1}{2}(m,\sfrac{J}{\sqrt N} m)
$
and 
$
\bar{\varphi}(m)=(g,m)\frac{\|m\|_2^1}{\sqrt N}
$,
\eqv(6.lem2.2) is equivalent to
\be
\Delta\varphi\equiv
\E\left[
\varphi(m)-\varphi(\bar{m})
\right]^2
\leq 
{\bf{E}}\left[
\bar{\varphi}(m)-\bar{\varphi}(\bar{m})
\right]^2
\equiv\Delta\bar{\varphi}.
\Eq(6.lem2.3)
\ee
Working out the expectations in the left and right-hand side of \eqv(6.lem2.3) gives
\bea
\Delta\varphi
&\hspace{-6pt} = \hspace{-6pt} &
\frac{1}{2N}\left[
\left(\|m\|_2^2\right)^2-2(m,\bar{m})+\left(\|\bar{m}\|_2^2\right)^2
\right]
-\frac{1}{2N}\left[
\textstyle\sum_{i=1}^N\left(m_i^2-\bar{m}_i^2\right)^2
\right],
\Eq(6.lem2.4)
\\
\Delta\bar{\varphi}
&\hspace{-6pt} = \hspace{-6pt} &
\frac{1}{N}\left[
\left(\|m\|_2^2\right)^2-2(m,\bar{m})\|m\|_2^1\|\bar{m}\|_2^1+\left(\|\bar{m}\|_2^2\right)^2
\right].
\Eq(6.lem2.5)
\eea
Since
\be
(m,\bar{m})^2-2(m,\bar{m})\|m\|_2^1\|\bar{m}\|_2^1+\|m\|_2^2\|\bar{m}\|_2^2
=
\left\{(m,\bar{m})-\|m\|_2^1\|\bar{m}\|_2^1\right\}^2
\geq 0
\Eq(6.lem2.6)
\ee
then
\be
-2(m,\bar{m})^2\leq 2\left\{-2(m,\bar{m})\|m\|_2^1\|\bar{m}\|_2^1+\|m\|_2^2\|\bar{m}\|_2^2\right\}.
\Eq(6.lem2.7)
\ee
Inserting \eqv(6.lem2.7) in the first term on the right-hand side of \eqv(6.lem2.4) and dropping the second,
\bea
\Delta\varphi
&\hspace{-6pt} \leq \hspace{-6pt} &
\frac{1}{2N}
\left[
\left(\|m\|_2^2\right)^2+\left(\|\bar{m}\|_2^2\right)^2-4(m,\bar{m})\|m\|_2^1\|\bar{m}\|_2^1+2\|m\|_2^2\|\bar{m}\|_2^2
\right].
\Eq(6.lem2.8)
\eea
Now, using that $(a+b)^2\leq 2(a^2+b^2)$ with $a=\|m\|_2^2$ and $b=\|\bar{m}\|_2^2$ and recalling \eqv(6.lem2.5)
\be
\Delta\varphi
\leq
\frac{1}{2N}
\left\{
2\left[\left(\|m\|_2^2\right)^2+\left(\|\bar{m}\|_2^2\right)^2\right]-4(m,\bar{m})\|m\|_2^1\|\bar{m}\|_2^1
\right\}
=\Delta\bar{\varphi}.
\ee
This proves \eqv(6.lem2.3), and hence \eqv(6.lem2.2). The proof of Lemma \thv(6.lem2) is complete.
\end{proof}

We now return to the proof of Proposition \thv(6.prop1). Combining Lemma \thv(6.lem1) and Lemma \thv(6.lem2), there exists $\wt\O_N\subseteq \O$ with $\P\bigl(\wt\O_N\bigr)\geq 1-e^{-N}$ such that on $\wt\O_N$,
\be
\sup_{m\in \BB_{N}(\varrho)} f(m)
\leq 
{\bf{E}}\sup_{m\in \BB_{N}(\varrho)}\bar{f}(m)
+\frac{\b\varrho}{\sqrt N}.
\Eq(6.prop1.7)
\ee
By \eqv(6.theo1.0)
\be
\sup_{m\in \BB_{N}(\varrho)}\bar{f}(m) 
=\sup_{0\leq \rho<\varrho}\sup_{m\in [-1,1]^N : q_{\text{EA}}(m)=\rho}\bar{f}(m).
\Eq(6.prop1.8)
\ee
Going back to the definition \eqv(6.lem2.0) of $\bar{f}$ and using the fact (seen in the proof of Proposition \thv(2.prop1)) that the functions $\II_N$ and $\II_N^{*}$ defined in \eqv(2.5) form a pair of Legendre-Fenchel conjugates, we have
\be
\begin{split}
&
\sup_{m\in [-1,1]^N : q_{\text{EA}}(m)=\rho}\bar{f}(m) 
\leq \sup_{m:q_{\text{EA}}(m)=\rho}\bar{f}(m) 
\\
& = \sup_{m:q_{\text{EA}}(m)=\rho}\frac{1}{N}\left\{(\b\sqrt{\rho}{\bf{Z}}+h{\bf{1}},m)-\sum_{i=1}^{N} I(m_i)\right\}
\\
& = \sup_{m:q_{\text{EA}}(m)=\rho}\frac{1}{N}\left\{(\b\sqrt{\rho}{\bf{Z}}+h{\bf{1}},m)
-\sup_{y\in\R^N}\left\{(m,y)-\sum_{i=1}^{N}I^{*}(y_i)\right\}\right\}
\\
& =\inf_{y\in\R^N}\left\{\sqrt{\rho}\frac{\|\b\sqrt{\rho}{\bf{Z}}+h{\bf{1}}-y\|_2^1}{\sqrt{N}}+\frac{1}{N}\sum_{i=1}^{N}I^{*}(y_i)\right\}.
\end{split}
\Eq(6.prop1.9)
\ee
Finding the above infimum explicitly is beyond our reach.
Instead, we make an arbitrary (and hopefully judicious) choice of $y$ by taking
\be
y=\sqrt{\frac{\rho}{q}}\left(\b\sqrt{q}{\bf{Z}}+h{\bf{1}}\right).
\Eq(6.choice1.1)
\ee
With this choice, it follows from \eqv(6.prop1.7), \eqv(6.prop1.8) and \eqv(6.prop1.9) that on $\wt\O_N$
\be
\begin{split}
\sup_{m\in \BB_{N}(\varrho)}f(m) 
\leq \,
& {\bf{E}}\sup_{0\leq \rho<\varrho}\Biggl\{
\sqrt{\rho}h\left(1-\sqrt{\sfrac{\rho}{q}}\right) 
+\frac{\b^2}{2}\left(1-q\right)\left(q-\rho\right)
\\
&+\frac{1}{N}\sum_{i=1}^N
\log\cosh\left(\sqrt{\sfrac{\rho}{q}}\left(\b\sqrt{q}Z_i+h\right)\right)
\Biggr\}
+\log 2+\frac{\b\varrho}{\sqrt N}.
\end{split}
\Eq(6.prop1.10)
\ee
To complete the proof of Proposition \thv(6.prop1), it remains to replace the  random function within braces in \eqv(6.prop1.10) by its expectation.  Let ${\bf{Z}}=(Z_i)_{1\leq i\leq N}$ be as in \eqv(6.lem2.0) and set
\be
g(\rho, {\bf{Z}})
\equiv
\frac{1}{N}\sum_{i=1}^N\log\cosh\left(\sqrt{\sfrac{\rho}{q}}\left(\b\sqrt{q}Z_i+h\right)\right).
\Eq(6.lem7.0)
\ee
\begin{lemma}
    \TH(6.lem7)
For  all $k> 1$
\be  
\P\left(  
\sup_{0\leq \rho<\varrho}\left[g(\rho, {\bf{Z}})- {\bf{E}}g(\rho, {\bf{Z}})\right]
\geq
2\b\sqrt{\frac{k q\log N}{N}}(1+N^{-k})
\right)
\leq N^{-k}.
\Eq(6.lem7.1)
\ee
Moreover,
\be
{\bf{E}}
\left\{
\sup_{0\leq \rho<\varrho}\left[g(\rho, {\bf{Z}})- {\bf{E}}g(\rho, {\bf{Z}})\right]^2
\right\}
\leq
(\b^2q+h^2)^2.
\Eq(6.lem7.1bis)
\ee
\end{lemma}
\begin{proof}[Proof of Lemma \thv(6.lem7)] 

To control the probability of the supremum in \eqv(6.lem7.1), we introduce the discrete set
$
\left\{\rho_j\equiv \varrho jN^{-k}, j=0,1,\dots, N^k-1\right\}\subset  [0, \varrho]
$
where $k>1$ is to be chosen later. Using this set, we define the sequence of functions
\be
g^{(j)}(\rho, {\bf{Z}})
=
g(\rho, {\bf{Z}})-g(\rho_j, {\bf{Z}}), \quad 0\leq j< N^k.
\Eq(6.lem7.2)
\ee
The supremum in \eqv(6.lem7.1) can then be rewritten as
\be
\begin{split}
& \sup_{0\leq \rho<\varrho}
\bigl\{g(\rho, {\bf{Z}})- {\bf{E}}g(\rho, {\bf{Z}})\bigr\}
\\
= \,\hspace{-10pt}& 
\sup_{0\leq j< N^k}\biggl\{
\sup_{\rho_{j}\leq \rho<\rho_{j+1}}
\Bigl[
g^{(j)}(\rho, {\bf{Z}})-{\bf{E}}g^{(j)}(\rho, {\bf{Z}})
\Bigr]
+
\bigl[g(\rho_j, {\bf{Z}})- {\bf{E}}g(\rho_j, {\bf{Z}})\bigr]
\biggr\}.
\end{split}
\Eq(6.lem7.3)
\ee
To deal with the first term in braces, note that since
\be
|\log\cosh(x)-\log\cosh(y)|=\left|\int_{x}^{y}\tanh(t)dt\right|\leq |x-y|,
\Eq(6.lem7.4)
\ee 
then, for each $0\leq j< N^k$ and all $\rho_{j}\leq \rho<\rho_{j+1}$
\be
\begin{split}
g^{(j)}(\rho, {\bf{Z}}) 
& \leq 
\left|\sqrt{\sfrac{\rho_{j+1}}{q}}-\sqrt{\sfrac{\rho_{j}}{q}}\right|\frac{1}{N}\sum_{i=1}^N\left|\b\sqrt{q}Z_i+h\right|
\leq  \sqrt{\frac{\varrho}{qN^k}}\frac{1}{N}\sum_{i=1}^N\left|\b\sqrt{q}Z_i+h\right|
\\
&
\equiv 
\bar g({\bf{Z}}).
\end{split}
\Eq(6.lem7.5)
\ee
Also note that
\bea
\left|{\bf{E}}g^{(j)}(\rho, {\bf{Z}})\right|
\leq 
{\bf{E}}\bar g({\bf{Z}})
\leq
\sqrt{\frac{\varrho}{qN^k}}
\frac{1}{N}
{\textstyle\left[{\bf{E}}\left(\sum_{i=1}^N\left|\b\sqrt{q}Z_i+h\right|\right)^2\right]^{1/2}}
=
t_k
\Eq(6.lem7.6)
\eea
where $t_k\equiv t_k(\varrho,\b,h)=\sqrt{\frac{\varrho(\b^2q+h^2)}{qN^k}}$. 
Thus, by \eqv(6.lem7.5) and \eqv(6.lem7.6)
\bea
\sup_{\rho_{j}\leq \rho<\rho_{j+1}}
\Bigl[
g^{(j)}(\rho, {\bf{Z}})-{\bf{E}}g^{(j)}(\rho, {\bf{Z}})
\Bigr]
\leq 
\bar g({\bf{Z}})-{\bf{E}}\bar g({\bf{Z}})+2t_k.
\Eq(6.lem7.7)
\eea
Next, observe that for each fixed $\rho_j$, $0\leq j< N^k$,
the function $\bar g(x)+g(\rho_j, x)$ viewed as a function of 
$x=(x_i)_{1\leq i\leq N}\in\R^N$ obeys
\be
\begin{split}
|(\bar g(x)+g(\rho_j, x))-(\bar g(y)+g(\rho_j, y))| 
& 
\leq \left|\bar g(x)-\bar g(y)\right|
+
\left|g(\rho_j, x)-g(\rho_j, y)\right|
\\
& 
\leq \frac{\b}{N^{k}} \sqrt{\frac{q}{N}} \|x-y\|_2
+
\b \sqrt{\frac{q}{N}} \|x-y\|_2
\end{split}
\Eq(6.lem7.9)
\ee
where we used
in turn \eqv(6.lem7.4) to bound $\left|g(\rho_j, x)-g(\rho_j, y)\right|$ and Cauchy-Schwarz's inequality.
Finally, combining \eqv(6.lem7.3), \eqv(6.lem7.7) and the above Lipschitz property, \eqv(6.lem7.1)  follows from Tsirelson-Ibragimov-Sudakov  inequality \eqv(6.lem1.2) with 
$
L\equiv \b \sqrt{\frac{q}{N}}\left(1+\frac{1}{N^{k}}\right)
$ 
and
$
t=2t_k+2t'_k
$,
where 
\be
t'_k\equiv\b\sqrt{k qN^{-1}\log N}(1+N^{-k}).
\Eq(6.lem7.14)
\ee
For later use we denote by $\O_{k,N}$ the event
\be
\O_{k,N}=\left\{\o\in\O \,:\, \sup_{0\leq \rho<\varrho}\left[g(\rho, {\bf{Z(\o)}})- {\bf{E}}g(\rho, {\bf{Z(\o)}})\right]
\geq
2t'_k
\right\}.
\Eq(6.lem7.13)
\ee

We now turn to  \eqv(6.lem7.1bis). From \eqv(6.lem7.0) and  the bound $\log\cosh(x)\leq  x^2/2$, $x\in\R$, we get, 
for all $0\leq \rho\leq \varrho\leq q$
\be
g(\rho, {\bf{Z}})
\leq 
\frac{1}{2N}\frac{\rho}{q}\sum_{i=1}^N\left(\b\sqrt{q}Z_i+h\right)^2
\leq 
\frac{1}{2N}\sum_{i=1}^N\left(\b\sqrt{q}Z_i+h\right)^2
\equiv
\hat g({\bf{Z}}).
\Eq(6.prop1.18)
\ee
This and the fact that $g(\rho, {\bf{Z}})\geq 0$ yield
\bea
\sup_{0\leq \rho<\varrho}\left[g(\rho, {\bf{Z}})- {\bf{E}}g(\rho, {\bf{Z}})\right]^2
\leq 
\hat g^2({\bf{Z}})+\left({\bf{E}}\hat g({\bf{Z}})\right)^2.
\Eq(6.prop1.19)
\eea
Now
\be
\begin{split}
\left({\bf{E}}\hat g({\bf{Z}})\right)^2 
&
= \frac{1}{4}(\b^2q+h^2)^2,
\\
{\bf{E}}\hat g^2({\bf{Z}}) 
& 
=\frac{1}{4}(\b^2q+h^2)^2+ \frac{1}{2N}\b^2q(\b^2q+2h^2)
\leq 
(\b^2q+h^2)^2.
\end{split}
\Eq(6.prop1.20)
\ee
Taking the expectation of both sides of \eqv(6.prop1.19) and inserting the bounds \eqv(6.prop1.20) gives \eqv(6.lem7.1bis). 
The proof of Lemma \thv(6.lem7) is complete. 
\end{proof}

Set
\be
\phi(\rho) =
\sqrt{\rho}h\left(1-\sqrt{\sfrac{\rho}{q}}\right)+\frac{\b^2}{2}\left(1-q\right)\left(q-\rho\right)
+E\log\cosh\left(\sqrt{\sfrac{\rho}{q}}\left(\b\sqrt{q}Z+h\right)\right).
\Eq(6.prop1.14)
\ee
Then, on $\wt\O_N$ (see the line above \eqv(6.prop1.7)), \eqv(6.prop1.10) can be rewritten as
\be
\sup_{m\in \BB_{N}(\varrho)}f(m) 
\leq
 \sup_{0\leq \rho<\varrho}\phi(\rho)+\frac{\b\varrho}{\sqrt N}+\log 2+\EE
\Eq(6.prop1.13)
\ee
where, recalling the definition \eqv(6.lem7.0) of $g(\rho, {\bf{Z}})$
\be
\EE\equiv{\bf{E}}\left\{\sup_{0\leq \rho<\varrho}\left[g(\rho, {\bf{Z}})- {\bf{E}}g(\rho, {\bf{Z}})\right]\right\}.
\Eq(6.prop1.14')
\ee
Using the notation \eqv(6.lem7.13), we decompose \eqv(6.prop1.14') into $\EE=\EE^{(1)}+\EE^{(2)}$
where
\bea
\EE^{(1)}
&\hspace{-6pt}=\hspace{-6pt} &
{\bf{E}}\left\{
\sup_{0\leq \rho<\varrho}\left[g(\rho, {\bf{Z}})- {\bf{E}}g(\rho, {\bf{Z}})\right]
\1_{\left\{\O^c_{k,N}\right\}}
\right\},
\Eq(6.prop1.15)
\\
\EE^{(2)}
&\hspace{-6pt}=\hspace{-6pt} &
{\bf{E}}\left\{
\sup_{0\leq \rho<\varrho}\left[g(\rho, {\bf{Z}})- {\bf{E}}g(\rho, {\bf{Z}})\right]
\1_{\left\{\O_{k,N}\right\}}
\right\}.
\Eq(6.prop1.16)
\eea
Clearly 
$
\EE^{(1)}\leq 2t'_k
$
(see \eqv(6.lem7.14)). To bound $\EE^{(2)}_{\b,h}$ we use successively Cauchy-Schwarz's inequality and Lemma \thv(6.lem7) to write
\be
\begin{split}
\EE^{(2)}
\leq
&
\sqrt{ 
{\bf{E}}
\left\{
\sup_{0\leq \rho<\varrho}\left[g(\rho, {\bf{Z}})- {\bf{E}}g(\rho, {\bf{Z}})\right]^2
\right\}
}
\sqrt{
\P\left(\O_{k,N}\right)
}
\leq (\b^2q+h^2)N^{-k/2}.
\end{split}
\Eq(6.prop1.17)
\ee
Collecting our bounds we obtain, taking, e.g., $k=2$,
\be
\EE
\leq2\b\sqrt{\sfrac{\log N}{N}}+(\b^2q+h^2)N^{-1}.
\Eq(6.prop1.21)
\ee
Plugging in \eqv(6.prop1.13),  we have on $\wt\O_N$
\be
\sup_{m\in \BB_{N}(\varrho)}f(m) 
\leq
\sup_{0\leq \rho<\varrho}\phi(\rho)+\log 2+2\b\sqrt{\sfrac{\log N}{N}}+\b\varrho N^{-1/2}+(\b^2q+h^2)N^{-1}.
\Eq(6.prop1.22)
\ee
We are ready to complete the proof of Proposition \thv(6.prop1). By \eqv(6.prop1.3),
\be
\sup_{m\in \BB_{N}(\varrho)}F_{N,\b,h}^{HT}(m)= \sup_{m\in \BB_{N}(\varrho)}f(m)+\frac{\b^2}{4}(1-q)^2
\Eq(6.prop1.23)
\ee
while by \eqv(1.6), \eqv(6.choice1.4) and \eqv(6.prop1.14) 
\be
\sup_{0\leq \rho<\varrho}\phi(\rho)+\frac{\b^2}{4}(1-q)^2+\log 2 = SK(\b,h) + \sup_{0\leq \rho<\varrho}\psi_{\b,h}(\rho).
\Eq(6.prop1.24)
\ee
Proposition \thv(6.prop1) now follows from \eqv(6.prop1.22), \eqv(6.prop1.23) and \eqv(6.prop1.24).
\end{proof}

We now turn to the proof of Proposition \thv(6.prop2).

\begin{proof}[Proof of Proposition \thv(6.prop2)]
Let us write
\bea
\psi_{\b,h}(\rho) 
\hspace{-6pt}&=&\hspace{-6pt}
\sqrt{\rho}h\left(1-\sqrt{\sfrac{\rho}{q}}\right) 
+\frac{\b^2}{2}\left(1-q\right)\left(q-\rho\right)
+\Delta_{\b,h}(\rho),
\Eq(6.prop2.5)
\\
\Delta_{\b,h}(\rho) 
\hspace{-6pt}&\equiv&\hspace{-6pt}
E\left[\log\cosh\left(\sqrt{\sfrac{\rho}{q}}\left(\b\sqrt{q}Z+h\right)\right)-\log\cosh(\b\sqrt{q}Z+h)\right].
\Eq(6.prop2.6)
\eea
The next lemma collects properties of $\Delta_{\b,h}(\rho)$ and provides two bounds that will be useful in different range of $\rho$  	(and for $h$ large enough).
\begin{lemma}
    \TH(6.lem4)
For all $0\leq \rho<q$, $\Delta_{\b,h}(\rho)< 0$,
$\Delta_{\b,h}(\rho)\uparrow 0$ as $\rho\uparrow q$ and we have
\bea
&&\Delta_{\b,h}(\rho) 
\leq 
\frac{\rho}{2q}E\left(\b\sqrt{q}Z+h\right)^2-E|\b\sqrt{q}Z+h|+\log 2,
\Eq(6.lem4.2)
\\
&&
\begin{split}
\Delta_{\b,h}(\rho) 
\leq & -\left(1-\sqrt{\sfrac{\rho}{q}}\right) \Bigl(E|\b\sqrt{q}Z+h|-E2|\b\sqrt{q}Z+h|e^{-2|\b\sqrt{q}Z+h|}\Bigr)\quad
\\
&+ \frac{q}{4\rho}\left(1-\sqrt{\sfrac{\rho}{q}}\right)^2.
\end{split}
\Eq(6.lem4.1)
\eea
\end{lemma}

\begin{proof}[Proof of Lemma \thv(6.lem4)]
We deduce from the identity
\be
\textstyle
\log\cosh\bigl(\sqrt{\rho/q} x\bigr)-\log\cosh(x)=-\int_{\sqrt{\rho/q}|x|}^{|x|}\tanh(y)dy, \quad x\in\R,
\Eq(6.lem4.1')
\ee
that $\Delta_{\b,h}(\rho)$ is strictly negative for $0<\rho<q$, and increases to $0$ as $\rho$ increases to $q$.
To prove \eqv(6.lem4.1), consider the function $f(x)=\log\cosh(\a x)$, $0<\a<\infty$. 
By Taylor's theorem to second order (with remainder in Lagrange form), 
$f$ is approximated at $x_0$ by
\be
f(x)=f(x_0)+(x-x_0)\a\tanh(\a x_0)+\frac{1}{2}(x-x_0)^2\left(\sfrac{\a}{\cosh(\a \xi)}\right)^2
\Eq(6.lem4.5)
\ee
for some $\xi$ between $x$ and $x_0$. 
Since $\left(\sfrac{\a \xi}{\cosh(\a \xi)}\right)^2\leq\sfrac{1}{2}$, this implies that for all $0<x\leq x_0$
\bea
f(x)
&\leq&
f(x_0)+(x-x_0)\a\tanh(\a x_0)+\frac{1}{4x^2}(x-x_0)^2.
\Eq(6.lem4.6)
\eea
We now use \eqv(6.lem4.6) with $x=\sqrt{\rho/q}$, $x_0=1$ and $\a=|\b\sqrt{q}Z+h|$ to bound $\Delta_{\b,h}(\rho)$. To do this, we first introduce a truncation threshold $L>0$, split $\Delta_{\b,h}(\rho)$ into two terms according to whether $|\b\sqrt{q}Z+h|<L$ 
or $|\b\sqrt{q}Z+h|\geq L$, apply \eqv(6.lem4.6) to the first term, show that the second decays to zero exponentially fast in $L$
(using e.g.~the bounds
$
\log\cosh(y)+\log 2=|y|+\log\left(1+e^{-2|y|}\right)
$,
$
0\leq \log\left(1+e^{-2|y|}\right)\leq e^{-2|y|}
$
).
We skip the simple but lengthy details of the proof. Doing so and passing to the limit $L\uparrow \infty$, we get
\be
\Delta_{\b,h}(\rho) 
\leq 
-\left(1-\sqrt{\sfrac{\rho}{q}}\right) E|\b\sqrt{q}Z+h|\tanh(|\b\sqrt{q}Z+h|)+\sfrac{q}{4\rho}\left(1-\sqrt{\sfrac{\rho}{q}}\right)^2.
\Eq(6.lem4.7)
\ee
Observing that 
$
|y|\tanh(|y|)=|y|\left(1-2\sfrac{e^{-2|y|}}{1+e^{-2|y|}}\right)\geq |y|\left(1-2e^{-2|y|}\right)
$
finally gives \eqv(6.lem4.1).

Eq.~\eqv(6.lem4.2) follows from the classical bounds, valid for all $x\in\R$,
\be
|x|-\log 2\leq |x|+\log\left(1+e^{-2|x|}\right)-\log 2=\log\cosh(x)\leq  x^2/2.
\Eq(6.lem4.3)
\ee
The proof of Lemma \thv(6.lem4) is complete.
\end{proof}

The next lemma is needed to estimate the expectations appearing in Lemma \thv(6.lem4). Define 
\be
\erfc(z)=2\int_{z}^\infty e^{-\frac{x^2}{2}}\frac{dx}{\sqrt{2\pi}}.
\ee
\begin{lemma}
    \TH(6.lem3)
Let $Z$ be standard Gaussian random variable. For all $a>0$ and $b\geq 0$
\bea
E|aZ+b| &\hspace{-6pt}=\hspace{-6pt}& a\sqrt{\sfrac{2}{\pi}}e^{-\frac{1}{2}\left(\frac{b}{a}\right)^2}
+b\left[1-\erfc\left(\sfrac{b}{a}\right)\right],
\Eq(6.lem3.1)
\\
Ee^{-|aZ+b|}&\hspace{-6pt}=\hspace{-6pt}&t^-_{a,b}+t^+_{a,b},
\Eq(6.lem3.1')
\\
E|aZ+b|e^{-|aZ+b|}&\hspace{-6pt}=\hspace{-6pt}&
b\left(t^-_{a,b}-t^+_{a,b}\right)-a^2\left(t^-_{a,b}+t^+_{a,b}\right)+a\sqrt{\sfrac{2}{\pi}}e^{-\frac{1}{2}\left(\frac{b}{a}\right)^2},\,
\Eq(6.lem3.1'')
\eea
where
\bea
t^-_{a,b}&\hspace{-6pt}=\hspace{-6pt}&
e^{\sfrac{a^2}{2}-b}\left[1-\sfrac{1}{2}\erfc\left(\sfrac{b}{a}-a\right)\right]1_{\left\{-\frac{b}{a}+a\leq 0\right\}}
+
\sfrac{1}{2}e^{\frac{a^2}{2}-b}\erfc\left(-\sfrac{b}{a}+a\right)1_{\left\{-\frac{b}{a}+a\geq 0\right\}},
\nonumber
\\
t^+_{a,b}&\hspace{-6pt}=\hspace{-6pt}&\sfrac{1}{2}e^{\frac{a^2}{2}+b}\erfc\left(\sfrac{b}{a}+a\right).
\nonumber
\eea
\end{lemma}

\begin{proof}[Proof of Lemma \thv(6.lem3)]
Eq.~\eqv(6.lem3.1) and \eqv(6.lem3.1') are straightforward. Eq.~\eqv(6.lem3.1'') relies on Gaussian integration by parts.
\end{proof}

It is well known  (see\cite{AbSt}, inequalities 7.1.13) that $\erfc(z)$ obeys the bounds
\be
C^-(z)=\frac{2}{z+\sqrt{z^2+4}}
\leq 
\sqrt{\frac{\pi}{2}}e^{\frac{z^2}{2}}\erfc\bigl(z\bigr)
\leq 
C^+(z)=\frac{2}{z+\sqrt{z^2+\frac{8}{\pi}}}
\leq
\frac{1}{z}.\quad
\Eq(6.lem3.0')
\ee

\begin{corollary}
    \TH(6.cor1)
For all $(\b,h)$, $E|\b\sqrt{q}Z+h| \geq h$.
\end{corollary}
\begin{proof} 
Use \eqv(6.lem3.1) and the rightmost upper bound of \eqv(6.lem3.0').
\end{proof}

We now proceed in three steps, using the two bounds of Lemma \thv(6.lem4)  in turn.

\noindent{\bf Step 1}: 
Using Corollary \thv(6.cor1) in \eqv(6.lem4.2) together with our assumptions on $(\b,h)$
\be
\Delta_{\b,h}(\rho)
\leq  
\frac{\rho}{2q}\left(\b^2q+h^2\right)-h+\log 2
\leq  
\frac{5}{8}\frac{\rho}{q}h^2-h+\log 2,
\Eq(6.step1.2)
\ee
and inserting in \eqv(6.choice1.4), we get
\be
\psi_{\b,h}(\rho) 
\leq 
\Upsilon_{\b,h}(\rho)\equiv
h\left\{
\sqrt{\rho}
+\frac{q\b^2}{2h}\left(1-q\right)
+\frac{5}{8}\frac{\rho}{q}h-1+\frac{\log 2}{h}
\right\}.
\Eq(6.step1.2')
\ee
The right-hand-side of \eqv(6.step1.2') is a quadratic fonction of $\sqrt{\rho}$. One checks that $\Upsilon_{\b,h}(\rho)=0$ has a single strictly positive root, $\rho^+>0$, and that $\Upsilon_{\b,h}(\rho)<0$ in the interval $[0,\rho^+)$. One also checks that for all $(\b,h)$ such that $\b^2(1-q)\leq 1$ and $h\geq 4$, 
\be
\rho^+\geq \rho^{(1)}\equiv\frac{q}{h}\left({4}/{5}\right)^3(4/3)
\Eq(6.step1.3)
\ee
and so, on that domain, $\psi_{\b,h}(\rho)<0$ for all 
\be
0\leq  \rho< \rho^{(1)}.
\Eq(6.step1.4)
\ee

\noindent{\bf Step 2}: We now assume that $q>\rho\geq \rho^{(1)}$. Inserting \eqv(6.lem4.1) in \eqv(6.prop2.5),
observing that
\be
\sqrt{\rho}\left(1-\sqrt{\sfrac{\rho}{q}}\right) \leq\sqrt{\sfrac{\rho}{q}}\left(1-\sqrt{\sfrac{\rho}{q}}\right),
\quad
\frac{q}{4h\rho}\leq \frac{3}{8},
\quad
\frac{1}{2}\left(q-\rho\right)
\leq
q\left(1-\sqrt{\sfrac{\rho}{q}}\right),
\Eq(6.step2.1)
\ee
and using again Corollary \thv(6.cor1), we get
\bea
\psi_{\b,h}(\rho) 
&\hspace{-6pt} \leq \hspace{-6pt} &
 h\left(1-\sqrt{\sfrac{\rho}{q}}\right)\left\{-\left(1- \sfrac{3}{8}\right)\left(1-\sqrt{\sfrac{\rho}{q}}\right)+r(\b,h)\right\}
\Eq(6.step2.2)
\eea
where 
$
r(\b,h)\equiv
h^{-1}\left[\b^2q\left(1-q\right)+E2|\b\sqrt{q}Z+h|e^{-2|\b\sqrt{q}Z+h|}\right].
$
Thus $\psi_{\b,h}(\rho)< 0$ if 
\be
\sqrt{\frac{\rho}{q}}< 1-\frac{8}{5}r(\b,h).
\Eq(6.step2.4)
\ee
The next lemma provides bounds on $r(\b,h)$. We postpone its proof to the end of the section.

\begin{lemma}
    \TH(6.lem5)
For all $(\b,h)$, $h>0$
\be
r(\b,h) \leq r^*(\b,h)\equiv
\begin{cases}
\left(3+\sfrac{4}{\sqrt{2\pi}}\frac{\b\sqrt{q}}{h}\right)e^{-h}
& \text{if $h>2\b^2q$,} \\
\left(1+\sfrac{6}{\sqrt{2\pi}}\frac{\b\sqrt{q}}{h}\right)e^{-\frac{1}{2}\left(\frac{h}{\b\sqrt{q}}\right)^2}
& \text{if $h\leq2\b^2q$.}
\end{cases}
\Eq(6.lem5.1)
\ee
\end{lemma}
One checks, using Lemma \thv(6.lem5), that for all $(\b,h)$ such that $h\geq 4$ and $h/\b\geq 2$, $0<1-\frac{8}{5}r^*(\b,h)\leq 1$ and 
$\rho^{(1)}<q\left(1-\frac{8}{5}r^*(\b,h)\right)^2$.
Hence, on that domain, $\psi_{\b,h}(\rho)<0$ for all 
\be
\rho^{(1)}\leq \rho<\rho^{(2)}\equiv q\left(1-\frac{8}{5}r(\b,h)\right)^2.
\Eq(6.step2.5)
\ee

\noindent{\bf Step 3}: In this last step we repeat step 2 assuming this time that $q>\rho\geq \rho^{(2)}$. Doing this we get
that $\psi_{\b,h}(\rho)<0$ if 
\be
\sqrt{\frac{\rho}{q}}<\sqrt{\frac{\rho^{(3)}}{q}}
\equiv 
1-\left(1-\left[4h\left(1-\sfrac{8}{5}r(\b,h)\right)^2\right]^{-1}\right)^{-1}r(\b,h).
\Eq(6.step3.1)
\ee
One checks, using  Lemma \thv(6.lem5), that for $(\b,h)$ as in Step 2,
$
4\left(1-\sfrac{8}{5}r(\b,h)\right)^2>1
$.
Hence
\be
\rho^{(3)}>
q\left[1-\left(1+(h-1)^{-1}\right)r(\b,h)\right]^2\geq q\left[1-(4/3)r(\b,h)\right]^2>\rho^{(2)}.
\Eq(6.step3.2)
\ee
To go from \eqv(6.step3.1) to \eqv(6.step3.2) we checked, again by Lemma \thv(6.lem5), that under our assumptions on $(\b,h)$ the right-hand side of \eqv(6.step3.1) and all the quantities in square brackets in \eqv(6.step3.2) are positive.

If $\left[1-\left(1+(h-1)^{-1}\right)r(\b,h)\right]^2<q$, the lower bound on $\rho^{(3)}$ obtained in \eqv(6.step3.2) can be improved. Indeed, for all $\rho< q$ we can write, instead of \eqv(6.step2.1) 
\be
\frac{q}{4h\rho}\left(1-\sqrt{\sfrac{\rho}{q}}\right)
\leq \frac{3}{8}\left(1-\sqrt{\rho}\right),
\quad
\frac{1}{2}\left(q-\rho\right)
\leq
q\left(1-\sqrt{\sfrac{\rho}{q}}\right).
\Eq(6.step2.1bis)
\ee
Eq.~\eqv(6.step2.2) then becomes 
\bea
\psi_{\b,h}(\rho) 
&\hspace{-6pt} \leq \hspace{-6pt} &
 h\left(1-\sqrt{\sfrac{\rho}{q}}\right)\left\{-\left(1- \sfrac{3}{8}\right)\left(1-\sqrt{\rho}\right)+r(\b,h)\right\}.
\Eq(6.step2.2bis)
\eea
Setting $q$ to one in the definitions of $\rho^{(2)}$ and $\rho^{(3)}$ and calling $\hat\rho^{(2)}$ and $\hat\rho^{(3)}$ the resulting quantities, the conclusions of Step 2 and Step 3 above hold unchanged for $\hat\rho^{(2)}$ and $\hat\rho^{(3)}$ whenever
$\hat\rho^{(3)}\leq q$.

Combining  \eqv(6.step1.4), \eqv(6.step2.5), \eqv(6.step3.2) and the above observation we proved that $\psi_{\b,h}(\rho)<0$ for all $0\leq \rho\leq \bar\varrho$ and $(\b,h)\in\DD^{(3)}$ with $\bar\varrho$ and $\DD^{(3)}$ defined in \eqv(6.theo1.3) and \eqv(6.theo1.1), respectively. The proof of Proposition \thv(6.prop2) is complete.
\end{proof}

We now prove Lemma \thv(6.lem5) and Lemma  \thv(6.lem8) together.  The proof uses the next

\begin{lemma}
    \TH(6.lem6)
$
Ee^{-2|\b\sqrt{q}Z+h|}\leq 1-q\leq 4Ee^{-2|\b\sqrt{q}Z+h|}
$.
\end{lemma}

\begin{proof}[Proof of Lemma \thv(6.lem6)] By \eqv(1.5), 
$
1-q
=E\cosh^{-2}(\b\sqrt{q}Z+h)
$ . 
The claim of the lemma then follows from the bounds 
$
e^{-2|x|}\leq\cosh^{-2}(x)\leq 4e^{-2|x|}
$, 
$x\in\R$.
\end{proof}
\begin{proof}[Proof of Lemma \thv(6.lem5) and \thv(6.lem8)] Recall \eqv(6.theo1.2)  and set $a=2\b\sqrt{q}$ and $b=2h$. By  \eqv(6.lem3.1'')
\bea
r(\b,h) 
\hspace{-6pt}&=&\hspace{-6pt}
\frac{a^2}{4h}\left(1-q\right)
+\frac{1}{h}\Bigl[(b-a^2)\left(t^-_{a,b}+t^+_{a,b}\right)-2bt^+_{a,b}
+a\sqrt{\sfrac{2}{\pi}}e^{-\frac{1}{2}\left(\frac{b}{a}\right)^2}\Bigr]
\Eq(6.lem5.3a)
\\
\hspace{-6pt}&\leq &\hspace{-6pt}
2Ee^{-|aZ+b|}+\frac{4a}{b\sqrt{2\pi}}e^{-\frac{1}{2}\left(\frac{b}{a}\right)^2}
\Eq(6.lem5.3)
\eea
where the last inequality follows from Lemma \thv(6.lem6) and \eqv(6.lem3.1'). We next use \eqv(6.lem3.1') and \eqv(6.lem3.0') to bound $Ee^{-|aZ+b|}$, distinguishing two cases.  Assume first that $-\frac{b}{a}+a\leq 0$. Then
\bea
\hspace{-6pt}&&\hspace{-6pt}
r(\b,h)
\nonumber
\\
\hspace{-6pt}&\leq&\hspace{-6pt}
2\left\{
e^{\sfrac{a^2}{2}-b}\left[1-\sfrac{1}{2}\erfc\left(\sfrac{b}{a}-a\right)\right]
+
\sfrac{1}{2}e^{\frac{a^2}{2}+b}\erfc\left(\sfrac{b}{a}+a\right)
\right\}
+\sfrac{4a}{b\sqrt{2\pi}}e^{-\frac{1}{2}\left(\frac{b}{a}\right)^2}
\Eq(6.lem5.4)
\\
\hspace{-6pt}&\leq&\hspace{-6pt}
\Bigl[2+
\sqrt{\sfrac{2}{\pi}}
C^+\left(\sfrac{b}{a}+a\right)
+\sfrac{4a}{b\sqrt{2\pi}}
\Bigr]
e^{-\sfrac{b}{2}}
\Eq(6.lem5.5)
\eea
where we used \eqv(6.lem3.0') in the last inequality together with the following two facts:
$\sfrac{a^2}{2}-b\leq -\frac{b}{2}$ (which follows from the assumption that $-\frac{b}{a}+a\leq 0$) and 
$
\sfrac{a^2}{2}-b\geq-\frac{1}{2}\left(\frac{b}{a}\right)^2
$.
The first inequality of \eqv(6.lem5.1) then  follows from \eqv(6.lem5.5) and the  bound $\sqrt{{2}/{\pi}}C^+(z)\leq 1$ $\forall z\geq 0$. Assume now that $-\frac{b}{a}+a>0$.
\bea
r(\b,h)
\hspace{-6pt}&\leq&\hspace{-6pt}
2\left[\sfrac{1}{2}e^{\frac{a^2}{2}-b}\erfc\left(-\sfrac{b}{a}+a\right)
+
\sfrac{1}{2}e^{\frac{a^2}{2}+b}\erfc\left(\sfrac{b}{a}+a\right)
\right]
+
\sfrac{4a}{b\sqrt{2\pi}}e^{-\frac{1}{2}\left(\frac{b}{a}\right)^2}
\Eq(6.lem5.7)
\\
\hspace{-6pt}&\leq&\hspace{-6pt}
\Bigl\{
\sqrt{\sfrac{2}{\pi}}\left[
C^+\left(-\sfrac{b}{a}+a\right)
+
C^+\left(\sfrac{b}{a}+a\right)
\right]
+\sfrac{4a}{b\sqrt{2\pi}}
\Bigr\}
e^{-\frac{1}{2}\left(\frac{b}{a}\right)^2}
\Eq(6.lem5.8)
\\
\hspace{-6pt}&\leq&\hspace{-6pt}
\Bigl[
1
+
\sqrt{\sfrac{2}{\pi}}\left(\sfrac{b}{a}+a\right)^{-1}
+
\sqrt{\sfrac{2}{\pi}}\sfrac{2a}{b}
\Bigr]
e^{-\frac{1}{2}\left(\frac{b}{a}\right)^2}
\Eq(6.lem5.9)
\eea
where we used that $\sqrt{{2}/{\pi}}C^+(z)\leq 1$ and $C^+(z)\leq z^{-1}$ to bound, respectively, the first and second occurence of this function in \eqv(6.lem5.8). Since $\left(\sfrac{b}{a}+a\right)^{-1}\leq \sfrac{a}{b}$, the second inequality of \eqv(6.lem5.1) follows. The upper bound on $r(\b,h)$ of Lemma \thv(6.lem8) is proved in the same way but we now bound both occurrences of $C^+(z)$ in  \eqv(6.lem5.8) by $C^+(z)\leq z^{-1}$, namely,
\be
C^+\left(-\sfrac{b}{a}+a\right)
+
C^+\left(\sfrac{b}{a}+a\right)
\leq 
\sfrac{2a}{a^2-\left(\frac{b}{a}\right)^2}
\leq
\sfrac{2}{a(1-\eta^2)}
\leq \sfrac{2\eta}{1-\eta^2}\sfrac{a}{b}
\Eq(6.lem5.10)
\ee 
where the last two inequalities follow from the assumption that $b\leq \eta a^2$. To prove the associated lower bound we go back to \eqv(6.lem5.3a) (equivalently, to  \eqv(6.theo1.2)). By  \eqv(6.lem3.1''), the observation that $t^-_{a,b}-t^+_{a,b}\geq 0$ and  the lower bound on $(1-q)$ of Lemma \thv(6.lem6), we have
\bea
r(\b,h)
\hspace{-6pt}&\geq&\hspace{-6pt}
h^{-1}\Bigl[-\sfrac{3}{4}a^2Ee^{-|aZ+b|}
+
a
\sqrt{\sfrac{2}{\pi}}e^{-\frac{1}{2}\left(\frac{b}{a}\right)^2}\Bigr].
\Eq(6.lem5.11)
\eea
The lower bound of \eqv(6.lem8.1) now follows from  the upper bound on $Ee^{-|aZ+b|}$ established in the proof of the upper bound  of \eqv(6.lem8.1) using \eqv(6.lem5.10). However, here, we do not use the last inequality of \eqv(6.lem5.10) but only but the one before last.

We now turn to \eqv(6.lem8.2). In view of the upper bound of Lemma \thv(6.lem6), $1-q$ is bounded above by twice the first term in the second line of \eqv(6.lem5.3). Based on this observation, the proof of the upper bound on $1-q$ is a by-product of the proof of the upper bound on $r(\b,h)$ (note that here again, we do not use the last inequality in \eqv(6.lem5.10) but the one before last). To prove the associated lower bound we write, combining the lower bound of Lemma \thv(6.lem6), \eqv(6.lem3.1'), the lower bound \eqv(6.lem3.0') and the assumption that $b\leq \eta a^2$,
\bea
1-q
\geq 
\sfrac{1}{2}e^{\frac{a^2}{2}+b}\erfc\left(\sfrac{b}{a}+a\right)
\geq
\sfrac{1}{2}\sqrt{\sfrac{2}{\pi}}
C^-\left(\sfrac{b}{a}+a\right)
e^{-\frac{1}{2}\left(\frac{b}{a}\right)^2}
\geq
\sfrac{e^{-\frac{1}{2}\left(\frac{b}{a}\right)^2}}{\sqrt{2\pi}\sqrt{a^2(1+\eta)^2+4}}.
\eea
Lemma \thv(6.lem5) and Lemma \thv(6.lem8) are proved.
\end{proof} 

\begin{proof}[Proof of Proposition \thv(main.prop1)] Let us establish that 
$\wt\DD^{(2)}\subset\DD^{(2)}_{\bar\varrho(\b,h)}\cap\DD^{(3)}\cap\DD^{(4)}$.
The condition $3\leq  {h}/{\b}\leq\b q/10$ implies that ${h}/{\b}\geq 3$ and $\b \geq 30/q$
so that $h=({h}/{\b})\b\geq 90/q\geq 90$. It also implies that $10 h\leq \b^2q$. Thus,
the assumptions of Lemma \thv(6.lem8) are satisfied with $\eta=1/20$. 
This key lemma is used throughout the proof.
It first guarantees that $\sqrt{3/4}\leq q\leq 1$. Indeed, using that $\b\sqrt{q}\geq 10h/(\b\sqrt{q})$ and setting 
$f(x)=x^{-1}e^{-x^2/2}$, we deduce from the upper bound of \eqv(6.lem8.2) that
\be
1-q
\leq  \sfrac{2}{10(1-\eta^2)}\sqrt{\sfrac{2}{\pi}}f(h/(\b\sqrt{q}))
\leq \sfrac{2}{10(1-\eta^2)}\sqrt{\sfrac{2}{\pi}}f(3)
\leq 1-\sqrt{3/4},
\Eq(main.prop1.4)
\ee
where we used that $f$ is strictly decreasing on $\R^+$ and that $h/(\b\sqrt{q})\geq h/\b\geq 3$.
Next, it guarantees that $\b^2(1-q)<1$. Indeed, using again  the upper bound of \eqv(6.lem8.2) 
and the fact just established that $q\geq  \sqrt{3/4}$, we get
\be
\b^2(1-q)\leq \sfrac{4}{1-\eta^2}\sqrt{\sfrac{2}{3\pi}}\b e^{-\frac{1}{2}(h/\b)^2}\leq 12\b e^{-\frac{1}{9}(h/\b)^2}<1,
\Eq(main.prop1.5)
\ee
where the last inequality is the first condition in the definition of $\wt\DD^{(2)}$. 
So far, we have established that $\wt\DD^{(2)}\subseteq\DD^{(3)}$.
Let us now check that $\wt\DD^{(2)}\subseteq\DD^{(4)}$. 
Comparing the prefactors of the lower bound of \eqv(6.lem8.1) and of the upper bound of \eqv(6.lem8.2),
we see that $1-\left(1+\sfrac{1}{h-1}\right)r(\b,h)<q$ on $\wt\DD^{(2)}$.
Thus, by \eqv(6.theo1.3), $\bar\varrho(\b,h)=\left[1-\left(1+\sfrac{1}{h-1}\right)r(\b,h)\right]^2$. 
By the lower bound  of \eqv(6.lem8.1) this implies that
$\bar\varrho(\b,h)>\sqrt{3/4}$ if (for $f$ defined as in \eqv(main.prop1.4))
\be
f(h/(\b\sqrt{q}))
\leq
\left( 1-(3/4)^{1/4}\right)/\left(2\left(1+\sfrac{1}{h-1}\right)\left(1+\sfrac{\eta}{1-\eta^2}\right)\sqrt{{2}/{\pi}}\right),
\Eq(main.prop1.6)
\ee
which is satisfied on $\wt\DD^{(2)}$. Thus, $\wt\DD^{(2)}\subseteq\DD^{(4)}$. 
Using Lemma \thv(6.lem8) once more, one proves that $\b\vartheta(\bar\varrho(\b,h))<1$ 
if the first condition in the definition of  $\wt\DD^{(2)}$ is satisfied. 
Hence,  $\wt\DD^{(2)}\subset\DD^{(2)}_{\bar\varrho(\b,h)}$. 
We skip the elementary details.

It remains to check that $\wt\DD^{(2)}$ is contained in the AT-region. Using that $\cosh^{-4}(x)\leq 16e^{-4|x|}$ and proceeding as in the proof of the upper bound of \eqv(6.lem8.2), we readily get that  if $h\leq 4\eta\b^2q$,  choosing $\eta=1/40$, the AT-region contains the region $4\b e^{-\frac{1}{2}(h/\b)^2}<1$, which itself contains $12\b e^{-\frac{1}{9}(h/\b)^2}<1$. The proof of the proposition is complete.
\end{proof}

\begin{proof}[Proofs of Theorem \thv(main.theo6) and Theorem \thv(main.theo2)] 
Let $\DD$ be given by \eqv(main.theo6.1). 
Combining Theorem \thv(main.theo3), Theorem \thv(main.theo5) and  \eqv(1.theo1.1) from Theorem \thv(main.theo4)  proves Theorem \thv(main.theo6) and,  as a consequence, Theorem \thv(main.theo2). The claim that with $\P$-probability one, for all large enough $N$, $F_{N,\b,h}^{HT}$ has a unique global maximum over $[-1,1]^N$, follows from Lemma \thv(S5.6.lem1) and the fact that its global maximum is achieved in $\BB^c_{N,\e}(\varrho)$.
\end{proof}
	

\section{An integral representation formula}
    \TH(S7)

In this section, an integral representation of the partition function \eqv(1.2) derived from the Hubbard-Stratonovitch transformation \cite{Hub59, Str57} is used to prove Theorem  \thv(main.theo1). This transformation has proved to be a useful tool for identifying the free energy functionals of mean-field models, from the early work of Kac  \cite{Kac69} on the Curie-Weiss model to the more recent analysis of the Hopfield model \cite{BGP, BG97, BG98b, BG98a}. 
Here, it allows us to identify the function $F_{N,\b,h}^{HT}$  in \eqv(7.31bis)  as the free energy functional of the SK model at high temperature

Consider the matrix $A\equiv A_N$ defined in \eqv(2.1). Since $M$ is real symmetric, there exists an orthogonal matrix $\OO$ and a diagonal matrix $\L=\diag(\l_1,\l_2,\dots,\l_N)$ such that $A=\OO^t\L\OO$. Let $\sqrt{A}$ denote the matrix $\sqrt{A}\equiv \OO^t\sqrt{\L}\OO$ where $\sqrt{\L}\equiv \diag(\sqrt{\l_1},\sqrt{\l_2},\dots,\sqrt{\l_N})$ with the  convention that if $\a$ is a real number,
\be
\sqrt{\a}=\zeta\sqrt{|\a|}
\quad 
\text{where}
\quad
\zeta
=
\begin{cases} 
i&\mbox{if}\,\, \a<0,\\
1&\mbox{if}\,\, \a\geq 0,
\Eq(7.0)
\end{cases}
\ee
where $i$ is the unit imaginary number. Thus $\sqrt{A}$ is a complex symmetric matrix 
that satisfies $\sqrt{A}\sqrt{A}=A$. Let $\FF_{N,\b,h}: \C^N\mapsto\C$ be the function 
\be
\FF_{N,\b,h}(x)=-\frac{1}{2}\sum_{j=1}^{N}x_j^2+\sum_{j=1}^{N}\log\cosh\left(\sqrt{\b}(\sqrt{A} x)_j+h\right)+N\log 2
\Eq(7.28)
\ee
where the log function is defined in the principal branch (it is therefore continuous in any open set in the complex plane from which the negative axes and zero have been removed). 

\begin{lemma}
     \TH(7.lem1)
\be
Z_{N,\b,h}=
e^{\frac{1}{2}N\b^2(1-q)}\int_{\R^N}\left(\prod_{j=1}^{N}\frac{dx_j}{\sqrt{2\pi}}\right)\, e^{\FF_{N,\b,h}(x)}.
\Eq(7.lem1.1)
\ee
\end{lemma}

To evaluate such an integral, one usually starts by looking for a critical point that maximises the real part of $\FF_{N,\b,h}$ and, if this point is unique, one tries to deform the integration path so that it passes through this point in the direction of steepest descent (equivalently, a direction of constant phase). On such a contour, in the vicinity of the critical point, the integral should resemble a Laplace integral. Although our attempts to compute this integral failed, part of this programme can be carried out as we now explain.

Our first lemma links the critical points  of $\FF_{N,\b,h}$ to the solutions of the TAP equations.
Recall the definitions of $\Psi_{N,\b,{\bf{h}}}$ and $\Phi_{N,\b,{\bf{h}}}$  from \eqv(2.2). Observe that for all  $z\in\R^N$,
\be
\FF_{N,\b,h}\bigl(\sqrt{\b}\sqrt{A}z\bigr)=\Phi_{N,\b,{\bf{h}}}(z).
\Eq(7.lem2.2)
\ee 
    
\begin{lemma}
     \TH(7.lem2)
If $z\in\R^N$ is a solution of the TAP equations \eqv(1.11) satisfying $q_{\text{EA}}(z)=q$, 
then $x=\sqrt{\b}\sqrt{A}z\in\C^N$ is a critical point of $\FF_{N,\b,h}$, and for each such pair of critical points \eqv(7.lem2.2) holds.
\end{lemma}

Points of the form $x=\sqrt{\b}\sqrt{A}z$, $z\in\R^N$, are in general complex. 
Thus the prospective maximiser of $\Re(\FF_{N,\b,h})$ is in $\C^N$ while the integration contour
in \eqv(7.lem1.1) is the collection of the $N$ real axes. 
The next lemma state that we can shift the contour so that it passes through a given point in $\C^N$. 

\begin{lemma}
     \TH(7.lem3)
The following holds with $\P$-probability one for sufficiently enough $N$. For all $x^*=a+ib$, $a,b\in \R^N$ such that $\|b\|_{\infty}<\infty$ 
\be
Z_{N,\b,h}=
e^{\frac{1}{2}N\b^2(1-q)}\int_{\R^N}\left(\prod_{j=1}^{N}\frac{dx_j}{\sqrt{2\pi}}\right)\, e^{\FF_{N,\b,h}(x+x^*)}.
\Eq(7.10)
\ee
\end{lemma}

We next must choose the shift in Lemma \thv(7.lem3). Given $z\in\R^N$, set $x^*(z)=\sqrt{\b}\sqrt{A} z$ and let $\RR_{N,\b,{\bf{h}}}: \R^N\mapsto \R$ be the function 
\be
\RR_{N,\b,{\bf{h}}}(z)
=
\frac{1}{N}\log
\left(
\int_{\R^N}\left(\textstyle{\prod_{j=1}^{N}\frac{dx_j}{\sqrt{2\pi}}}\right)\, e^{\FF_{N,\b,h}(x+x^*(z))-\FF_{N,\b,h}(x^*(z))}
\right).
\Eq(7.prop1)
\ee

\begin{proposition}
     \TH(7.prop1bis)
With $\P$-probability one, for all  $N$ large enough, the following holds: 
\item{(i)} For any $z\in\R^N$ such that $\|z\|_{\infty}<\infty$,
\be
\frac{1}{N}\log Z_{N,\b,h}=\Phi_{N,\b,{\bf{h}}}(z)
+
\frac{1}{2}\b^2(1-q)+\RR_{N,\b,{\bf{h}}}(z).
\Eq(7.prop1bis.0)
\ee
\item{(ii)}
If $\Psi_{N,\b,{\bf{h}}}$ attains its global maximum uniquely at a point $z^*$ lying in $(-1,1)^N$, then
\be
\Phi_{N,\b,{\bf{h}}}(z^*)=\sup_{z\in[-1,1]^N}\Psi_{N,\b,{\bf{h}}}(z).
\Eq(7.prop1bis.1)
\ee
Moreover, $\Psi_{N,\b,{\bf{h}}}$ cannot be replaced with $\Phi_{N,\b,{\bf{h}}}$ in  \eqv(7.prop1bis.1).
\end{proposition}

\begin{remark}
One can prove that under the assumptions and with the notation of Proposition \thv(7.prop1bis), (ii), the supremum of
$\Re\bigl(\FF_{N,\b,h}(x+x^*(z^*))\bigr)$ over the integration contour $\R^N$ is attained uniquely at $x=0$. Hence, on the integration contour, the exponent in \eqv(7.prop1) has a unique saddle point at $x=0$ and its real part is strictly negative away from this point. (To limit the length of this paper, we refrain from giving the proof of this result.) While it would be unwise to draw too close a parallel with the classical setting (where the dimension of $\FF$ does not diverge with the asymptotic parameter), we note that the above properties would typically put us in a position to apply Laplace method.
\end{remark}

\begin{remark} More generally, Lemma \thv(7.lem2) holds for $z\in\C^N$. In this case the TAP equations become a system of equations in $\C^N$ and the function $\Phi_{N,\b,{\bf{h}}}$ also takes values in $\C^N$. While we could not rule out the existence of such critical points, they are of no interest to us since a complex $\Phi_{N,\b,{\bf{h}}}$ in \eqv(7.prop1bis.0) would not lead to a meaningful representation of $Z_N$.
\end{remark}

We first prove Theorem  \thv(main.theo1), assuming the above results. Then we successively prove  Lemma \thv(7.lem1), Lemma \thv(7.lem2), Proposition \thv(7.prop1bis) and Lemma \thv(7.lem3) in this order.

\begin{proof}[Proof of Theorem  \thv(main.theo1)]  
Taking  $z=m^{(k)}$  in item (i) of Proposition \thv(7.prop1bis), it follows from \eqv(7.prop1bis.0) that on a set of full measure, for all large enough $N$
\be
\frac{1}{N}\log Z_{N,\b,h}=\frac{1}{N}\Phi_{N,\b,h{\bf{1}}}\left(m^{(k)}\right)
+
\frac{1}{2}\b^2(1-q)+\RR_{N,\b,h{\bf{1}}}\left(m^{(k)}\right).
\Eq(7.prop1bis.6)
\ee
We know from Section \thv(S3) that $m^{(k)}$ is a near solution of the system of specialised TAP equations \eqv(1.12) and that the functions $\Phi_{N,\b,h{\bf{1}}}$ and $\Psi_{N,\b,h{\bf{1}}}$ can be modified according to the strategy of Section \thv(S2.2) (see \eqv(2.4)) to make $m^{(k)}$ an exact critical point for which the duality formula holds (see Lemma \thv(2.lem1)).
We will not repeat the details of this argument which we have used many times before (see, e.g., \eqv(4.lem4.2)-\eqv(4.lem4.5) in the proof Lemma \thv(4.lem4) and the proof of Theorem \thv(main.theo3) after the proof of Lemma \thv(S5.6.lem1)).
Proceeding in this way, we obtain that for all  $(\b,h)$ in the AT-region,
\be
\lim_{k\rightarrow\infty}\lim_{N\rightarrow\infty}\left|\frac{1}{N}\Phi_{N,\b,h{\bf{1}}}\left(m^{(k)}\right)-\frac{1}{N}\Psi_{N,\b,h{\bf{1}}}\left(m^{(k)}\right)\right|
= 0\quad \P-\text{a.s.}
\Eq(7.prop1bis.9)
\ee
We also know by Theorem \thv(main.theo6) that for all $(\b,h)$ in the AT-region intersected with the region $\DD$ defined by \eqv(main.theo6.1), $ F_{N,\b,h}^{HT}\left(m^{(k)}\right)$ converges $\P$-almost surely to the global maximum of $F_{N,\b,h}^{HT}$. Hence, by \eqv(2.cor1.2bis), the same holds true for the function $\frac{1}{N}\Psi_{N,\b,h{\bf{1}}}$, and so,
\be
\lim_{k\rightarrow\infty}\lim_{N\rightarrow\infty}\left|
\frac{1}{N}\Psi_{N,\b,h{\bf{1}}}\left(m^{(k)}\right)-\sup_{z\in[-1,1]^N}\frac{1}{N}\Psi_{N,\b,h{\bf{1}}}(z)
\right|
=0\quad \P-\text{a.s.}
\Eq(7.prop1bis.2)
\ee
Thus, using \eqv(7.prop1bis.9), \eqv(7.prop1bis.2) and \eqv(2.cor1.2bis), for all $(\b,h)$ in the intersection of the AT-region and the region $\DD$, \eqv(7.prop1bis.6)  can be written as
\be
\frac{1}{N}\log Z_{N,\b,h}
=\left\{\sup_{x\in [-1,1]^N}F_{N,\b,h}^{HT}(x)\right\}
+\frac{\b^2}{4}\left(1-q\right)^2
+\RR_{N,\b,h{\bf{1}}}\left(m^{(k)}\right)
+r_{k,N}(\b,h)
\Eq(7.theo1.8)
\ee
where $r_{k,N}(\b,h)$ satisfies $\lim_{k\rightarrow\infty}\lim_{N\rightarrow\infty}r_{k,N}(\b,h)=0$ $\P$-almost surely.

As already mentioned, we have not been able to work out the term $\RR_{N,\b,h{\bf{1}}}\left(m^{(k)}\right)$
by direct methods. With additional information, however, we can identify its limits. More precisely, 
we know from \eqv(1.4) that for all $(\b,h)$ in the high-temperature region of Definition \thv(1.def1)
 \be
\lim_{N\rightarrow\infty} \frac{1}{N}\log Z_{N,\b,h}
=SK(\b,h)\quad \P\textstyle{-a.s.}.
\Eq(7.prop1bis.4)
\ee
On the other hand, by Lemma \thv(4.lem3), for all $(\b,h)$ in the AT-region,
\be
\lim_{k\rightarrow\infty}\lim_{N\rightarrow\infty}
\frac{1}{N}\Phi_{N,\b,h{\bf{1}}}\left(m^{(k)}\right)
=SK(\b,h)-\frac{\b^2}{4}\left(1-q^2\right)\quad \P\textstyle{-a.s.}
\Eq(7.prop1bis.5)
\ee
Since  convergence in \eqv(7.prop1bis.4) and \eqv(7.prop1bis.5) holds $\P$-almost surely, it follows from \eqv(7.prop1bis.6) that
in the high-temperature region 
\be
\lim_{k\rightarrow\infty}\lim_{N\rightarrow\infty}\RR_{N,\b,h{\bf{1}}}\left(m^{(k)}\right)=-\frac{\b^2}{4}\left(1-q\right)^2\quad \P\textstyle{-a.s.}
\Eq(7.prop1bis.7)
\ee
It should be remembered here that, although it is generally accepted in the physics literature that the high-temperature region of the SK model coincides with the AT-region, from a rigorous point of view it is only known that the former is a subregion of the latter  (see the discussion below \eqv(1.7)). Thus, combining \eqv(7.prop1bis.7) and \eqv(7.theo1.8) proves \eqv(main.theo1.1) in the intersection of $\DD$ and the high-temperature region. 

To see see that \eqv(main.theo1.1) holds in $\P$-probability when replacing the assumption of almost sure convergence in \eqv(7.prop1bis.2)  by convergence in  $\P$-probability, simply recall  that almost sure convergence implies convergence in probability, and that if two sequences converge in probability, then so does their sum.
\end{proof}
 
\begin{proof}[Proof of Lemma \thv(7.lem1)]  
By the definition and notation \eqv(1.1) and \eqv(1.2) of Section \thv(S1.1), 
\be
Z_{N,\b,h}=\sum_{\s\in\S_N}e^{\frac{1}{2}N\b^2(1-q)}e^{ \frac{\b}{2}\left(\s,A \s\right)+h(1,\s)}.
\Eq(7.1)
\ee
Next, using the identity $A=\OO^t\L\OO$ to express the quadratic form in \eqv(7.1) yields
\be
e^{ \frac{\b}{2}\left(\s,A \s\right)}=\prod_{j=1}^{N}e^{\frac{\b\l_j}{2}(\OO \s)_j^2}.
\Eq(7.3)
\ee
Eq.~\eqv(7.3) can be further expressed applying the Hubbard-Stratonovich transformation which, for any pair of real numbers $\a$ and $y$, is defined through the identity
\be
e^{\a\frac{y^2}{2}}
=\int_{-\infty}^{+\infty}\frac{dx}{\sqrt{2\pi}}e^{-\frac{x^2}{2}-\sqrt{\a} xy}, 
\Eq(7.4)
\ee
where $\sqrt{\a}$ is as in \eqv(7.0). To check \eqv(7.4), simply note that for $\a<0$, it is the Fourier transform of a Gaussian density while for $\a\geq 0$, it is its two-sided Laplace transform. Eq.~\eqv(7.3) can thus be rewritten as
\be
e^{ \frac{\b}{2}\left(\s,A \s\right)}
=
\prod_{j=1}^{N}\int_{-\infty}^{+\infty}\frac{dx_j}{\sqrt{2\pi}}e^{-\frac{1}{2}x_j^2-\sqrt{\l_j} (\OO \s)_jx_j}.
\Eq(7.5)
\ee
Next observe that
\be
\textstyle
\sum_{j=1}^{N}\sqrt{\l_j} (\OO \s)_jx_j
=\sum_{j'=1}^{N}\s_{j'}(\OO^t\sqrt{\L} x)_{j'}
=\sum_{j'=1}^{N}\s_{j'}(\sqrt{A} \OO^tx)_{j'}.
\Eq(7.6)
\ee
By this and the change of variable $x \mapsto \OO x$, \eqv(7.5) becomes
\be
e^{ \frac{\b}{2}\left(\s,A \s\right)}
=
\int_{-\infty}^{+\infty}\frac{dx_1}{\sqrt{2\pi}}\dots\int_{-\infty}^{+\infty}\frac{dx_N}{\sqrt{2\pi}} 
e^{-\frac{1}{2}\sum_{j=1}^{N}x_j^2-\sum_{j=1}^{N}\s_j(\sqrt{A} x)_j}.
\Eq(7.7)
\ee
Inserting \eqv(7.7) in \eqv(7.1), the summation in $\s$ is easily carried out and we get
\be
Z_{N,\b,h}=
e^{\frac{1}{2}N\b^2(1-q)}\int_{\R^N}\left(\prod_{j=1}^{N}\frac{dx_j}{\sqrt{2\pi}}\right)\, \Xi_{N,\b,h}(x),
\Eq(7.8)
\ee
where $\Xi_{N,\b,h}: \R^N\rightarrow\C$ is defined by
\be
\Xi_{N,\b,h}(x)=2^Ne^{ -\frac{1}{2}\sum_{j=1}^{N}x_j^2}\prod_{j=1}^{N}\cosh\left(\sqrt{\b}(\sqrt{A} x)_j+h\right).
\Eq(7.9)
\ee
We can  view the function $\Xi_{N,\b,h}$ as the restriction to the hyperplane $\R^N$ of a holomorphic function defined on the whole of $\C^N$. Eq.~\eqv(7.lem1.1) then follows from \eqv(7.9) and the identity $z=e^{\log z}$, $z\neq 0$, where the log function is defined in the principal branch. 
\end{proof}

\begin{proof}[Proof of Lemma \thv(7.lem2)] The claim of the lemma follows in a straightforward way by differentiation of 
$\FF_{N,\b,h}$.
\end{proof}

\begin{proof}[Proof of Proposition \thv(7.prop1bis)]
Choosing $x^*=x^*(z)$ in  Lemma \thv(7.lem3), item (i) follows from Lemma \thv(7.lem1), the identity \eqv(7.lem2.2) and the definition \eqv(7.prop1).
As seen in the proof of Proposition \thv(2.prop1), the assumption in item (ii) that $\Psi_{N,\b,{\bf{h}}}$ attains its global maximum at a point $z^*$ in $\rm{int}(\rm{dom\,} \Psi_{N,\b,{\bf{h}}})=(-1,1)^N$ guarantees that $z^*$ is a critical point of $\Psi_{N,\b,{\bf{h}}}(z)$. By Proposition \thv(2.prop1), $z^*$  is also a critical point of  $\Phi_{N,\b,{\bf{h}}}$ and by \eqv(2.prop1.2), $\Phi_{N,\b,{\bf{h}}}(z^*)=\Psi_{N,\b,{\bf{h}}}(z^*)$. This proves \eqv(7.prop1bis.1). As explained in the remark at the end of Section \thv(S2.1), below \eqv(2.cor1.2bis), $\Phi_{N,\b,{\bf{h}}}$ is unbounded. Thus, $\Psi_{N,\b,{\bf{h}}}$ cannot be replaced with $\Phi_{N,\b,{\bf{h}}}$ in the right-hand side of \eqv(7.prop1bis.1).
\end{proof}

\begin{proof}[Proof of Lemma \thv(7.lem3)] 
The probabilistic part of the statement of the lemma serves to guarantee, as Theorem \thv(4.theo1) permits, that
\be
\max_{j}|\l_j|\leq 2(1+o(1))+\b|1-q|\equiv c_{q}.
\Eq(7.10')
\ee
From now on, we place ourselves on the set of $\P$-probability one for which Theorem  \thv(4.theo1) is obtained, and  assume that $N$ is sufficiently large for \eqv(7.10') to be satisfied.

Clearly, by the change of variable $x+a\mapsto x$, it suffices to prove \eqv(7.10) for pure imaginary vectors $x^*=ib$. Since  $\Xi_{N,\b,h}(x)$ is holomorphic and continuous, it follows from Osgood's lemma that it is holomorphic in each variable separately.
We can thus apply the one variable Cauchy-Goursat integral formula to each of the variables $x_j$, $j=1,2,\dots,N$, successively.

We begin with the variable $x_1$. Let $\theta_1(x_1): \C\mapsto\C$ be the function defined by integrating $\Xi_{N,\b,h}$ over all variables except $x_1$, which is kept fixed
\be
\theta_1(x_1)=\int_{\R^{N-1}}\left(\prod_{j=2}^{N}\frac{dx_j}{\sqrt{2\pi}}\right)\, \Xi_{N,\b,h}(x).
\Eq(7.11)
\ee
According to the Cauchy-Goursat integral formula
\be
\int_{\CC_1}\frac{dx_1}{\sqrt{2\pi}}\theta_1(x_1)=0
\Eq(7.12)
\ee
where, writing $x_1=u_1+iv_1$, $\CC_1$ is the rectangular closed path  in the plane $(u_1,v_1)$ defined as the boundary of the rectangle of vertices $A=(-R_N,0)$, $B=(R_N,0)$, $C=(R_N,b_1)$, and $D(-R_N,b_1)$, oriented counter-clockwise. Here $R_N=NR$, $R>0$, and  $b_1>0$ so that the rectangle lies in the upper half-plane (the case $b_1<0$, which corresponds to a rectangle in the lower half-plane, is treated in the same way). The left-hand side of \eqv(7.12) naturally decomposes into four integrals, each of them along the path that follows a given side of the rectangle
\be
\int_{\CC_1}\frac{dx_1}{\sqrt{2\pi}}\theta_1(x_1)=\II_{AB}+\II_{BC}+\II_{CD}+\II_{DA}.
\Eq(7.13)
\ee
Let us establish that
\be
\lim_{R\rightarrow\infty}\II_{BC}\equiv\int_{R_N+i0}^{R_N+ib_1}\frac{dx_1}{\sqrt{2\pi}}\theta_1(x_1)= 0.
\Eq(7.14)
\ee
By a change of variables
\be
\II_{BC}
=
2^N\int_{0}^{b_1}\frac{dv_1}{\sqrt{2\pi}}e^{ -\frac{1}{2}(R_N+iv_1)^2}
\int_{\R^{N-1}}\left(\prod_{j=2}^{N}\frac{du_j}{\sqrt{2\pi}}\right)
e^{ -\frac{1}{2}\sum_{j=2}^{N}u_j^2}\Theta(u+iv)
\Eq(7.15)
\ee
where $u=(R_N,u_2,\dots,u_N)$ and $v=(v_1,0,\dots,0)$ are vectors in $\R^N$ and
\be
\Theta(u+iv)=\prod_{j=1}^{N}\cosh\left(\sqrt{\b}(\sqrt{A} (u+iv))_j+h\right).
\Eq(7.16)
\ee
Taking the modulus,
\be
\left|\II_{BC}\right|
\leq
2^N
\int_{0}^{b_1}\frac{dv_1}{\sqrt{2\pi}}e^{-\frac{1}{2}(R^2_N-v_1^2)}\int_{\R^{N-1}}\left(\prod_{j=2}^{N}\frac{du_j}{\sqrt{2\pi}}\right)
e^{ -\frac{1}{2}\sum_{j=2}^{N}u_j^2}\left|\Theta(u+iv)\right|.
\Eq(7.17)
\ee
To bound $\left|\Theta(u+iv)\right|$, first note that
\be
\nonumber
\left|\cosh(u_0+iv_0)\right|
=
\left| \cosh(u_0)\right|
\sqrt{\cos^2(v_0)+ \tanh^2(u_0)\sin^2(v_0)}
\leq \left| \cosh(u_0)\right|
\leq e^{\left|u_0\right|}.
\ee
Now, writing $\sqrt{A}= U+iV$ where $U=\OO^t\Re(\sqrt{\L})\OO$ and $V=\OO^t\Im(\sqrt{\L})\OO$, and using the above bound
\be
\left|\Theta(u+iv)\right|
\leq 
e^{\sqrt{\b}\sum_{j=1}^{N}\left|(Uu)_j-(Vv)_j\right|+Nh}.
\Eq(7.19)
\ee
Observing that $UV=VU=0$,
\bea
\textstyle\sum_{j=1}^{N}\left|(Uu)_j-(Vv)_j\right|
&\leq& \textstyle\sqrt{\sum_{j=1}^{N}((Uu)_j-(Vv)_j)^2}
\nonumber
\\
&=&\textstyle\sqrt{(uU^2u)+(vV^2v)}
\Eq(7.20)
\\
&\leq&\textstyle\sqrt{\max_{j}|\l_j|(\left\|u\right\|_2^2+\left\|v\right\|_2^2)}
\leq  \textstyle\sqrt{\max_{j}|\l_j|}(\left\|u\right\|_1+\left\|v\right\|_1)
\nonumber
\eea
where $\max_{j}|\l_j|$ is bounded in \eqv(7.10') and, by definition of $u$ and $v$, $\left\|u\right\|_1=R_N+\sum_{j=2}^{N}|u_j|$ and $\left\|v\right\|_1=b_1$.
Thus, recalling the definition of $c_{q}$ from \eqv(7.10')
\be
\left|\Theta(u+iv)\right|\leq
e^{\sqrt{\b c_{q}}\left(R_N+\sum_{j=2}^{N}|u_j|+b_1\right)+Nh}.
\Eq(7.21)
\ee
Inserting this bound in \eqv(7.17),
\bea
\left|\II_{BC}\right|
&\leq&
2^N\int_{0}^{b_1}\frac{dv_1}{\sqrt{2\pi}}e^{-\frac{1}{2}(R^2_N-v_1^2) +\sqrt{\b c_{q}}(R_N+b_1)+Nh}
\prod_{j=2}^{N}\int_{\R}\frac{du_j}{\sqrt{2\pi}}e^{ -\frac{1}{2}u_j^2+\sqrt{\b c_{q}}|u_j|}
\nonumber
\\
&\leq&
\frac{2b_1}{\sqrt{2\pi}}e^{-\frac{1}{2}(R^2_N-b_1^2) +\sqrt{\b c_{q}}(R_N+b_1)+Nh}\left(4e^{\frac{1}{2}\b c_{q}}\right)^{N-1},
\Eq(7.22)
\eea
and taking the limit $R\rightarrow\infty$ of both sides of \eqv(7.22), we obtain \eqv(7.14).
We prove in exactly the same way that
\be
\lim_{R\rightarrow\infty}\II_{DA}\equiv\int_{-R_N+ib_1}^{-R_N+i0}\frac{dx_1}{\sqrt{2\pi}}\theta_1(x_1) =0.
\Eq(7.14bis)
\ee
Passing to the limit $R\rightarrow\infty$ in \eqv(7.13), it follows from \eqv(7.12), \eqv(7.14) and \eqv(7.14bis) that
\be
\int_{-\infty}^{\infty}\frac{dx_1}{\sqrt{2\pi}}\theta_1(x_1)
=
\int_{-\infty}^{\infty}\frac{dx_1}{\sqrt{2\pi}}\theta_1(x_1+ib_1).
\ee
By definition of  $\theta_1(x_1)$ (see \eqv(7.11)), setting $b^{(1)}=(b_1,0,\dots,0)$, this is equivalent to
\be
\int_{\R^{N}}\left(\prod_{j=1}^{N}\frac{dx_j}{\sqrt{2\pi}}\right)\, \Xi_{N,\b,h}(x)
=\int_{\R^{N}}\left(\prod_{j=1}^{N}\frac{dx_j}{\sqrt{2\pi}}\right)\, \Xi_{N,\b,h}(x+ib^{(1)}).
\Eq(7.23)
\ee

To deal with the next variable, $x_2$, we start from the right-hand side of \eqv(7.23) and, in complete analogy to \eqv(7.11),
we let $\theta_2(x_2): \C\rightarrow\C$ be the function defined by integrating $\Xi_{N,\b,h}$ over all variables except $x_2$, which is kept fixed
\be
\theta_2(x_2)=\int_{\R^{N-1}}\left(\prod_{1\leq j\leq N:j\neq 2}\frac{dx_j}{\sqrt{2\pi}}\right)\, \Xi_{N,\b,h}(x+b^{(1)}).
\Eq(7.24)
\ee
We then consider the Cauchy-Goursat integral formula
\be
\int_{\CC_2}\frac{dx_1}{\sqrt{2\pi}}\theta_2(x_2)=0
\Eq(7.25)
\ee
where $\CC_2$ is defined as $\CC_1$, replacing $b_1$ by $b_2$,
and use it, proceeding exactly as in the proof of \eqv(7.14) to prove that
\be
\int_{-\infty}^{\infty}\frac{dx_2}{\sqrt{2\pi}}\theta_2(x_2)
=
\int_{-\infty}^{\infty}\frac{dx_2}{\sqrt{2\pi}}\theta_2(x_2+ib_2).
\Eq(7.26)
\ee
We omit the details of the straightforward adaptation of the bounds \eqv(7.17)-\eqv(7.22).
Setting $b^{(2)}=(b_1,b_2,\dots,0)$, \eqv(7.26) is equivalent to
\be
\int_{\R^{N}}\left(\prod_{j=1}^{N}\frac{dx_j}{\sqrt{2\pi}}\right)\, \Xi_{N,\b,h}(x)
=\int_{\R^{N}}\left(\prod_{j=1}^{N}\frac{dx_j}{\sqrt{2\pi}}\right)\, \Xi_{N,\b,h}(x+ib^{(2)}).
\Eq(7.27)
\ee 
Iterating this procedure over the variable $x_j$, we obtain  \eqv(7.10). Lemma \thv(7.lem3) is proven.
\end{proof}

We conclude this section with the 
\begin{proof}[Proofs of Theorem \thv(main.theo7) and \thv(main.theo0)] 
Theorem \thv(main.theo7) follows from Theorem \thv(main.theo1) and Theorem \thv(main.theo6). Theorem \thv(main.theo0) follows from Theorem \thv(main.theo7).
\end{proof}



\def\cprime{$'$}



\begin{thebibliography}{10}


\bibitem{AbSt}
M.~Abramowitz and I.~A. Stegun.
\newblock {\em Handbook of mathematical functions with formulas, graphs, and
  mathematical tables}.
\newblock National Bureau of Standards Applied Mathematics Series, No. 55. U.
  S. Government Printing Office, Washington, D.C., 1964.
\newblock Tenth Printing, December 1972, with corrections.

\bibitem{ABSY21}
A.~Adhikari, C.~Brennecke, P.~von Soosten, and H.-T. Yau.
\newblock Dynamical approach to the {TAP} equations for the
  {S}herrington-{K}irkpatrick model.
\newblock {\em J. Stat. Phys.}, 183(3):27--66, 2021.

\bibitem{AdTa}
R.~J. Adler and J.~E. Taylor.
\newblock {\em Random fields and geometry}.
\newblock Springer Monographs in Mathematics. Springer, New York, 2007.

\bibitem{ALR}
M.~Aizenman, J.~L. Lebowitz, and D.~Ruelle.
\newblock Some rigorous results on the {S}herrington-{K}irkpatrick spin glass
  model.
\newblock {\em Comm. Math. Phys.}, 112(1):3--20, 1987.

\bibitem{ABM04}
T.~Aspelmeier, A.~J. Bray, and M.~A. Moore.
\newblock Complexity of {I}sing spin glasses.
\newblock {\em Phys. Rev. Lett.}, 92:087203, 2004.

\bibitem{BS10}
Z.~Bai and J.~W. Silverstein.
\newblock {\em Spectral analysis of large dimensional random matrices}.
\newblock Springer Series in Statistics. Springer, New York, second edition,
  2010.

\bibitem{BvH}
A.~S. Bandeira and R.~van Handel.
\newblock Sharp nonasymptotic bounds on the norm of random matrices with
  independent entries.
\newblock {\em Ann. Probab.}, 44(4):2479--2506, 2016.


\bibitem{Belius}
D.~Belius.
\newblock High temperature tap upper bound for the free energy of mean field
  spin glasses.
\newblock Preprint, 2022.
\newblock arXiv:2204.00681.

\bibitem{EB14}
E.~Bolthausen.
\newblock An iterative construction of solutions of the {TAP} equations for the
  {S}herrington-{K}irkpatrick model.
\newblock {\em Comm. Math. Phys.}, 325(1):333--366, 2014.

\bibitem{EB19}
E.~Bolthausen.
\newblock A {M}orita type proof of the replica-symmetric formula for {SK}.
\newblock In {\em Statistical mechanics of classical and disordered systems},
  volume 293 of {\em Springer Proc. Math. Stat.}, pages 63--93. Springer, Cham,
  2019.

\bibitem{BLM}
S.~Boucheron, G.~Lugosi, and P.~Massart.
\newblock {\em Concentration inequalities}.
\newblock Oxford University Press, Oxford, 2013.
\newblock A nonasymptotic theory of independence, With a foreword by Michel
  Ledoux.

\bibitem{BG97}
A.~Bovier and V.~Gayrard.
\newblock The retrieval phase of the {H}opfield model: a rigorous analysis of
  the overlap distribution.
\newblock {\em Probab. Theory Related Fields}, 107(1):61--98, 1997.

\bibitem{BG98b}
A.~Bovier and V.~Gayrard.
\newblock Hopfield models as generalized random mean field models.
\newblock In {\em Mathematical aspects of spin glasses and neural networks},
  volume~41 of {\em Progr. Probab.}, pages 3--89. Birkh\"{a}user Boston,
  Boston, MA, 1998.

\bibitem{BG98a}
A.~Bovier and V.~Gayrard.
\newblock Metastates in the {H}opfield model in the replica symmetric regime.
\newblock {\em Math. Phys. Anal. Geom.}, 1(2):107--144, 1998.

\bibitem{BGP}
A.~Bovier, V.~Gayrard, and P.~Picco.
\newblock Gibbs states of the {H}opfield model in the regime of perfect memory.
\newblock {\em Probab. Theory Related Fields}, 100(3):329--363, 1994.

\bibitem{BM79}
A.~J. Bray and M.~A. Moore.
\newblock Evidence for massless modes in the 'solvable model' of a spin glass.
\newblock {\em Journal of Physics C: Solid State Physics}, 12(11):L441, 1979.

\bibitem{CDM17}
M.~Capitaine and C.~Donati-Martin.
\newblock Spectrum of deformed random matrices and free probability.
\newblock In {\em Advanced topics in random matrices}, volume~53 of {\em Panor.
  Synth\`eses}, pages 151--190. Soc. Math. France, Paris, 2017.

\bibitem{CDMFF}
M.~Capitaine, C.~Donati-Martin, D.~F\'{e}ral, and M.~F\'{e}vrier.
\newblock Free convolution with a semicircular distribution and eigenvalues of
  spiked deformations of {W}igner matrices.
\newblock {\em Electron. J. Probab.}, 16(64):1750--1792, 2011.

\bibitem{CaPe}
M.~Capitaine and S.~P\'{e}ch\'{e}.
\newblock Fluctuations at the edges of the spectrum of the full rank deformed
  {GUE}.
\newblock {\em Probab. Theory Related Fields}, 165(1-2):117--161, 2016.

\bibitem{CGPM03}
A.~Cavagna, I.~Giardina, G.~Parisi, and M.~M{\'e}zard.
\newblock On the formal equivalence of the {TAP} and thermodynamic methods in
  the {SK} model.
\newblock {\em Journal of Physics A: Mathematical and General},
  36(5):1175--1194, 2003.

\bibitem{Chat}
S.~Chatterjee.
\newblock Spin glasses and {S}tein's method.
\newblock {\em Probab. Theory Related Fields}, 148(3-4):567--600, 2010.

\bibitem{CP}
W.-K. Chen and D.~Panchenko.
\newblock On the {TAP} free energy in the mixed {$p$}-spin models.
\newblock {\em Comm. Math. Phys.}, 362(1):219--252, 2018.

\bibitem{CPS2}
W.-K. Chen, D.~Panchenko, and E.~Subag.
\newblock The generalized {TAP} free energy {II}.
\newblock {\em Comm. Math. Phys.}, 381(1):257--291, 2021.

\bibitem{CPS1}
W.-K. Chen and E.~Subag.
\newblock Generalized {TAP} free energy.
\newblock {\em Comm. Pure Appl. Math.}, 76(7):1329--1415, 2023.

\bibitem{ChTa}
W.-K. Chen and S.~Tang.
\newblock On convergence of the cavity and {B}olthausen's {TAP} iterations to
  the local magnetization.
\newblock {\em Comm. Math. Phys.}, 386(2):1209--1242, 2021.

\bibitem{CLPR04}
A.~Crisanti, L.~Leuzzi, G.~Parisi, and T.~Rizzo.
\newblock Quenched computation of the dependence of complexity on the free
  energy in the {S}herrington-{K}irkpatrick model.
\newblock {\em Phys. Rev. B}, 70:064423, 2004.

\bibitem{AT78}
J.~R.~L. de~Almeida and D.~J. Thouless.
\newblock Stability of the {S}herrington-{K}irkpatrick solution of a spin glass
  model.
\newblock {\em Journal of Physics A: Mathematical and General}, 11(5):983--990,
  1978.


\bibitem{Eke}
I.~Ekeland.
\newblock {\em Convexity methods in {H}amiltonian mechanics}, volume~19 of {\em
  Ergebnisse der Mathematik und ihrer Grenzgebiete (3)}.
\newblock Springer-Verlag, Berlin, 1990.

\bibitem{Ge}
S.~Geman.
\newblock A limit theorem for the norm of random matrices.
\newblock {\em Ann. Probab.}, 8(2):252--261, 1980.

\bibitem{GT02}
F.~Guerra and F.~L. Toninelli.
\newblock The thermodynamic limit in mean field spin glass models.
\newblock {\em Comm. Math. Phys.}, 230(1):71--79, 2002.

\bibitem{GIK21}
S.~Gufler, J.~Igelbrink, and N.~Kistler.
\newblock {TAP} equations are repulsive.
\newblock Preprint, 2021.
\newblock arXiv:2111.02134.

\bibitem{Adrien22}
S.~Gufler, A.~Schertzer, and M.~A. Schmidt.
\newblock On concavity of {TAP} free energy in the {SK} model.
\newblock Preprint, 2022.
\newblock arXiv:2209.08985v3.

\bibitem{Hub59}
J.~Hubbard.
\newblock Calculation of partition functions.
\newblock {\em Phys. Rev. Lett.}, 3:77--78, 1959.

\bibitem{JaTo}
A.~Jagannath and I.~Tobasco.
\newblock Some properties of the phase diagram for mixed {$p$}-spin glasses.
\newblock {\em Probab. Theory Related Fields}, 167(3-4):615--672, 2017.

\bibitem{Kac69}
M.~Kac.
\newblock {\em Statistical Physics, Phase Transitions, and Superfluidity. Vol.
  I.}, chapter Mathematical mechanisms of phase transition, pages 241-- 305.
\newblock Chretien, M.~Gross, E.~P.~Deser (Eds.). New York, Gordon and Breach,
  Science Publishers, 1969.

\bibitem{Nico18}
G.~Kersting, N.~Kistler, A.~Schertzer, and M.~A. Schmidt.
\newblock From {P}arisi to {B}oltzmann: {G}ibbs potentials and high temperature
  expansions in mean field.
\newblock In {\em Statistical mechanics of classical and disordered systems},
  volume 293 of {\em Springer Proc. Math. Stat.}, pages 193--214. Springer,
  Cham, 2019.

\bibitem{LvHY}
R.~Latala, R.~van Handel, and P.~Youssef.
\newblock The dimension-free structure of nonhomogeneous random matrices.
\newblock {\em Invent. Math.}, 214(3):1031--1080, 2018.

\bibitem{L07}
M.~Ledoux.
\newblock {\em Deviation Inequalities on Largest Eigenvalues}, pages 167--219.
\newblock Springer Berlin Heidelberg, Berlin, Heidelberg, 2007.

\bibitem{O82}
J.~C. Owen.
\newblock Convergence of sub-extensive terms for long-range {I}sing spin
  glasses.
\newblock {\em Journal of Physics C: Solid State Physics}, 15(30):L1071, 1982.

\bibitem{PanBook}
D.~Panchenko.
\newblock {\em The {S}herrington-{K}irkpatrick model}.
\newblock Springer Monographs in Mathematics. Springer, New York, 2013.

\bibitem{P79}
G.~Parisi.
\newblock Infinite number of order parameters for spin-glasses.
\newblock {\em Phys. Rev. Lett.}, 43:1754--1756, 1979.

\bibitem{Plef82}
T.~Plefka.
\newblock Convergence condition of the {TAP} equation for the infinite-ranged
  {I}sing spin glass model.
\newblock {\em Journal of Physics A: Mathematical and General},
  15(6):1971--1978, 1982.

\bibitem{Plef02}
T.~Plefka.
\newblock Modified {TAP} equations for the {SK} spin glass.
\newblock {\em Europhysics Letters ({EPL})}, 58(6):892--898, 2002.

\bibitem{Plef20}
T.~Plefka.
\newblock The marginal stability of the metastable {TAP} states.
\newblock {\em J. Phys. A}, 53(37):375005, 11, 2020.

\bibitem{SK}
D.~Sherrington and S.~Kirkpatrick.
\newblock Solvable model of a spin-glass.
\newblock {\em Phys. Rev. Lett.}, 35:1792--1796, 1975.

\bibitem{Str57}
R.~L. Stratonovich.
\newblock A method for the computation of quantum distribution functions.
\newblock {\em Dokl. Akad. Nauk SSSR}, 115:77--78, 1957.

\bibitem{Ta12}
M.~Talagrand.
\newblock {\em Mean field models for spin glasses. {V}olume {I}\,\&\,{II}},
  volume 54\,\&\,55 of {\em Ergebnisse der Mathematik und ihrer Grenzgebiete.
  3. Folge. A Series of Modern Surveys in Mathematics}.
\newblock Springer, Heidelberg, 2011.

\bibitem{TAP}
D.~J. Thouless, P.~W. Anderson, and R.~G. Palmer.
\newblock Solution of 'solvable model of a spin glass'.
\newblock {\em Philosophical Magazine}, 35(3):593--601, 1977.

\bibitem{To}
F.~L. Toninelli.
\newblock About the {A}lmeida-{T}houless transition line in the
  {S}herrington-{K}irkpatrick mean-field spin glass model.
\newblock {\em Europhysics Letters ({EPL})}, 60(5):764--767, 2002.

\bibitem{V14}
R.~Vershynin.
\newblock Invertibility of symmetric random matrices.
\newblock {\em Random Structures Algorithms}, 44(2):135--182, 2014.

\end{thebibliography}
\end{document}